\documentclass[a4paper,11pt,subscriptaddress]{amsart}

\usepackage[defaultlines=2,all]{nowidow}

\allowdisplaybreaks

\usepackage{ mathtools, amssymb, mathdots, bbm, mleftright, faktor, stmaryrd, bm }

\usepackage{ letltxmacro }    
\usepackage{ xparse }

\usepackage{ tikz, tikz-cd }
\usetikzlibrary{shapes.misc, calc, patterns}

\usepackage[shortlabels]{ enumitem }

\usepackage{ standalone, graphicx }
\usepackage[section]{ placeins } 
\usepackage[margin=0.25cm]{ caption }
\usepackage{ marginnote }

\usepackage{ amsthm, thmtools, hyperref, url }
\usepackage[nameinlink]{ cleveref }

\numberwithin{equation}{section}
\declaretheorem[style=plain,numberlike=equation]{theorem}
\declaretheorem[style=plain,numberlike=theorem]{lemma}
\declaretheorem[style=plain,numberlike=theorem]{proposition}
\declaretheorem[style=plain,numberlike=theorem]{corollary}
\declaretheorem[style=remark,numberlike=theorem]{remark}
\declaretheorem[style=definition,numberlike=theorem]{definition}
\declaretheorem[style=definition,numberlike=theorem]{example}
\declaretheorem[name=Example,style=definition,numbered=no]{example*}

\newcounter{subequation}

\newenvironment{subequation}
  {\stepcounter{subequation}%
    \addtocounter{equation}{-1}%
    \equation}
  {\endequation}
\declaretheorem[style=definition,numberlike=subequation,name=Example]{subexample}
\declaretheorem[style=plain,numberlike=subequation,name=Claim]{subclaim}

\newcounter{thmlistcnt}

\newenvironment{thmlist}%
	{\setcounter{thmlistcnt}{0}%
	\begin{list}{\emph{(\roman{thmlistcnt})}}{%
		\usecounter{thmlistcnt}%
		\setlength{\topsep}{0pt}%
		\setlength{\leftmargin}{0pt}%
		\setlength{\itemsep}{0pt}%
		\setlength{\labelwidth}{17pt}
		\setlength{\itemindent}{30pt}}%
	}%
	{\end{list}}%
	
\newcounter{rmklistcnt}

\newenvironment{rmklist}%
	{\setcounter{rmklistcnt}{0}%
	\begin{list}{(\roman{rmklistcnt})}{%
		\usecounter{rmklistcnt}%
		\setlength{\topsep}{0pt}%
		\setlength{\leftmargin}{0pt}%
		\setlength{\itemsep}{0pt}%
		\setlength{\labelwidth}{17pt}
		\setlength{\itemindent}{30pt}}%
	}%
	{\end{list}}%
  
\NewDocumentCommand\set{s m}{%
    \IfBooleanTF#1%
    {\left\{ #2 \right\}}%
    {\{#2\}}%
}
\NewDocumentCommand\setbuild{s m m}{%
    \IfBooleanTF#1%
    {\ensuremath{\left\{ #2 \, : \, #3 \right\}}}
    {\ensuremath{\{ #2 \, : \, #3 \}}}%
}
\NewDocumentCommand\spangle{s m m m}{%
    \IfBooleanTF#1%
    {\ensuremath{\left\langle\, #2 \, \middle| \, #3 \,\right\rangle_{#4}}}%
    {\ensuremath{\langle\, #2 \, \mid \, #3 \,\rangle_{#4}}}%
}

\DeclarePairedDelimiter{\abs}{\lvert}{\rvert}

\newcommand{\Z}{\mathbb{Z}}

\newcommand{\C}{\mathbb{C}}

\newcommand{\N}{\mathbb{N}}

\newcommand{\iso}{\cong}
\newcommand{\tensor}{\otimes}

\newcommand{\crampedtensor}{{\tensor}}
\newcommand{\crampedwedge}{{\wedge}}

\DeclareMathOperator{\id}{id}
\DeclareMathOperator{\im}{im}
\DeclareMathOperator{\supp}{supp}
\newcommand{\blank}{{-}}
\DeclareMathOperator{\sign}{sgn}
\DeclareMathOperator{\sgn}{sgn} 
\newcommand{\dual}{\star}       
\newcommand{\condual}{\circ}    

\usepackage[vcentermath,enableskew]{youngtab} 

\newcommand{\Y}[1]{[#1]}

\newcommand{\CSYT}{\mathrm{CSYT}}
\newcommand{\SSYT}{\mathrm{SSYT}}
\newcommand{\col}{\mathrm{col}}

\newcommand{\GR}{\mathrm{GR}} 
\newcommand{\GRspace}[2]{\GR^{#1}(#2)}
\newcommand{\Grel}{R}
\newcommand{\act}[1]{|#1|}   

\DeclareMathOperator{\Sym}{Sym}
\newcommand{\wwedge}{\mathchoice{{\textstyle\bigwedge}}%
    {{\bigwedge}}%
    {{\textstyle\wedge}}%
    {{\scriptstyle\wedge}}}
\DeclareMathOperator{\Wedge}{\wwedge}

\DeclareMathOperator{\Stab}{Stab}

\DeclareMathOperator{\Hom}{Hom}

\renewcommand*{\det}{\qopname\relax o{det}}

\renewcommand{\t}{\mathrm{t}}

\DeclareMathOperator{\GL}{GL}
\DeclareMathOperator{\SL}{SL}

\newcommand{\twobytwosmallmatrix}[4]{\ensuremath{%
\begin{psmallmatrix}
    #1 & #2 \\ #3 & #4 \\
\end{psmallmatrix} %
}}

\newcommand{\congSLC}{\cong} 
\newcommand{\Borel}{B}              

\renewcommand{\epsilon}{\varepsilon}
\renewcommand{\phi}{\varphi}
\renewcommand{\leq}{\leqslant}

\renewcommand{\geq}{\geqslant}
\renewcommand{\emptyset}{\varnothing}

\newlength{\truelen}
\newcommand{\padbox}[3][c]{%
    \settowidth{\truelen}{\ensuremath{#2}}%
    \ifdim\truelen < #3%
        \makebox[#3][#1]{\ensuremath{#2}}%
    \else%
        \ensuremath{#2}%
    \fi%
}

\newlength{\minlen}
\newcommand{\flexbox}[3][c]{%
    \settowidth{\minlen}{\ensuremath{#3}}%
    \padbox[#1]{#2}{\minlen}%
}

\newcommand{\longsubsum}[3]{%
    \flexbox{%
        \smashoperator[r]{#1_{#2}}\;#3%
    }{%
        {\scriptstyle#2}%
    }%
}

\newcommand{\mbf}[1]{\mathbf{#1}}

\newcommand{\symt}{\mathbf{s}} 
\newcommand{\polyt}{\mathbf{e}}
\newcommand{\CPP}{\mathrm{CPP}}

\newcommand{\ls}{k} 

\DeclareMathOperator{\Col}{col} 
\newcommand{\jc}{j}
\newcommand{\kc}{k}
\newcommand{\lambdacircprime}{\lambda^{\circ\raisebox{1pt}{\(\scriptscriptstyle\prime\)}}}

\newcommand{\ppa}{\cdot}        

\newcommand{\surplus}{\mathrm{S}}
\newcommand{\surplusp}{\mathrm{s}}

\newcommand{\sym}{\mathrm{sym}}
\newcommand{\ix}{\mbf{i}}
\newcommand{\jx}{\mbf{j}}
\newcommand{\dx}{\mbf{d}}
\newcommand{\mx}{\ix_\mathrm{max}}
\newcommand{\repx}[1]{\mbf{c}^{(#1)}}
\newcommand{\unitx}[1]{\mbf{k}^{(#1)}}
\newcommand{\Mdecreasing}{\mathcal{M}_{\geq}}
\renewcommand{\sl}{\mathsf{sl}}
\newcommand{\svar}{r} 
\newcommand{\Tot}{\mathrm{T}}

\NewDocumentCommand\Ftensor{s m}{%
    \IfBooleanTF#1%
    {F_\tensor\bigl( #2 \bigr)}%
    {F_\tensor(#2)}%
}
\NewDocumentCommand\Fsym{s m}{%
    \IfBooleanTF#1%
    {F_\sym\bigl( #2 \bigr)}%
    {F_\sym(#2)}%
}
\NewDocumentCommand\Fwedge{s m}{%
    \IfBooleanTF#1%
    {F_\wedge\bigl( #2 \bigr)}%
    {F_\wedge(#2)}%
}

\newcommand{\maxt}{\mathrm{max}}
\newcommand{\tmax}{t_\mathrm{max}}
\newcommand{\eps}{\epsilon}

\newcommand{\betass}{{\scriptscriptstyle \beta}}
\newcommand{\betaxss}{
    {\raisebox{-2pt}{\scalebox{0.9}{\(\scriptscriptstyle \beta\)}}}
}

\newcommand{\defect}{\mathcal{D}}   
\newcommand{\cfreesummand}{\trianglelefteqslant}    

\newcommand{\sfrac}[2]{{\scriptstyle\frac{#1}{#2}}}
	
\mathcode`l="8000
\begingroup
\makeatletter
\lccode`\~=`\l
\DeclareMathSymbol{\lsb@l}{\mathalpha}{letters}{`l}
\lowercase{\gdef~{\ifnum\the\mathgroup=\m@ne \ell \else \lsb@l \fi}}%
\endgroup

\newcommand{\bs}{\hskip-0.5pt}

\newcommand{\KSLK}{K\bs\SL_2(K)}

\usepackage{ scalerel }
\newcommand{\xlen}{0.95}
\newcommand{\ylen}{0.95}
\newcommand{\hgap}{0.25}    
\newcommand{\vgap}{0.42}    
\newcommand{\ABshift}{0.07,1-0.21} 
\newcommand{\DNshift}{0.96,1-0.12} 
\newcommand{\drawthreesides}[2]{\draw (#1,#2+1)--++(0,-1)--++(1,0)--++(0,1);} 
\newcommand{\drawfoursideswstyle}[3]{\draw[#3] (#1,#2+1)--++(0,-1)--++(1,0)--++(0,1)--cycle;}  
\newcommand{\drawfoursides}[2]{\drawfoursideswstyle{#1}{#2}{}} 
\newcommand{\sstyle}{\scriptstyle}

\newcommand{\fb}{0} 
\newcommand{\ABstyle}[1]{\(\mathrlap{\sstyle \scaleobj{0.95}{\bm{#1}}}\)} 
\newcommand{\DNstyle}[1]{\(\mathllap{\smash{\sstyle \scaleobj{0.85}{\bm{#1}}}}\)} 
\definecolor{darkazure}{HTML}{0059b3}
\definecolor{azure}{HTML}{007fff}
\definecolor{paleazure}{HTML}{a6d2ff}
\definecolor{vdarkgold}{HTML}{807100}
\definecolor{darkgold}{HTML}{f2d400}
\definecolor{gold}{HTML}{ffdf00}
\definecolor{palegold}{HTML}{fff7bf}
\definecolor{myred}{HTML}{f03f02}
\definecolor{palered}{HTML}{f68c67}
\definecolor{mygrey}{HTML}{e6e6e6}
\definecolor{darkgrey}{HTML}{bbbbbb}

\usepackage{amsaddr}

\begin{document}

\title[Modular plethystic isomorphisms]{Modular plethystic isomorphisms for two-dimensional linear groups}

\author{Eoghan McDowell and Mark Wildon}
\email{eoghan.mcdowell@oist.jp}
\email{mark.wildon@rhul.ac.uk}

\date{\today \newline\hspace*{9pt} \emph{Affiliation}: Royal Holloway, University of London}

\makeatletter
\@namedef{subjclassname@2020}{\textup{2020} Mathematics Subject Classification}
\makeatother

\maketitle
\thispagestyle{empty}

\begin{abstract}
Let $E$ be the natural representation of the special linear group $\mathrm{SL}_2(K)$ over an arbitrary field $K$. We use the two dual constructions of the symmetric power when $K$ has prime characteristic to construct an explicit isomorphism $\mathrm{Sym}_m \mathrm{Sym}^\ell E \cong \mathrm{Sym}_\ell \mathrm{Sym}^m E$. This generalises Hermite reciprocity to arbitrary fields. We prove a similar explicit generalisation of the classical Wronskian isomorphism, namely $\mathrm{Sym}_m \mathrm{Sym}^\ell E \cong \bigwedge^m \mathrm{Sym}^{\ell+m-1} E$. We also generalise a result first proved by King, by showing that if $\nabla^\lambda$ is the Schur functor for the partition $\lambda$ and $\lambda^\circ$ is the complement of $\lambda$ in a rectangle with $\ell+1$ rows, then $\nabla^\lambda \mathrm{Sym}^\ell E \cong \nabla^{\lambda^\circ} \mathrm{Sym}_\ell E$. To illustrate that the existence of such `plethystic isomorphisms' is far from obvious, we end by proving that the generalisation $\nabla^\lambda \mathrm{Sym}^\ell E \cong \nabla^{\lambda'} \mathrm{Sym}^{\ell + \ell(\lambda') - \ell(\lambda)}E$ of the Wronskian isomorphism, known to hold for a large class of partitions over the complex field, does not generalise  to fields of prime characteristic, even after considering all possible dualities.
\end{abstract}

\section{Introduction}
Let \(E\) be the natural \(2\)-dimensional
representation of the special linear group \(\SL_2(\C)\) of \(2 \times 2\) complex matrices
with determinant \(1\). The classical Hermite reciprocity law, discovered
by Cayley, Hermite and Sylvester in the setting of invariant theory,
states that
\[
    \Sym^m \Sym^\ell\! E \congSLC \Sym^\ell \Sym^m\! E
\]
for all \(m\), \(\ell \in \N\) (see \cite[Exercise~6.18]{FultonHarrisReps}). 
A related classical 
result is the Wronskian isomorphism 
\[
    \Sym^m \Sym^\ell\! E \congSLC \Wedge^m \Sym^{\ell+m-1} \! E
\]
again for \(m\), \(\ell \in \N\) (see for instance \cite[\S 2.4]{AC}).
More recently King \cite[\S 4.2]{KingSU2Plethysms} used the character theory of \(\mathrm{SU}_2\) to prove that, if
\(\lambda\) is a partition,~\(\nabla^\lambda\) is the corresponding Schur functor (defined
in \S\ref{subsec:Schur} below), and \(\lambda^{\circ}\) is the
complement of \(\lambda\) in a rectangle with \(\ell+1\) rows, then
\[ \nabla^\lambda \Sym^\ell \! E \congSLC \nabla^{\lambda^\circ} \Sym^\ell \! E.\]
In this paper we construct explicit isomorphisms showing that
\emph{provided suitable dualities are introduced}
each of these results holds
when \(\C\) is replaced with an arbitrary
field. 
To illustrate that the existence of such `plethystic isomorphisms' is far
from obvious, we end
by proving that the generalization
\(\nabla^\lambda \Sym^\ell \! E \congSLC \nabla^{\lambda'} \Sym^{\ell + \ell(\lambda') - \ell(\lambda)}E\)
of the Wronskian isomorphism, shown to hold for a large class of partitions by King in \cite[\S 4.2]{KingSU2Plethysms}, \emph{does not} have a modular analogue, even after considering 
all possible dualities.

To state our main results it is essential to distinguish between the two dual constructions
of the symmetric power. Let \(K\) be a field. 
Given a \(K\)-vector space~\(V\) and \(r \in \N\), the symmetric
group \(S_r\) acts on \(V^{\otimes r}\) on the right by linear extension of the place permutation
action \((v_1 \otimes \cdots \otimes v_r) \cdot \sigma
= v_{1\sigma^{-1}} \otimes \cdots \otimes v_{r\sigma^{-1}}\). Let \(\Sym_r V = (V^{\otimes r})^{S_r}\) be the invariants for this action, and let \(\Sym^r\bs V\) be the coinvariants; that is,
\begin{equation}
\label{eq:symUpper}
\Sym^r\bs V 
    = \frac{
        V^{\otimes r} 
    }{
        \spangle{
            w \cdot \sigma - w
        }{
            w \in V^{\otimes r}, \sigma \in S_r
        }{K}
    }
\end{equation}
is the  symmetric power as usually defined.
Let \(\bigwedge^r V\) be the  exterior power, defined by quotienting \(V^{\tensor r}\) by the submodule generated by the fixed points of transpositions in \(S_r\) (here there is no need to consider a dual construction: see \Cref{prop:SchurContravariantDual} and the comments that follow).
Write \(\det V \iso \Wedge^{\dim V} V\) for the \(1\)-dimensional determinant representation corresponding to \(V\); write \(V^\dual\) for the dual of \(V\) (see \S\ref{subsec:duality}).

\subsubsection*{Complementary partition isomorphism}

Our first main result gives an isomorphism for representations of an arbitrary group.

\begin{restatable}[Complementary partition isomorphism]{theorem}
{comppartiso}
\label{thm:comp_partition_iso}
Let \(G\) be a group and let \(V\) be a \(d\)-dimensional representation of \(G\) over an arbitrary field.
Let \(s \in \N\), and let \(\lambda\) be a partition with \(\ell(\lambda) \leq d\) and
first part at most \(s\).
Let \(\lambda^\circ\) denote the complement of~\(\lambda\) in the \(d \times s\) rectangle.
Then there is an isomorphism
\[
    \nabla^\lambda V \iso \nabla^{\lambda^\circ} V^\dual \tensor (\det V)^{\otimes s}.
\]
\end{restatable}

Our map, described in \Cref{defn:Psi}, is explicit, and sends a canonical basis element labelled by a tableau to a canonical basis element labelled by a `complementary' tableau.

Two interesting special cases of this theorem are that
\(\bigwedge^s V \cong \bigwedge^{d-s} V^\dual\) and
\(\nabla^{(d,d-1,\ldots,1)} V \cong \nabla^{(d,d-1,\ldots, 1)}V^\dual\) 
whenever  \(\det V\) is trivial.
This assumption on the determinant is not very restrictive:
for instance it holds whenever  \(V\) is 
obtained by restricting a polynomial representation of \(\GL_2(K)\) to a subgroup
of \(\SL_2(K)\).
For example we obtain (\Cref{cor:exteriorSymmetric}) an explicit isomorphism \(\bigwedge^\ell \Sym^{\ell+m-1}\bs E \cong \bigwedge^m \Sym_{\ell+m-1}\bs E\), where \(E\) is the natural representation of \(\SL_2(K)\).
More generally, we obtain from \Cref{thm:comp_partition_iso} the following plethystic isomorphism.

\begin{restatable}{corollary}
{comppartisoSL}
\label{cor:complement_iso_for_SL_2}
Let \(\ell, s \in \N_0\), and let \(\lambda\) be a partition with \(\ell(\lambda) \leq l+1\) and first part at most \(s\).
Let \(\lambda^\circ\) denote the complement of~\(\lambda\) in the \(\hbox{$(l+1)$} \times s\) rectangle.
Let \(K\) be a field and let \(E\) be the natural \(2\)-dimensional representation of \(\SL_2(K)\).
Then there is an isomorphism
\[
    \nabla^\lambda \Sym^{l}\! E \iso \nabla^{\lambda^\circ}\! \Sym_{l}\! E.
\]
\end{restatable}

\subsubsection*{Wronskian isomorphism}
Our second main theorem is an explicitly defined Wronskian isomorphism that again holds in arbitrary characteristic.
Let \(\set{X, Y}\) be the canonical basis for the natural representation \(E\).

\begin{restatable}[Modular Wronskian isomorphism]{theorem}
{Wronskian}
\label{thm:Wronskian}
Let \(m\), \(\ell \in  \N\). 
Let \(K\) be a field and let \(E\) be the natural \(2\)-dimensional representation of \(\GL_2(K)\).
There is an isomorphism of \(\GL_2(K)\)-representations
\[
\Sym_m \Sym^\ell \! E \otimes (\det E)^{ \otimes m(m-1)/2}
    \cong \Wedge^m \Sym^{\ell + m - 1}\!E
\]
given by restriction of the \(K\)-linear map $(\Sym^\ell E)^{\otimes m} \rightarrow \bigwedge^m \Sym^{\ell+m-1} E$ 
defined on the canonical basis of $(\Sym^\ell E)^{\otimes m}$ by
\[
    \bigotimes_{j=1}^m X^{i_j} Y^{l-i_j} \,\mapsto\, \bigwedge_{j=1}^m X^{i_j+m-j} Y^{l-i_j+j-1}.
\]
\end{restatable}

We note that Section 3.4 of~\cite{AproduEtAl}, published after this work was begun, 
proves a related isomorphism $\Sym^m \Sym_\ell \! E \cong \Wedge^\ell \Sym^{\ell+m-1}\! E$ of $\SL_2(K)$ representations.
This is equivalent to the existence of the isomorphism in \Cref{thm:Wronskian}:
using \Cref{cor:exteriorSymmetric} (a more basic result, stated also in \cite{AproduEtAl}), the codomain $\Wedge^\ell \Sym^{\ell+m-1}\! E$
is isomorphic to $\Wedge^{m} (\Sym^{\ell+m-1} E)^\dual$ and hence by \Cref{lemma:exteriorPowerDuality} (another basic result) to $(\Wedge^m \Sym^{\ell+m-1} E)^\dual$, the dual of our right-hand side; meanwhile by \Cref{prop:SchurContravariantDual} and \Cref{lemma:SL2duality}, the
domain $\Sym^m \Sym_\ell \! E$ is isomorphic to \((\Sym_m \Sym^l E)^\dual\), the dual of our left-hand side.
The isomorphism in~\cite{AproduEtAl} is constructed indirectly
using maps into, and out of, the ring of symmetric functions; the proof that it is $\SL_2(K)$-invariant requires Pieri's rule
and a somewhat intricate inductive argument.
By contrast our isomorphism $\zeta$ has a simple one-line definition and a direct proof that it is $\GL_2(K)$-invariant.
We therefore believe that our approach is well worth recording.

\subsubsection*{Hermite reciprocity}
Composing our Wronskian isomorphism with a special case (\Cref{cor:exteriorSymmetric}) of the complementary partition isomorphism, we obtain the following modular version of Hermite reciprocity.
This result is obtained, without an explicit description of the maps, in a similar manner in \cite[Remark 3.2]{AproduEtAl}; we illustrate our explicit map in \Cref{eg:explicit_Hermite} and then make a connection with Foulkes' Conjecture.

\begin{restatable}[Modular Hermite reciprocity]{corollary}
{Hermitereciprocity}
\label{cor:Hermite_reciprocity}
Let \(m\), \(\ell \in \N\) and let \(E\) be the natural \(2\)-dimensional
representation of \(\GL_2(K)\). Then
\[ \Sym_m \Sym^\ell\! E \cong \Sym^\ell \Sym_m \! E. \]
\end{restatable} 

It is well known (see \Cref{prop:symDuals}) that
when \(K\) has characteristic~\(p\) and  \(m \leq p-1\), the functors \(\Sym_m\) and \(\Sym^m\) are naturally
isomorphic.
Thus \Cref{cor:Hermite_reciprocity} implies
that \(\Sym^m \Sym^\ell \! E \cong \Sym^\ell\Sym^m \!E\) when \(m \leq p-1\).
This special case of the corollary was first proved by
Kouwenhoven \cite[pp.~1699--1700]{KouwenhovenII}, where it is also  shown that
\(\Sym^p \Sym^l \! E \not\cong \Sym^l \Sym^p \! E\) if \(p < \ell < p(p-1)\).
In \Cref{prop:converseHermite} we give, for each prime \(p\), infinitely many examples of such non-isomorphisms, considering
all combinations of the upper and lower symmetric powers.
Thus our work shows that \Cref{cor:Hermite_reciprocity} is the unique modular generalization of Hermite reciprocity.

\subsubsection*{Obstructions to the conjugate partition isomorphism}
Another classical result, due to King \cite[\S 4.2]{KingSU2Plethysms}
(reproved as the main theorem in \cite{CaglieroPenazzi} and proved in a stronger version in \cite[Theorem 1.3]{PagetWildonSL2}), is that
the representations \(\nabla^{(a+1,1^b)}\Sym^{m+b}\bs E\) and \(\nabla^{(b+1,1^a)}\Sym^{m+a}\bs E\) of \(\SL_2(\C)\) are isomorphic for all \(m \in \N_0\).
By our final theorem,
proved using the new modular invariant introduced in \Cref{defn:defectSet},
this isomorphism has, in general, no modular analogue, even after considering all possible dualities. 
Let \(\Delta^\lambda\) denote the dual of the Schur functor~\(\nabla^\lambda\), as defined 
in~\S\ref{sec:background} below.

\begin{restatable}{theorem}
{hookObstructions}
\label{thm:hookObstructions}
Let \(\alpha\), \(\beta\), \(\epsilon \in \N\) with \(\alpha < \beta < \epsilon\).
If \(K\) has characteristic~\(p\) and \(\abs{K} > 1+ 2(p^\epsilon + p^\betass)(p^\alpha + p^\betass + 1) - p^\alpha(p^\alpha + 1)\), then
the eight representations of \(\SL_2(K)\) obtained from \(\Delta^{(p^{\alpha}+1, 1^{p^\betaxss})} \Sym^{p^\epsilon + p^\betaxss}\! E\)
by any combination of
\begin{thmlist}
\item
    replacing \(\Delta\) with \(\nabla\),
\item
    replacing \(\Sym^{\blank}\) with \(\Sym_{\blank}\),
\item
    swapping \(\alpha\) and \(\beta\), 
\end{thmlist}
are pairwise non-isomorphic.
\end{restatable}

\subsubsection*{Existing results}
We emphasise that while much is known about tensor products of the symmetric powers \(\Sym^\ell\! E\) of \(\SL_2(K)\)
and the related projective and tilting modules  when \(K\) has prime characteristic
(see \cite{ErdmannHenke}, \cite{KouwenhovenI},  \cite{McDowellTensorProducts}),
their modular behaviour under Schur functors is far less studied. 

As already noted, \cite{AproduEtAl} gives an alternative, less explicit, version of the Wronskian isomorphism and deduces Hermite reciprocity;
this is used to study Koszul cohomology in positive characteristic. 
The only other relevant results are in \cite{KouwenhovenII} on symmetric and exterior powers
of the irreducible representations of \(\GL_2(\mathbb{F}_p)\).
Kouwenhoven's strongest results, see for instance his Proposition~2.3,
are for the exterior powers \(\bigwedge^m \Sym^\ell E\) when \(m \leq p\); 
typically they are stated only up to projective summands.

Our five main results give new explicit
isomorphisms over fields of arbitrary characteristic, or rule out their existence.
The existence of such isomorphisms for the complex field
can be proved using the character theory of \(\GL_2(\C)\) and the plethysm
product on symmetric functions (see~\cite{PagetWildonSL2} for a comprehensive account), 
hence the term `modular plethystic isomorphism' in our title. We believe these explicit
isomorphisms merit further study,
even over fields of characteristic zero.

\subsubsection*{Outline}
In \S\ref{subsec:Schur} we recall the construction of Schur functors,
using the construction from~\cite{deBoeckPagetWildon}: we expect this will be background for most readers,
although the presentation by generators and relations may be less familiar.
In \S\ref{subsec:duality} we give a unified treatment of
some background results on duality which, while known to experts, have to be pieced
together from the literature.
In \S\ref{sec:complementary_partition_iso} we prove \Cref{thm:comp_partition_iso} and deduce
\Cref{cor:complement_iso_for_SL_2}.
In \S\ref{sec:Wronskian} we prove \Cref{thm:Wronskian}.
In the short \S\ref{sec:Hermite_reciprocity} we apply these results to prove \Cref{cor:Hermite_reciprocity}.
We end in \S\ref{sec:weights} by proving \Cref{thm:hookObstructions}.
We remark that \S\ref{sec:complementary_partition_iso}--\ref{sec:weights} are, for the most part, logically independent (each section relying only on the background in \S\ref{sec:background} and at most the isomorphisms proved in the previous sections).

\subsubsection*{Reduction to the special linear group}
We have stated \Cref{thm:Wronskian} and \Cref{cor:Hermite_reciprocity} for the general linear group \(\GL_2(K)\).
When \(K\) is an algebraically closed field, these results are isomorphisms between polynomial representations of equal degree, and so the results for~\(\GL_2(K)\) follow from those for \(\SL_2(K)\): when all elements in \(K\) are squares, \(\GL_2(K)\) is generated by \(\SL_2(K)\) and the scalar matrices, and scalar matrices have identical actions on polynomial representations of equal degree.
Once established over algebraically closed fields, restricting to subgroups yields the results for all fields.
Thus it will suffice to work over \(\SL_2(K)\).

\section{Background on Schur functors and duality}
\label{sec:background}

In this section we generalise the multilinear constructions seen in the introduction by defining Schur functors
and their duals.
Throughout let \(G\) be a group and let \(V\) be a left \(KG\)-module.

\subsection{Schur functors}\label{subsec:Schur}
We define a \emph{partition} to be a weakly decreasing sequence of natural numbers;
the entries are called its \emph{parts}.
The \emph{length} of a partition, already denoted \(\ell(\lambda)\) above, is the number of its parts.
The \emph{conjugate} of a partition \(\lambda\), already denoted \(\lambda'\) above, 
is defined by \(\lambda'_j = \max \setbuild{ i }{ \lambda_i \ge j }\)
for \(j \in \{1,\ldots, a\}\), where \(a\) is the largest part of~\(\lambda\).

\enlargethispage{12pt}
\subsubsection*{Young diagrams and tableaux}
Fix a partition \((\lambda_1, \ldots, \lambda_\ls)\).
The \emph{Young diagram} of 
\(\lambda\), denoted \([\lambda]\), is \(\setbuild{(i,j) }{ 1 \leq i \leq \ls,\, 1 \leq j \leq \lambda_i}\).
We refer to its elements as \emph{boxes}. 
Let \(\Col_j[\lambda] = \setbuild{(i,j) }{ 1 \leq i \leq \lambda_j'}\) be the set of boxes in column \(j\) of \([\lambda]\).
A \(\lambda\)-\emph{tableau} with entries from a subset \(\mathcal{B}\) of \(\N_0\) is a function
\(t \colon [\lambda] \rightarrow \mathcal{B}\).
We represent partitions and tableaux in the `English' convention: for example
the Young diagram of \((3,2)\) and three \((3,2)\)-tableaux are shown below.
\[
\yng(3,2)\,\;  \quad \young(213,32)\,\;  \quad \young(123,22)\,\; \quad \young(122,33)  
\]
A tableaux whose rows are weakly increasing when read left-to-right is called \emph{row semistandard};
a tableaux whose columns are strictly increasing when read top-to-bottom is called \emph{column standard}.
A \emph{semistandard tableau} is one which is both row semistandard and column standard. 
Thus the tableaux above are respectively column standard, row semistandard and semistandard. We denote the 
sets of column standard and semistandard \(\lambda\)-tableaux with entries from \(\mathcal{B}\) by \(\CSYT_\mathcal{B}(\lambda)\)
and \(\SSYT_\mathcal{B}(\lambda)\), respectively.

\subsubsection*{Partition-labelled symmetric and exterior powers}
Fix a basis \(\setbuild{ v_i }{ i \in \mathcal{B} }\) for~\(V\).
We have already defined \(\Sym^r V\) and \(\bigwedge^r V\) as quotients of \(V^{\otimes r}\). 
Let \(u_1 \cdot \ldots \cdot u_r\) and \(u_1 \wedge \cdots \wedge u_r\) 
denote the images of \(u_1 \otimes \cdots \otimes u_r\) in \(\Sym^r V\) and \(\bigwedge^r V\), respectively.
Observe that \(\Sym^r V\) and \(\bigwedge^r\) have bases
\begin{equation}
\begin{gathered}
\label{eq:symUpperBasis}
\setbuild{ v_{i_1} \cdot \ldots \cdot v_{i_r} }{  i_1 \leq \ldots \leq i_r },  \\
\setbuild{ v_{i_1} \wedge \ldots \wedge v_{i_r} }{ i_1 < \ldots < i_r  } 
\end{gathered}
\end{equation}
where \(i_j \in \mathcal{B}\) for each \(j\).

Let \(\Sym^\lambda\bs V = \Sym^{\lambda_1}\bs V \otimes \cdots \otimes \Sym^{\lambda_\ls}\bs V \) and let
\(\bigwedge^{\lambda' }\! V = \bigwedge^{\lambda'_1}\bs V \otimes \cdots \otimes \bigwedge^{\lambda_a'}\bs V\), where \(a = \lambda_1\) is the largest part of \(\lambda\).
Observe that
\(\Sym^\lambda\bs V\) has a basis indexed by row semistandard \(\lambda\)-tableaux and \(\bigwedge^{\lambda'}\! V\) has a basis
indexed by column standard \(\lambda\)-tableaux.
Let \(|t| \in \bigwedge^{\lambda'}\! V\) be the canonical basis
element corresponding to the column standard \(\lambda\)-tableau \(t\).
For instance, if~\(t\) is the semistandard tableau above
then \(|t| = (v_1 \wedge v_3) \otimes (v_2 \wedge v_3) \otimes v_2\).
We say that $|t|$ is a \emph{column tabloid}.

\subsubsection*{Place permutation action}
The symmetric group on \([\lambda]\), denoted \(S_{[\lambda]}\), acts on \(\lambda\)-tableaux by place permutation.
Given \(\sigma \in [\lambda]\) and a \(\lambda\)-tableau~\(t\), we define \(t \cdot \sigma\) by
\((t\cdot \sigma)(i,j) = t\bigl( (i,j)\sigma^{-1} \bigr)\).
Thus the entry of \(t\) in box \((i,j)\) is found in
\(t\sigma\) in box \((i,j)\sigma\).
Let \(\CPP(\lambda)\) be the subgroup of \(S_{[\lambda]}\) of permutations
that permute amongst themselves the boxes in each column of \([\lambda]\).

\subsubsection*{Construction of Schur functors}
Let \(t\) be a \(\lambda\)-tableau with entries from \(\mathcal{B}\).
Define a canonical basis element \(\symt(t) \in \Sym^\lambda V\) by
\(\symt(t) = \bigotimes_{i=1}^{\ell(\lambda)} \prod_{j=1}^{\lambda_i} v_{t(i,j)}\)
and define \(\polyt(t) \in \Sym^\lambda V\) by
\[
    \polyt(t) = \sum_{\sigma \in \CPP(\lambda)} \sgn(\sigma) \hskip1pt \symt(t \cdot \sigma).
\]
For example, if \(t\) is the semistandard \((3,2)\)-tableau above then 
\[
\polyt(t) \;=\;
    v_1v_2^2 \otimes v_3^2
    \,-\, v_1v_2v_3 \otimes v_2v_3
    \,-\, v_2^2 v_3 \otimes v_1v_3
    \,+\, v_2v_3^2 \otimes v_1v_2.
\]

\begin{definition}\label{defn:Schur}
Define
\(\nabla^\lambda V\) to be the subspace of \(\Sym^\lambda V\) spanned by all \(\polyt(t)\) for \(t\) a \(\lambda\)-tableau
with entries from \(\mathcal{B}\).
\end{definition}

It is clear that it suffices to consider column standard
\(\lambda\)-tableaux when defining \(\nabla^\lambda V\).
Indeed, if \(\tau \in \CPP(\lambda)\) and \(t\) is a \(\lambda\)-tableau then
\begin{equation}\label{eq:columnAction} 
    \polyt(t \cdot \tau) = \sgn(\tau) \polyt(t),
\end{equation}
and furthermore \(\polyt(t) = 0\) if \(t\) has a repeated entry in a column.
(If the characteristic is not \(2\), this second fact follows immediately from~\eqref{eq:columnAction};
for a characteristic-free proof, observe there exists a transposition \(\tau \in \CPP(\lambda)\) which fixes \(t\), and partition \(\CPP(\lambda)\) into pairs \(\set{\sigma, \tau\sigma}\) of permutations whose contributions to the sum cancel.)

By the following well-known result,
an even smaller set of \(\lambda\)-tableaux suffices to construct \(\nabla^\lambda V\). 

\begin{proposition}[{\cite[Proposition 2.11]{deBoeckPagetWildon}}]
\label{prop:ssytBasis}
The set
\[
\setbuild{ \polyt(t) }{ t \in \SSYT_{\mathcal{B}}(\lambda) }
\]
is a \(K\)-basis of \(\nabla^\lambda V\).
\end{proposition}

Since the right place permutation of \(\CPP(\lambda)\) on tableaux induces an action on \(\Sym^\lambda V\) that
commutes with the left action of \(G\),
each \(\nabla^\lambda V\) is a \(KG\)-submodule. 
For the same reason,
given a homomorphism \(V \rightarrow W\) of \(KG\)-modules 
there is a corresponding homomorphism \(\nabla^\lambda V \rightarrow \nabla^\lambda W\). We say that~\(\nabla^\lambda\)
is the \emph{Schur functor} for the partition~\(\lambda\).

\begin{example}
\label{eg:polytabloid_action_expansion}

To illustrate a practical method for computing the action of~\(G\), which we use in \S\ref{sec:weights} below,
suppose that \(V = \langle v_1, v_2, v_3 \rangle_K\) and  that
\(g \in G\) has action 
 given by \(gv_1 = v_1 + \alpha v_3 \), \(gv_2 = v_2\), \(gv_3 = \beta v_1 +  v_3\).
Then
\begin{align*}\footnotesize
g \polyt\Bigl(\, \young(122,33)\, \Bigr)
&= \footnotesize
\polyt\Bigl(\! \raisebox{-9.3pt}{
    \begin{tikzpicture}[x=1.35cm,y=-0.36cm,line width=0.4pt] 
    \draw(0,0)--(3,0); \draw(0,1)--(3,1); \draw(0,2)--(2,2); 
    \draw(0,0)--(0,2); \draw(1,0)--(1,2);\draw(2,0)--(2,2); \draw(3,0)--(3,1);
    \node at (0.5,0.5) {\(v_1 + \alpha v_3\)};
    \node at (1.5,0.5) {\(v_2\)};
    \node at (2.5,0.5) {\(v_2\)};
    \node at (0.5,1.5) {\(\beta v_1 + v_3\)};
    \node at (1.5,1.5) {\(\beta v_1 + v_3\)};
    \end{tikzpicture}
} \!\Bigr) \nonumber \\
&= \footnotesize
\polyt\Bigl(\, \young(122,33)\, \Bigr)
    - \beta\polyt\Bigl(\, \young(112,32)\, \Bigr)
    - \alpha \beta \polyt\Bigl(\, \young(122,33)\, \Bigr)
    + \alpha\beta^2 \polyt\Bigl(\, \young(112,32)\, \Bigr)
\end{align*}
where the first line should be interpreted purely formally as indicating a multilinear expansion.
\end{example}

When \(V=  E\) is the natural representation of \(\GL_n(K)\), the \(K\GL_n(K)\)-modules \(\nabla^\lambda E\)
are, as noted in~\cite[Remark 2.16]{deBoeckPagetWildon}, isomorphic to the modules constructed by James in \cite[Ch.~26]{James},
and hence also to those constructed by Green in \cite[Ch.~4]{GreenGLn}, and,
when the field is the complex numbers, by Fulton in \cite[\S 8]{FultonYTs}.
These modules are commonly called \emph{dual Weyl modules}.

It is immediate from \Cref{defn:Schur} that \(\nabla^{(r)} V \cong \Sym^r V\); a short calculation using~\eqref{eq:columnAction}
(indicated in \cite[\S 4.4]{GreenGLn}; see also \Cref{lemma:exteriorPowerDuality} below) shows that \(\nabla^{(1^r)} V \cong \bigwedge^r V\).

\subsubsection*{Garnir relations}
In \S\ref{sec:complementary_partition_iso} it it most useful to describe \(\nabla^\lambda V\) instead by generators and relations.
Recall that \(\Col_j[\lambda]\) is the set of boxes in column \(j\) of the Young diagram \([\lambda]\).

\begin{definition}
\label{defn:Garnir}
Let \(t\) be a column standard \(\lambda\)-tableau with entries from~\(\mathcal{B}\).
Let \(1 \leq j < k \leq \lambda_1\)
and let \(A \subseteq \Col_j[\lambda]\) and \(B \subseteq \Col_{k}[\lambda]\)
be such that \(\abs{A} + \abs{B} > \lambda'_j\).
Choose \(\mathcal{S}\) a set of coset representatives for the left cosets \(\sigma(S_A \times S_B)\)
of  \(S_{A} \times S_B\) in \(S_{A \sqcup B}\).
The \emph{Garnir relation} labelled by \((t, A, B)\)~is
\[
    \Grel_{(t, A, B)} = \sum_{\sigma \in \mathcal{S}} \act{t \cdot \sigma} \sgn(\sigma).
\]

\vspace*{-3pt}
\noindent Let \(\GRspace{\lambda}{V} \subseteq \Wedge^{\lambda'}\! V\) denote the subspace spanned by the Garnir relations.
\end{definition}

It is easily checked 
that the summand \(|t \cdot \sigma| \sgn(\sigma)\) depends only on the coset
containing \(\sigma\), and so \(\Grel_{(t, A, B)}\) is well-defined.

\begin{proposition}\label{prop:nablaPresentation}
The surjective \(KG\)-homomorphism \(\bigwedge^{\lambda'}\!V \rightarrow \nabla^\lambda V\) defined by \(|t| \mapsto \polyt(t)\) has kernel \(\GRspace{\lambda}{V}\).
\end{proposition}

\begin{proof} 
In \cite{deBoeckPagetWildon}, the proof of Lemma 2.4 and the remark which follows show that 
if \(\mathcal{S}\) is a set of coset representatives defining a Garnir relation then
\(\sum_{\tau \in \mathcal{S}} \polyt(u \ppa \tau) \sgn(\tau) = 0\) for any \(\lambda\)-tableau \(u\).
The statements there concern a distinguished set of Garnir relations (called \emph{snake relations}), but the argument given applies equally well to any Garnir relation.
Thus \(\GRspace{\lambda}{V}\) is contained in the kernel.

It then suffices to show that the kernel is spanned by Garnir relations.
Indeed, suppose \(\kappa = \sum_{t \in \CSYT(\lambda)} \alpha_t \act{t}\) is an element of the kernel.
Analogously to the proof of Corollary 2.6 in \cite{deBoeckPagetWildon}, we can use the Garnir relations to rewrite this as a sum over semistandard tableaux; that is, there exists a linear combination 
\(\epsilon \in \GRspace{\lambda}{V}\) of Garnir relations such that 
\begin{align*}
    \kappa + \epsilon  &= \sum_{s \in \SSYT(\lambda)} \beta_s \act{s}
\intertext{for some coefficients \(\beta_s \in K\).
Applying the surjection to this equation  gives}
    0 &= \sum_{s \in \SSYT(\lambda)} \beta_s \polyt(s).
\end{align*}
But the semistandard polytabloids are \(K\)-linearly independent by \Cref{prop:ssytBasis}, so \(\beta_s = 0\) for all \(s\).
Hence \(\kappa = - \epsilon \in \GRspace{\lambda}{V}\), as required.
\end{proof}

Another consequence of the arguments in the proof of \Cref{prop:nablaPresentation}, stated as Corollary 2.6 in \cite{deBoeckPagetWildon}, is that the 
Garnir relations may be used to express any \(\polyt(t)\) as a linear combination of \(\polyt(s)\) for semistandard tableaux~\(s\).
For example, we may rewrite the non-semistandard tableaux from \Cref{eg:polytabloid_action_expansion} 
using the relation for
\(A = \set{(2, 1)}\), \(B = \set{(1,2), (2,2)}\) as follows:
\begin{align*}
\polyt\Bigl(\, \young(112,32)\, \Bigr)
    &= \polyt\Bigl(\, \young(112,32)\, \cdot \bigl( (2,1)\ (1,2)\bigr) \Bigr)
        + \polyt\Bigl(\, \young(112,32)\, \cdot \bigl( (2,1)\ (2,2) \bigr) \Bigr) \\
    &= \polyt\Bigl(\, \young(132,12)\, \Bigr)
        + \polyt\Bigl(\, \young(112,23)\, \Bigr) \\
    &= \polyt\Bigl(\, \young(112,23)\, \Bigr) .
\end{align*}

We record that certain Garnir relations can be disregarded.

\begin{lemma}
\label{lemma:nonempty_intersection_Garnir_relation}
Let \(A, B\) be sets of boxes as in \Cref{defn:Garnir}.
Suppose that \(t\) is a tableau such that \(t(A) \cap t(B) \neq \varnothing\).
Then \(\Grel_{(t,A,B)} = 0\).
\end{lemma}

\begin{proof}
Suppose that \(t\) has the same entry in box \(a \in A\) and box \(b \in B\); let \(\tau = (a\ b) \in S_{A \sqcup B}\).
Then \(\tau\) acts on the left cosets of \(S_A \times S_B\) in \(S_{A \sqcup B}\) by left multiplication.
Choose a set \(\mathcal{S}\) of coset representatives such that the representatives of cosets in
each orbit
of size \(2\) are \(\sigma\) and \(\tau\sigma\) for some~\hbox{\(\sigma \in S_{\Y{\lambda}}\)}.

If \(\sigma \in S_{\Y{\lambda}}\) is any permutation, then \(t \cdot \tau\sigma = t \cdot \sigma\).
In particular if \(\set{\sigma, \tau\sigma} \subseteq \mathcal{S}\) are the representatives of cosets in an orbit of size \(2\),
then \(\act{t \cdot \tau\sigma} = \act{t \cdot \sigma}\) and \(\sgn(\tau\sigma) = - \sgn(\sigma)\),
and hence the contributions to the Garnir relation \(\Grel_{(t,A,B)}\) from these representatives cancel. 

If \(\sigma \in \mathcal{S}\) is the representative of a coset in an orbit of size \(1\), then \(\sigma^{-1}\tau\sigma \in S_A \times S_B \subseteq \CPP(\lambda)\), and so the boxes \(a \sigma\) and \(a \tau\sigma = b \sigma\) lie in the same column.
But \((t\cdot \sigma)(a \sigma) = t(a) = t(b) = (t \cdot \sigma)(b \sigma)\), so \(t \cdot \sigma\) has a repeated entry in a column.
Thus \(\polyt(t \cdot \sigma) = 0\) and the contribution to the Garnir relation \(\Grel_{(t,A,B)}\) from this orbit is zero.
\end{proof}

\subsection{Duality}\label{subsec:duality}
The \emph{dual module} to \(V\), denoted \(V^\dual\), is the \(K\)-vector space \(V^\dual\) with \(G\)-action defined by
\((g\theta)(v) = \theta(g^{-1}v)\) for \(\theta \in V^\dual\), \(v \in V\),  and \(g \in G\). A standard calculation
shows that if \(\rho_V(g)\) is the matrix representing the action of \(g\) on \(V\) with respect to the chosen
basis \(\setbuild{v_i }{ i \in \mathcal{B} }\) then the matrix representing the action of \(g\) on \(V^\dual\) with respect to the dual
basis \(\setbuild{ v_i^\dual }{ i \in \mathcal{B} }\) is \(\rho_V(g^{-1})^\t\), where \(\t\) denotes matrix transpose. 
Thus \(\rho_{V^\dual}(g) = \rho_V(g^{-1})^\t\).

In this paper we also need a further notion of duality, defined for instance in \cite[(2.8a)]{GreenGLn}.

\begin{definition}\label{defn:contravariantDual}
Let \(G\) be a subgroup of \(\GL_n(K)\) which is closed under matrix transposition and let \(V\) be a \(KG\)-module. 
The \emph{contravariant dual} of~\(V\), denoted \(V^\circ\),
is the \(K\)-vector space~\(V^\dual\) with \(G\)-action defined by \((g\theta)(v) = \theta(g^\t v)\).
\end{definition} 

Another standard calculation shows that \(\rho_{V^\condual}(g) = \rho_V(g^\t)^\t\). 
Therefore contravariant duality preserves polynomial modules: if the matrix entries in~\(\rho_V(g)\) are
polynomials in the entries of \(G\) then so are the matrix entries in~\(\rho_V(g^\t)^\t\).

\begin{definition}\label{defn:dualSchurFunctor}
Let \(\lambda\) be a partition. 
We define
\[
    \Delta^\lambda V = (\nabla^\lambda(V^\dual))^\dual.
\]
\end{definition}

In \Cref{remark:Delta} we give a more explicit construction of these modules.

\begin{proposition}\label{prop:SchurContravariantDual}
Let \(\lambda\) be a partition and let \(G\) be a matrix group closed under transposition.
Then \(\Delta^\lambda V \iso (\nabla^\lambda (V^\condual))^\condual\).
In particular, if \(V\) is polynomial then so is \(\Delta^\lambda V\).
\end{proposition}

\begin{proof}
By our definition, \(\Delta^\lambda V\) is represented by the homomorphism \(\rho_{\Delta^\lambda V}\) where
\[ \rho_{\Delta^\lambda V} (g)
    = \rho_{\nabla^\lambda (V^\dual)}(g^{-1})^{\t}.\]
Let \(n = \dim V\) and let \(E\) be the natural representation of \(\GL_n(K)\).
Then \smash{\(\rho_{\nabla^\lambda V} = \rho_{\nabla^\lambda E} \rho_V\)}
(this
follows from the action of \(g\) on \(\nabla^\lambda V\) being given precisely by acting by \(\rho_V(g)\) within each box), 
and so
\[ \rho_{\Delta^\lambda V} (g)
    = \bigl( \rho_{\nabla^\lambda E}\rho_{V^\dual}(g^{-1}) \bigr)^{\t}
    = \bigl( \rho_{\nabla^\lambda E}(\rho_{V}(g)^{\t}) \bigr)^{\t}
    = \bigl( \rho_{\nabla^\lambda E}\rho_{V^\circ}(g^\t) \bigr)^\t
\]
which equals \(\rho_{(\nabla^\lambda V^\circ)^\circ}(g)\) by the same token.
\end{proof}

Thus \(\Delta^\lambda\) generalises
to arbitrary group representations the construction in \cite[Ch.~5]{GreenGLn}: writing \(E\) for the natural \(n\)-dimensional \(K\!\GL_n(K)\)-module, it is immediate from the definition of contravariant duality that \(E^\circ \cong E\), and hence that \(\Delta^\lambda E \iso (\nabla^\lambda E)^\condual\).
The examples in \cite[\S 5.2]{GreenGLn} show that \(\Delta^{(r)}E \cong \Sym_r E\) and \(\Delta^{(1^r)}E \cong \bigwedge^r E\), and hence that \(\Delta^{(r)} = \Sym_r \) and \(\Delta^{(1^r)} = \bigwedge^r\) (using as 
in the proof of \Cref{prop:SchurContravariantDual} that for any \(KG\)-module \(V\), the action of \(g \in G\) on \(\nabla^\lambda V\) is determined by the action of \(\rho_V(g)\) on \(\nabla^\lambda E\)).

Rearranging the duality in these examples, we see that
\((\Sym^r V)^\dual \cong \Sym_r V^\dual\) and \((\bigwedge^r V)^\dual \cong \bigwedge^r V^\dual\).
In \Cref{lemma:exteriorPowerDuality} we make the second isomorphism explicit. 
By \Cref{prop:SchurContravariantDual}, the same isomorphisms hold with duality replaced with contravariant duality.

In our proofs we typically work with the special linear group \(\SL_2(K)\), for which it is important that the two notions of duality coincide.

\begin{lemma}
\label{lemma:SL2duality}
Suppose \(G = \SL_2(K)\).
Then \(V^\dual \iso V^\condual\).
\end{lemma}

\begin{proof}
Let \(J =     \begin{psmallmatrix}
            0 & 1 \\[1pt] -1 & 0 \\
    \end{psmallmatrix} \in \SL_2(k)
\).
It is simple to verify that for any matrix \(g \in \SL_2(k)\), we have \(Jg^{-1}J^{-1} = g^\t\).
Then \(Y = \rho_V(J^{-1})^\t\) satisfies
\begin{align*}
Y \rho_{V^\dual}(g) Y^{-1}
    &= \rho_V(J^{-1})^\t \rho_V(g^{-1})^\t \rho_V(J)^\t  \\
    &= {\mleft(\rho_V(J) \rho_V(g^{-1}) \rho_V(J^{-1}) \mright)}^\t \\
    &= \bigl( \rho_V(Jg^{-1}J^{-1}) \bigr)^\t\\
    &= \rho_{V}(g^\t)^\t
\end{align*}
and since \(\rho_{V^\circ}(g) = \rho_V(g^\t)^\t\), the proposition follows.
\end{proof}

\section{Complementary partition isomorphism (proof~of~\Cref{thm:comp_partition_iso})}
\label{sec:complementary_partition_iso}

This section proves the following theorem and its corollaries.

\comppartiso*

Our proof has four steps.
In the first step we construct an explicit isomorphism \(\bigwedge^r V \cong \bigwedge^{d-r} V^\dual\, \otimes\, \det V\);
this proves the theorem when \(\lambda = (1^r)\), \(s=1\).
In the second step we define (writing \(\lambdacircprime\) for \((\lambda^\circ)'\)) the induced isomorphism
\[ \Psi \colon \textstyle \bigwedge^{\lambda'} V \cong \bigwedge^{\lambdacircprime}
V^\dual \,\otimes\, \det V^{\otimes s}.\]
In the third step we prove a technical result on the permutations column standardising a tableau, in order to identify the image of column non-standard tabloids under this map. 
Finally in the fourth step we use this result and further arguments to
show that  the image under~\(\Psi\) of the \(KG\)-submodule \(\GRspace{\lambda}{V}\)  of Garnir relations 
is contained in \(\GRspace{\lambda^\circ}{V^\dual} \otimes (\det V)^{\otimes s}\).
This easily implies \Cref{thm:comp_partition_iso}.

\subsection{First step: exterior powers}
\label{subsec:first_step_exterior_powers}

Fix a  basis \(\set{v_1, \ldots, v_d}\) of \(V\), and let \(\set{v_1^\dual, \ldots, v_d^\dual}\) be the dual basis of \(V^\dual\).
Let \(\setbuild{(v_{i_1} \wedge \cdots \wedge v_{i_r})^\dual}{1 \leq i_1 < \ldots < i_r \leq d}\) be the basis of \((\Wedge^r V)^\dual\) dual to the basis \(\setbuild{v_{i_1} \wedge \cdots \wedge v_{i_r}}{1 \leq i_1 < \ldots < i_r \leq d}\) of \(\Wedge^r V\).

\begin{lemma}\label{lemma:exteriorPowerDuality}
There is an isomorphism \((\bigwedge^r V)^\dual \cong \bigwedge^r V^\dual\) defined by
\[ (v_{i_1} \wedge \cdots \wedge v_{i_r})^\dual \mapsto v_{i_1}^\dual \wedge \cdots \wedge v_{i_r}^\dual. \]
\end{lemma}

\begin{proof}
Let \(\rho_{V}\) be the homomorphism representing the action on \(V\) with respect to the given basis, and likewise for the other relevant modules.
Let \(g \in G\), and for convenience write \(R = \rho_V(g)\); thus \(R^{-1} = \rho_V(g^{-1})\).
The entry of \(\rho_{\Wedge^r V}(g)\) in row \((j_1,\ldots, j_r)\) and column \((i_1,\ldots, i_r)\) is the coefficient of \(v_{j_1} \wedge \cdots \wedge v_{j_r}\) in \(gv_{i_1} \wedge \cdots \wedge gv_{i_r}\), namely
\[
\rho_{\Wedge^r V}(g)_{(j_1, \ldots, j_r), (i_1, \ldots, i_r)}
    = \sum_{\sigma \in S_r}
        \sign(\sigma)
        R_{j_{1 \sigma},i_1} \cdots R_{j_{r \sigma},i_r}.
\]
Therefore the action of \(g\) on \((\Wedge^r V)^\dual\) is given by
\begin{align*}
\rho_{(\Wedge^r V)^\dual}(g)_{(j_1, \ldots, j_r), (i_1, \ldots, i_r)}
    &= \rho_{\Wedge^r V}(g^{-1})_{(i_1, \ldots, i_r), (j_1, \ldots, j_r)} \\
    &= \sum_{\sigma \in S_r}
        \sign(\sigma)
        R^{-1}_{i_{1 \sigma},j_1} \cdots R^{-1}_{i_{r \sigma},j_r} \\
\intertext{while on \(\Wedge^r V^\dual\) it is given by}
\rho_{\Wedge^r V^\dual}(g)_{(j_1, \ldots, j_r), (i_1, \ldots, i_r)}
    &= \sum_{\sigma \in S_r}
        \sign(\sigma)
        \rho_{V^\dual}(g)_{j_{1 \sigma},i_1} \cdots \rho_{V^\dual}(g)_{j_{r \sigma},i_r} \\
    &= \sum_{\sigma \in S_r}
        \sign(\sigma)
        R^{-1}_{i_1, j_{1 \sigma}} \cdots R^{-1}_{i_r, j_{r \sigma}}.
\end{align*}
Using
\(R^{-1}_{i_1,j_{1\sigma}} \ldots R^{-1}_{i_r,j_r\sigma}
=  R^{-1}_{i_{1\sigma^{-1}},j_1} \ldots R^{-1}_{i_{r\sigma^{-1}},j_r}\)
and reindexing the sum shows that the two matrices \(\rho_{(\Wedge^r V)^\dual}(g)\) and \(\rho_{\Wedge^r V^\dual}(g)\) are equal, as required.
\end{proof}

\newcommand{\orderpreservingperms}{\Pi}

We now use \Cref{lemma:exteriorPowerDuality} to construct an explicit isomorphism \(\psi \colon \bigwedge^r V \cong \bigwedge^{d-r}V^\dual \otimes
\det V\).

Let \(\orderpreservingperms \subseteq S_d\) be the set of permutations of \(\{1,\ldots, d\}\)
which preserve the relative orders within each subset \(\set{1,\ldots, r}\) and \(\set{r+1,\ldots, d}\);
that is, \(\sigma \in \orderpreservingperms\) if and only if \(1\sigma < \ldots < r\sigma\) and \((r+1)\sigma < \ldots < d\sigma\).
Then we can write the standard basis of \(\Wedge^r V\) as \(\setbuild{v_{1\sigma} \wedge \cdots \wedge v_{r \sigma}}{\sigma \in \orderpreservingperms}\).

Let \(\psi \colon \bigwedge^r V \rightarrow \bigwedge^{d-r} V^\dual\) be the \(K\)-linear bijection defined by
\begin{equation}\label{eq:psi_definition}
\psi(v_{1\sigma} \wedge \cdots \wedge v_{r\sigma})
    = \sgn(\sigma) v_{(r+1)\sigma}^\dual \wedge \cdots \wedge v_{d\sigma}^\dual
\end{equation}
for any \(\sigma \in \orderpreservingperms\), and hence any \(\sigma \in S_d\).

\begin{proposition}\label{prop:complement}
Regarded as a map \(\bigwedge^r V \rightarrow \bigwedge^{d-r} V^\dual \otimes \det V\), the \(K\)-linear
isomorphism \(\psi\) is a \(KG\)-isomorphism.
\end{proposition}

\begin{proof}
Let \(\epsilon = (v_1 \wedge \cdots \wedge v_d)^\dual\) be the unique element of the canonical basis of  \((\bigwedge^d V)^\dual\).
Our strategy is to show that \(\psi\) is the image of \(\epsilon\) under a sequence of \(G\)-equivariant maps.
Assuming this is done,
since \((\bigwedge^d V)^\dual \cong (\det V)^{-1}\), for each \(g \in G\) and \(x \in \bigwedge^r V\) we have \((g \cdot \psi)(x) = (\det g^{-1}) \psi(x)\), as required.

In the following steps we 
apply the comultiplication map \((\bigwedge^d V)^\dual \rightarrow
(\bigwedge^r V \otimes \bigwedge^{d-r} V)^\dual\) to \(\epsilon\), using the canonical bases just introduced;
compose with the standard isomorphism \((U \otimes W)^\dual \cong U^\dual
\otimes W^\dual\);
and then apply the isomorphism from \Cref{lemma:exteriorPowerDuality} on the right-hand factor:
\begin{align*}
\epsilon
    &\mapsto
    \sum_{\sigma \in \orderpreservingperms}\sgn(\sigma) (v_{1\sigma} \wedge \cdots \wedge v_{r\sigma} \otimes v_{(r+1)\sigma} \wedge \cdots \wedge v_{d\sigma} )^\dual \\
    &\mapsto
    \sum_{\sigma \in \orderpreservingperms}\sgn(\sigma) (v_{1\sigma} \wedge \cdots \wedge v_{r\sigma})^\dual \otimes (v_{(r+1)\sigma} \wedge \cdots \wedge v_{d\sigma} )^\dual \\
    &\mapsto
    \sum_{\sigma \in \orderpreservingperms} \sgn(\sigma)(v_{1\sigma} \wedge \cdots \wedge v_{r\sigma} )^\dual
\otimes v_{(r+1)\sigma}^\dual \wedge \cdots \wedge v_{d\sigma}^\dual. \end{align*}
Finally we apply the standard isomorphism \(U^\dual \otimes W \cong \Hom_K(U, W)\) to obtain the \(K\)-linear isomorphism
\[    v_{1\sigma} \wedge \cdots \wedge v_{r \sigma} \mapsto  \sgn(\sigma) v_{(r+1)\sigma}^\dual \wedge \cdots \wedge v_{d\sigma}^\dual
\]
which is precisely the map \(\psi\).
\end{proof}

As an immediate application we obtain a corollary  for two-dimensional linear groups mentioned in the introduction.

\begin{corollary}\label{cor:exteriorSymmetric}
Let \(\ell\), \(m \in \N\) and
let \(E\) be the natural \(2\)-dimensional representation of \(\GL_2(K)\).
Then
\[
\Wedge^\ell \Sym^{\ell+m-1}\! E
    \iso \Wedge^m \Sym_{\ell+m-1}\! E \otimes (\det E)^{\tensor \sfrac{1}{2} (\ell-m)(\ell+m-1)}.
\]
\end{corollary}

\begin{proof}
The representations in the statement are polynomial of equal degree \(\ell(\ell+m-1)\), so,
using the argument from the end of the introduction, it suffices to show the 
isomorphism after restriction to \(\SL_2(K)\).
In this setting, \(\det V\) is trivial for any polynomial representation \(V\).
Thus taking \(G = \SL_2(K)\), \(r = \ell\), \(d = \ell+m\) and \(V = \Sym^{\ell+m-1} E\) 
in \Cref{prop:complement} gives that \(\psi\) is an \(\SL_2(K)\)-isomorphism 
\(\Wedge^\ell \Sym^{\ell+m-1} \! E \cong \Wedge^m (\Sym^{\ell + m-1}\! E)^\dual\).
But \((\Sym^{\ell+m-1} E)^\dual \cong \Sym_{\ell+m-1} E^\dual\) as noted in the discussion following \Cref{prop:SchurContravariantDual}, and \(E^\dual \iso E\) over \(\SL_2(K)\) by \Cref{lemma:SL2duality}.
It follows that
there is an \(\SL_2(K)\)-isomorphism \(\Wedge^\ell \Sym^{\ell+m-1} E \cong \Wedge^m \Sym_{\ell+m-1} E\). 
\end{proof}

\subsection{Second step: definition of \texorpdfstring{\(\Psi\)}{Psi}}
For each \(j \in \set{1,\ldots, s}\), let \(j^\circ = s+1 -j\) and observe that  column \(j^\circ\) of \(\lambda^\circ\) has length \(d-\lambda'_j\), where we set \(\lambda'_j = 0\) if
\(j\) exceeds the greatest part of \(\lambda\).
Given a column standard \(\lambda\)-tableau \(t\) with entries from \(\{1,\ldots, d\}\), let
\(t^\circ\) be the column standard \(\lambda^\circ\)-tableau whose entries in column~\(j^\circ\) are the 
complement in \(\{1,\ldots, d\}\) of the entries of \(t\) in column \(j\). 
Clearly \(t \mapsto t^\circ\) is a bijection.

Recall from~\S\ref{subsec:Schur}
that the column tabloid \(|t| \in \bigwedge^{\lambda'}\! V\) is the canonical basis element corresponding to the column standard \(\lambda\)-tableau \(t\) with entries in \(\mathcal{B}\).
For such a tableau, define the \emph{surplus} of \(t\) to be \(\surplus(t) = \sum_{(i,j) \in \Y{\lambda}}( t(i,j) - i\)). 

\begin{definition}
\label{defn:Psi}
Let \(\Psi \colon \bigwedge^{\lambda'}\bs V \rightarrow \bigwedge^{\lambdacircprime}\bs V^\dual\) be the \(K\)-linear isomorphism
defined by 
\[ \Psi(\hskip0.5pt |t|\hskip0.5pt ) = (-1)^{\surplus(t)} |t^\circ| 
\]
for \(t\) a column standard \(\lambda\)-tableau with entries from \(\set{1,\ldots, d}\).
\end{definition}

For example, if \(d=3\), \(s=4\) and \(\lambda = (3,1)\) with Young diagram \({\tiny \yng(3,1)}\), then \(\lambda^\circ = (4,3,1)\) with Young diagram \({\tiny \yng(4,3,1)}\).
If \(t = \tiny \young(112,2)\), then \(\surplus(t) = 0 + 0 + 1 + 0 = 1 \) and
\[
    \Psi\Bigl(\, \abs*{\, \small \young(112,2) \, }\, \Bigr)
    =
    - \abs*{\, \small \young(1123,233,3) \,}\, .
\]

When we apply the map \(\psi\), defined by \eqref{eq:psi_definition} in \S\ref{subsec:first_step_exterior_powers},
to each column of a tableau, the product of the signs is given by the surplus of the tableau.
This follows from the lemma below.
Recall that for \(r \in \set{1, \ldots, d}\), the set \(\Pi \subseteq S_d\) is the subset of permutations preserving the relative orders of \(\set{1, \ldots, r}\) and \(\set{r+1, \ldots, d}\).
For \(\sigma \in \Pi\), write \(\surplusp(\sigma) = -\frac12 r(r+1) + \sum_{i=1}^r i\sigma\).

\begin{lemma}\label{lemma:surplus}
Let \(\sigma \in \Pi \subseteq S_d\).
Then \(\sign(\sigma) = (-1)^{\surplusp(\sigma)}\).
\end{lemma}

\begin{proof}
We induct on the number of inversions of $\sigma$, i.e.~pairs $(j,k)$ with $1 \leq j < k \leq d$ and $j \sigma > k\sigma$.
If $\sigma$ is the identity permutation then $\sigma$ has no inversions and \( \surplusp(\sigma) = 0\), establishing the base case.
If $\sigma$ is not the identity permutation then, since $\sigma \in \Pi$, there exists $j \in \set{1,\ldots, r}$ and $k \in \set{r+1,\ldots, d}$
such that $j\sigma = k\sigma + 1$. Let $m = k\sigma$.
Observe that $(j,k)$ is an inversion of $\sigma$ that is not an inversion of \(\sigma(m\ m+1)\), and moreover \(\sigma(m\ m+1)\) has no inversions that \(\sigma\) does not.
Since also $\sigma (m\ m+1) \in \Pi$, the inductive hypothesis applies:
$\sgn \bigl( \sigma (m\, m+1) \bigr) = (-1)^{\surplusp(\sigma(m\ m+1))}$.
But the set \(\set{1\sigma(m\ m+1), \ldots, r\sigma(m\ m+1)}\) differs from \(\set{1\sigma, \ldots, r\sigma}\) only by the loss of \(m+1\) and the addition of \(m\), so \(\surplusp\bigl(\sigma(m\ m+1)\bigr) = \surplusp(\sigma)-1\).
Hence $\sgn(\sigma) = (-1)^{\surplusp(\sigma)}$, as required.
\end{proof}

Now applying \Cref{prop:complement} to each column in the \(d \times s\) rectangle in turn and using \Cref{lemma:surplus}, we see that, regarded as a map \(\Wedge^{\lambda'} V \rightarrow \bigwedge^{\lambdacircprime} V^\dual \otimes (\det V)^s\), the map \(\Psi\) is a \(KG\)-isomorphism.

\subsection{Third step: column sorting permutations}
\label{subsec:third_step_col_sort_perms}
We need to know how permuting the boxes of a tableau affects the image of its column tabloid under \(\Psi\).
The column sets of the resulting tabloid are clear, and permuting boxes does not change the value of its surplus, but each column must be sorted into ascending order before the map \(t \mapsto t^\circ\) can be applied, and more work is required to identify the sign which arises.

Fix \(t \in \CSYT_{\set{1, \ldots, d}}(\lambda)\) and two columns \(1 \leq j < k \leq \lambda_1\).
Let \(j^\circ = s+1-j\) and \(k^\circ = s+1-k\) be the columns in \(\lambda^\circ\) complementary to the columns \(j\) and \(k\) in \(\lambda\). 
Given a permutation \(\tau \in S_{\Y{\lambda}}\), the \emph{support} of \(\tau\), denoted \(\supp\tau\), is the set of points which are not fixed by \(\tau\).

Let \(\tau \in S_{\col_j[\lambda] \,\sqcup\, \col_k[\lambda]}\) be a product of disjoint transpositions of the form \((a\ b)\) where \(a \in \col_j[\lambda]\), \(b \in \col_k[\lambda]\), such that the boxes in the support of \(\tau\) have distinct entries in \(t\).
Suppose also that \(\act{t \ppa \tau} \neq 0\); this precisely says that, in \(t\), no box in column \(j\) in the support of \(\tau\) has an entry which appears in column \(k\), and vice versa.
Observe that for each box in the support of~\(\tau\), there is exactly one box in \(\col_{j^\circ}[\lambda^\circ] \,\sqcup\, \col_{k^\circ}[\lambda^\circ]\) containing in \(t^\circ\) the same entry: considering, for example, a box \(a \in \col_j[\lambda]\) in the support of \(\tau\), the entry \(t(a)\) does not appear in column \(k\) of \(t\) by the above assumptions, and so appears precisely once in column \(k^\circ\) of \(t^\circ\) (and does not appear in column~\(j^\circ\) of \(t^\circ\) because it appears in column \(j\) of \(t\)).
For \(a \in \col_j[\lambda] \,\sqcup\, \col_k[\lambda]\) in the support of \(\tau\), denote this corresponding box \((t^\circ)^{-1}t(a)\).
Then define \(\tau^\circ \in S_{\col_{j^\circ}[\lambda^\circ] \sqcup \col_{k^\circ}[\lambda^\circ]}\) by replacing in every transposition the box \(a\) with the box \((t^\circ)^{-1}t(a)\).

\begin{example}
Consider a tableau \(t \in \CSYT_{\set{1, \ldots, 9}}(\lambda)\)
with columns \(j\) and \(k\) as shown in the margin of the following page; the columns \(j^\circ\) and \(k^\circ\) of \(t^\circ\) are depicted inverted beneath their complementary column in \(t\).

Consider the permutation \(\tau = \bigl( (4,\jc)\ (3,\kc) \bigr) \in S_{\col_j[\lambda] \,\sqcup\, \col_k[\lambda]}\), swapping the boxes containing \(8\) in \(j\) and \(5\) in \(k\) (these boxes are shaded in the diagram).
These entries are distinct, and furthermore \(8\) does not appear in \(k\) and \(5\) does not appear in \(j\), so \(\act{t \ppa \tau} \neq 0\).
Then our construction yields \(\tau^\circ = \bigl( (3,\jc^\circ)\ (4,\kc^\circ) \bigr) \in S_{\col_{j^\circ}[\lambda^\circ] \,\sqcup\, \col_{k^\circ}[\lambda^\circ]}\), swapping the boxes containing \(5\) and \(8\) but now in \(j^\circ\) and \(k^\circ\) (again shaded).
Note that \(t \cdot \tau\) and \(t^\circ \cdot \tau^\circ\) are both sorted to
column standard tableaux by applying two transpositions, and we find that \(\Psi(\act{t \cdot \tau}) = (-1)^{\surplus(t)} \act{t^\circ \cdot \tau^\circ}\).

\marginnote{
\begin{center}
\begin{tikzpicture}[x=\xlen cm, y=-\ylen cm]
\fill[mygrey] (0+\fb,3+\fb) rectangle (1-\fb,4-\fb);
\fill[mygrey] (1+\hgap+\fb,2+\fb) rectangle (1+\hgap+1-\fb,3-\fb);
    \node at (0.5,-0.5) {\(j\)};                        \node at (1+\hgap+0.5,-0.5) {\(k\)};
\drawthreesides{0}{0}                               \drawthreesides{1+\hgap}{0}
    \node at (0.5,0.5) {\(1\)};                         \node at (1+\hgap+0.5,0.5) {\(2\)};
\drawthreesides{0}{1};                              \drawthreesides{1+\hgap}{1}
    \node at (0.5,1.5) {\(2\)};                         \node at (1+\hgap+0.5,1.5) {\(3\)};
\drawthreesides{0}{2};                              \drawthreesides{1+\hgap}{2}
    \node at (0.5,2.5) {\(6\)};                         \node at (1+\hgap+0.5,2.5) {\(5\)};
\drawthreesides{0}{3};                              \drawfoursides{1+\hgap}{3}
    \node at (0.5,3.5) {\(8\)};                         \node at (1+\hgap+0.5,3.5) {\(6\)};
\drawfoursides{0}{4};
    \node at (0.5,4.5) {\(9\)};
\begin{scope}[shift={(0,\vgap)}] 
\fill[mygrey] (0+\fb,6+\fb) rectangle (1-\fb,7-\fb);
\fill[mygrey] (1+\hgap+\fb,5+\fb) rectangle (1+\hgap+1-\fb,6-\fb);
                                                    \drawthreesides{1+\hgap}{4}
                                                        \node at (1+\hgap+0.5,4.5) {\(9\)};
\drawthreesides{0}{5};                              \drawthreesides{1+\hgap}{5}
    \node at (0.5,5.5) {\(7\)};                         \node at (1+\hgap+0.5,5.5) {\(8\)};
\drawthreesides{0}{6};                              \drawthreesides{1+\hgap}{6};
    \node at (0.5,6.5) {\(5\)};                         \node at (1+\hgap+0.5,6.5) {\(7\)};
\drawthreesides{0}{7};                              \drawthreesides{1+\hgap}{7};
    \node at (0.5,7.5) {\(4\)};                         \node at (1+\hgap+0.5,7.5) {\(4\)};
\drawfoursides{0}{8};                               \drawfoursides{1+\hgap}{8};
    \node at (0.5,8.5) {\(3\)};                         \node at (1+\hgap+0.5,8.5) {\(1\)};
    \node at (0.5,9.5) {\(j^\circ\)};                   \node at (1+\hgap+0.5,9.5) {\(k^\circ\)};
\end{scope}
\end{tikzpicture}
\end{center}
}[-4.8cm]%

Consider instead the permutation \(\tau = \bigl( (3,\jc)\ (3,\kc) \bigr) \in S_{\col_j[\lambda] \,\sqcup\, \col_k[\lambda]}\), swapping the box containing \(6\) in \(j\) with the box containing \(5\) in \(k\).
This does not satisfy the hypotheses above: the entry \(6\) appears in both column \(j\) and column \(k\) of \(t\), and so \(\act{t \cdot \tau} = 0\); since \(6\) does not appear in either \(j^\circ\) or \(k^\circ\), we cannot define \(\tau^\circ\).
\end{example}

It is clear from the construction that \(\Psi(\act{t \ppa \tau}) = \pm \act{t^\circ \ppa \tau^\circ}\): the permutation~\(\tau^\circ\) swaps a pair of boxes between columns \(j^\circ\) and \(k^\circ\) if and only if the boxes containing their entries are swapped between columns \(j\) and \(k\) by \(\tau\).
We claim that the correct sign is \((-1)^{\surplus(t)}\).

\begin{lemma}\label{lemma:col_sort_perm}
Let \(t \in \CSYT_{\{1,\ldots, d\}}(\lambda)\).
Let \(x \in \Col_j(t)\) and \(y\in \set{1, \ldots, d} \setminus \Col_j(t)\).
Let \(u\) be the tableau obtained from \(t\) by replacing in column \(j\) the entry \(x\) with the entry \(y\),
and let \(u'\) be the tableau obtained from \(t^\circ\) by replacing in column \(j^\circ\) the entry \(y\) with the entry \(x\).
The unique place permutation in \(S_{\Y{\lambda}}\) which sorts both column \(j\) of 
\(u\) and column \(j^\circ\) of \(u'\) 
has sign \((-1)^{|x-y|-1}\).
\end{lemma}

\begin{proof}
Let \(C = \set{ \min\set{x,y} + 1, \ldots, \max\set{x,y}-1 }\).
Column \(j\) of \(u\) is sorted by a cycle of length \(1 + \abs{C \cap \Col_j(t)}\), while column \(j^\circ\) of \(u'\)
is sorted by a cycle of length \(1 + \abs{C \hskip1pt\cap\hskip1pt \Col_{j^\circ}(t^\circ)}\).
Let \(\sigma\) be the product of these disjoint cycles; this is the unique permutation in \(S_{\Y{\lambda}}\) which sorts both \(u\) 
and \(u'\). 
Then~\(\sigma\) has sign \((-1)^e\) where
\[
    e = \abs{C \cap \Col_j(t)} + \abs{C \cap \Col_{j^\circ}(t^\circ)}.
\]
But by the definition of \(t^\circ\)
we have \(\Col_j(t)\, \sqcup\, \Col_{j^\circ}(t^\circ) = \set{1,\ldots, d}\).
Thus \(e = \abs{C} = \abs{x-y}-1\), as required.
\end{proof}

Observe that in \Cref{lemma:col_sort_perm} the sign of the column sorting permutation depends only on the set \(\set{x, y}\), and not on \(t\) (except through
the requirement that \(x \in \Col_j(t)\) and \(y\not\in \Col_j(t)\), which holds by hypothesis).
Generalising, we obtain the following lemma.

\begin{lemma}\label{prop:generalised_col_sort_perm}
Let \(t \in \CSYT_{\set{1,\ldots, d}}(\lambda)\).
Let \(\set{x_1, \ldots, x_r} \subseteq \col_j(t)\) and \(\set{y_1, \ldots, y_r} \subseteq \set{1, \ldots, d} \setminus \col_j(t)\).
Let \(u\) be the tableau obtained from \(t\) by replacing in column \(j\) each entry \(x_i\) with the entry \(y_i\),
and let \(u'\) be the tableau obtained from \(t^\circ\) by replacing in column \(j^\circ\) each entry \(y_i\) with the entry \(x_i\).
The unique place permutation in \(S_{\Y{\lambda}}\) which sorts both column \(j\) of~\(u\) and column \(j^\circ\) of~\(u'\) has sign depending 
only on the pairs \(\set{x_i, y_i}\), and not on \(t\).
\end{lemma}

\begin{proof}
This follows by repeated application of \Cref{lemma:col_sort_perm}.
\end{proof}

\begin{proposition}\label{prop:image_of_permuted_tabloid}
Let \(t \in \CSYT_{\{1,\ldots, d\}}(\lambda)\).
Let \(\tau \in S_{\col_j[\lambda] \,\sqcup\, \col_k[\lambda]}\) be a product of disjoint transpositions of the form \((a\ b)\) where \(a \in \col_j[\lambda]\), \(b \in \col_k[\lambda]\), such that the boxes in the support of \(\tau\) have distinct entries in \(t\).
Suppose \(\act{t \ppa \tau} \neq 0\).
Then \(\Psi(\act{t \ppa \tau}) = (-1)^{\surplus(t)} \act{t^\circ \ppa \tau^\circ}\).
\end{proposition}

\begin{proof}
As has already been recorded, \(\Psi(\act{t \ppa \tau}) = \pm \act{t^\circ \ppa \tau^\circ}\), or equivalently \(\set{1,\ldots, d} \setminus \col_j(t \cdot \tau) = \col_{j^\circ}(t^\circ \ppa \tau^\circ)\) and \(\set{1,\ldots, d} \setminus \col_k(t \cdot \tau) = \col_k(t^\circ \cdot \tau^\circ)\).

Let \smash{\(\pi \in S_{\Col_j[\lambda]}\)}, \smash{\(\phi \in S_{\Col_k[\lambda]}\)}, \smash{\(\pi' \in S_{\Col_{j^\circ}[\lambda^\circ]}\)},
\smash{\(\phi' \in S_{\Col_{k^\circ}}[\lambda^\circ]\)} 
be the unique place permutations which sort, respectively,
columns \(j\) and \(k\) of 
\(t \cdot \tau\) and columns \(j^\circ\) and \(k^\circ\) of \(t^\circ \cdot \tau^\circ\).
By \Cref{prop:generalised_col_sort_perm}, the signs \(\sign(\pi\pi')\) and \(\sign(\phi\phi')\) depend only on the pairs \(\set{t(a), t(b)}\) where \((a\ b)\) are the disjoint transpositions comprising \(\tau\), and therefore these signs are equal.

The tableaux \(t \cdot \tau \pi\phi\) and  \(t^\circ \ppa \tau^\circ \pi'\phi'\) are column standard, their column sets are complementary as noted above, and their surpluses are both equal to \(\surplus(t)\).
Thus we have \(\Psi(\act{t \cdot \tau\pi\phi}) = (-1)^{\surplus(t)} \act{t^\circ \ppa \tau^\circ \pi' \phi'}\), and hence
\begin{align*}
\Psi (|t \cdot \tau| )
    &= \sgn(\pi \phi) \Psi (\act{t \ppa \tau \pi \phi}) \\
    &= \sgn(\pi \phi) (-1)^{\surplus(t)} \act{ t^\circ \cdot \tau^\circ \pi' \phi' } \\ 
    &= \sgn(\pi \phi)\sgn(\pi' \phi') (-1)^{\surplus(t)} \act{ t^\circ \cdot \tau^\circ } \\
    &= (-1)^{\surplus(t)} \act{t^\circ \cdot \tau^\circ}. \qedhere
\end{align*}
\end{proof}

\subsection{Fourth step: image of the Garnir relations}
Recall that \(\GRspace{\lambda}{V}\) and \(\GRspace{\lambda^\circ}{V^\dual}\) 
are the submodules of \(\bigwedge^{\lambda'}V\) and \(\bigwedge^{\lambdacircprime}\!V^\dual\) of Garnir relations,
as defined in~\S\ref{subsec:Schur}.
In this section we complete the strategy outlined at the start of this section by proving the following proposition.
The proof is unavoidably somewhat long: after the setup it is split into three claims.

\begin{proposition}
\label{prop:Psi_respects_Garnir_relations}
The map \(\Psi \colon \bigwedge^{\lambda'} V \rightarrow \bigwedge^{\lambdacircprime} V^\dual\) 
respects Garnir relations, in the sense that \(\Psi(\GRspace{\lambda}{V}) \subseteq \GRspace{\lambda^\circ}{V^\dual}\).
\end{proposition}

\begin{proof}
Let \(\Grel_{(t,A,B)}\) be a  Garnir relation as defined in \Cref{defn:Garnir}.
Thus \(t \in \CSYT_{\set{1,\ldots, d}}(\lambda)\), and \(A \subseteq \Col_j[\lambda]\) and \(B \subseteq \Col_{k}[\lambda]\) where \(1 \leq j < k \leq \lambda_1\) and \(|A| + |B| > \lambda'_j\).
Our aim is to show that \(\Psi ( \Grel_{(t,A,B)} )\in \GRspace{\lambda^\circ}{V^\dual}\).
Note that place permutations do not change the value of \(\surplus\), so all signs arising from application of \(\Psi\) in this lemma will be \((-1)^{\surplus(t)}\).

Recall that, by construction of \(t^\circ\), the entries in columns
\(j^\circ = s+1-j\) and \(k^\circ = s+1-k\) of \(t^\circ\) are complementary to the entries in columns \(j\) and~\(k\) of \(t\).
By \Cref{lemma:nonempty_intersection_Garnir_relation}, we may assume that \(t(A) \cap t(B) = \varnothing\).
Then since also~\(t\) is column standard, 
the entries of the boxes of \(t\) in \(A \sqcup B\) are distinct.

Let
\begin{align*}
	A^\circ &= \setbuild{ b \in \Col_{\kc^\circ}[\lambda^\circ] }{ t^\circ(b) \in t(A) } \\
    B^\circ &= \setbuild{ a \in \Col_{\jc^\circ}[\lambda^\circ] }{ t^\circ(a) \in t(B) } \\
    D_\jc   &= \setbuild{ a \in \Col_\jc[\lambda] }{ t(a) \in \Col_\kc(t) } \\
    D_\kc   &= \setbuild{ b \in \Col_\kc[\lambda] }{ t(b) \in \Col_\jc(t) }.
\end{align*}
The sets \(A^\circ\) and \(B^\circ\) are, respectively,
the boxes in columns~\(\jc^\circ\) and \(\kc^\circ\) of~\(\lambda^\circ\) 
whose entries in \(t^\circ\) lie in the boxes \(A\) and \(B\) in \(t\).
The sets \(D_\jc\) and \(D_\kc\) are, respectively, the boxes in columns \(\jc\) and \(\kc\) of \(\lambda\) whose entries appear in both columns \(\jc\) and \(\kc\) of \(t\). 
Note  that \(t^\circ(A^\circ) \subseteq t(A)\) and \(t^\circ(B^\circ) \subseteq t(B)\), but equality need not hold because entries which appear in both columns of \(t\) do  not appear in either column of \(t^\circ\).
Thus \(t^\circ(A^\circ)\) omits the entries in \(D_\kc\) and \(t^\circ(B^\circ)\) omits
the entries in \(D_\jc\) and 
\begin{subequation}
t^\circ(A^\circ) = t(A \setminus D_\jc), \quad t^\circ(B^\circ) = t(B \setminus D_\kc).
\label{eq:ABcirc}
\end{subequation}
Since \(t\) and \(t^\circ\) are injective on the sets of boxes appearing above, \(|A^\circ| = |A| - |D_\jc|\), \(|B^\circ| = |B|-|D_\kc|\).

\marginnote{
\begin{center}
\begin{tikzpicture}[x=\xlen cm, y=-\ylen cm]
\fill[paleazure] (0+\fb,2+\fb) rectangle (1-\fb,3-\fb);
\fill[paleazure] (0+\fb,3+\fb) rectangle (1-\fb,4-\fb);
\fill[paleazure] (0+\fb,4+\fb) rectangle (1-\fb,5-\fb);
\fill[palegold] (1+\hgap+\fb,0+\fb) rectangle (1+\hgap+1-\fb,1-\fb);
\fill[palegold] (1+\hgap+\fb,1+\fb) rectangle (1+\hgap+1-\fb,2-\fb);
\fill[palegold] (1+\hgap+\fb,2+\fb) rectangle (1+\hgap+1-\fb,3-\fb);
    \node at (0.5,-0.5) {\(j\)};                        \node at (1+\hgap+0.5,-0.5) {\(k\)};
\drawthreesides{0}{0}                               \drawthreesides{1+\hgap}{0}
                                                        \node at ($(1+\hgap,0)+(\DNshift)$) {\DNstyle{D_k}};
    \node at (0.5,0.5) {\(1\)};                         \node at (1+\hgap+0.5,0.5) {\(2\)};
                                                        \node[vdarkgold] at ($(1+\hgap,0)+(\ABshift)$) {\ABstyle{B}};
\drawthreesides{0}{1};                              \drawthreesides{1+\hgap}{1}
    \node at ($(0,1)+(\DNshift)$) {\DNstyle{D_j}};
    \node at (0.5,1.5) {\(2\)};                         \node at (1+\hgap+0.5,1.5) {\(3\)};
                                                        \node[vdarkgold] at ($(1+\hgap,1)+(\ABshift)$) {\ABstyle{B}};
\drawthreesides{0}{2};                              \drawthreesides{1+\hgap}{2}
    \node at ($(0,2)+(\DNshift)$) {\DNstyle{D_j}};
    \node at (0.5,2.5) {\(6\)};                         \node at (1+\hgap+0.5,2.5) {\(5\)};
    \node[darkazure] at ($(0,2)+(\ABshift)$) {\ABstyle{A}};\node[vdarkgold] at ($(1+\hgap,2)+(\ABshift)$) {\ABstyle{B}};
\drawthreesides{0}{3};                              \drawfoursides{1+\hgap}{3}
                                                        \node at ($(1+\hgap,3)+(\DNshift)$) {\DNstyle{D_k}};
    \node at (0.5,3.5) {\(8\)};                         \node at (1+\hgap+0.5,3.5) {\(6\)};
    \node[darkazure] at ($(0,3)+(\ABshift)$) {\ABstyle{A}};
\drawfoursides{0}{4};
    \node at (0.5,4.5) {\(9\)};
    \node[darkazure] at ($(0,4)+(\ABshift)$) {\ABstyle{A}};
\begin{scope}[shift={(0,\vgap)}] 
\fill[palegold] (0+\fb,6+\fb) rectangle (1-\fb,7-\fb);
\fill[palegold] (0+\fb,8+\fb) rectangle (1-\fb,9-\fb);
\fill[paleazure] (1+\hgap+\fb,4+\fb) rectangle (1+\hgap+1-\fb,5-\fb);
\fill[paleazure] (1+\hgap+\fb,5+\fb) rectangle (1+\hgap+1-\fb,6-\fb);
                                                    \drawthreesides{1+\hgap}{4}
                                                        \node at (1+\hgap+0.5,4.5) {\(9\)};
                                                        \node[darkazure] at ($(1+\hgap,4)+(\ABshift)$) {\ABstyle{A^\circ}};
\drawthreesides{0}{5};                              \drawthreesides{1+\hgap}{5}
    \node at (0.5,5.5) {\(7\)};                         \node at (1+\hgap+0.5,5.5) {\(8\)};
                                                        \node[darkazure] at ($(1+\hgap,5)+(\ABshift)$) {\ABstyle{A^\circ}};
\drawthreesides{0}{6};                              \drawthreesides{1+\hgap}{6};
    \node at (0.5,6.5) {\(5\)};                         \node at (1+\hgap+0.5,6.5) {\(7\)};
    \node[vdarkgold] at ($(0,6)+(\ABshift)$) {\ABstyle{B^\circ}};
\drawthreesides{0}{7};                              \drawthreesides{1+\hgap}{7};
    \node at (0.5,7.5) {\(4\)};                         \node at (1+\hgap+0.5,7.5) {\(4\)};
\drawfoursides{0}{8};                               \drawfoursides{1+\hgap}{8};
    \node at (0.5,8.5) {\(3\)};                         \node at (1+\hgap+0.5,8.5) {\(1\)};
    \node[vdarkgold] at ($(0,8)+(\ABshift)$) {\ABstyle{B^\circ}};
    \node at (0.5,9.5) {\(j^\circ\)};                   \node at (1+\hgap+0.5,9.5) {\(k^\circ\)};
\end{scope}
\end{tikzpicture}
\end{center}
}[-4.5cm]%

An illustrative example in which \(\lambda'_\jc = 5\), \(\lambda'_{\kc} = 4\), \(d = 9\) and \(t(A) = \{6,8,9\}\), \(t(B) = \{2,3,5\}\) 
is shown in the margin, with the sets introduced above indicated. See also Figure~1, which
shows all the sets introduced in the course of the proof.

For each left coset of \(S_A \times S_B\) in \(S_{A \sqcup B}\), choose a coset representative which is a product of 
disjoint transpositions \((a\ b)\) with \(a \in A\), \(b \in B\).
Let~\(\mathcal{T}\) be the subset of those representatives \(\tau\) such that \(|t \cdot \tau| \neq 0\);
equivalently,~\(\mathcal{T}\) is the subset of coset representatives that fix all boxes in \(D_\jc\) and \(D_\kc\).
(Only entries in \(t(D_\jc) = t(D_\kc)\) can be repeated in a column of \(t \ppa \tau\), and since \(t(A) \cap t(B) = \emptyset\), such an entry appears as a repeat in \(t \cdot \tau\) if and only if it has changed column.)
Thus the specified Garnir relation may be written as
\begin{subequation}
\label{eq:specified_Garnir_rel}
    \Grel_{(t, A, B)} = \sum_{\tau \in \mathcal{T}} |t \cdot \tau| \sign{\tau}.
\end{subequation}

The chosen coset representatives \(\mathcal{T}\) precisely meet the properties assumed in \S\ref{subsec:third_step_col_sort_perms}.
Thus we can define for each \(\tau \in \mathcal{T}\) a permutation 
\(\tau^\circ \in S_{\col_{j^\circ}[\lambda^\circ] \,\sqcup\, \col_k[\lambda]}\)
by, in every transposition comprising \(\tau\), replacing the box \(a\) with the unique box \((t^\circ)^{-1}t(a)\) in column \(j^\circ\) or \(k^\circ\) containing the entry \(t(a)\).
Moreover, the conditions of \Cref{prop:image_of_permuted_tabloid} are met, and so, for all \(\tau \in \mathcal{T}\),
\begin{subequation}
\label{eq:image_of_one_term}
\Psi(\act{t \ppa \tau}) = (-1)^{\surplus(t)} \act{t^\circ \ppa \tau^\circ}.
\end{subequation}

Let \(\mathcal{T}^\circ = \setbuild{\tau^\circ}{\tau \in \mathcal{T}}\).
It follows from \eqref{eq:specified_Garnir_rel} and \eqref{eq:image_of_one_term} that
\begin{subequation}
\label{eq:S_circ_corollary}
\Psi (\Grel_{(t,A,B)})
    = (-1)^{\surplus(t)}\sum_{\tau^\circ \in \mathcal{T}^\circ} \act{t^\circ \ppa \tau^\circ} \sgn\tau^\circ.
\end{subequation}
We claim next that \(\mathcal{T}^\circ\) has one of the properties required of \(\mathcal{S}\) in the definition of a Garnir relation.

\begin{subclaim}\label{subclaim:S_circ}
Excluding precisely those cosets whose place permutation actions send \(\act{t^\circ}\) to~\(0\), the set
\(\mathcal{T}^\circ\) is a complete irredundant set of left coset representatives of \(S_{A^\circ} \times S_{B^\circ}\) in \(S_{A^\circ \sqcup B^\circ}\).
\end{subclaim}

\begin{proof}
If \(\tau\), \(\theta \in \mathcal{T}\) are such that \(\tau^\circ\) and \(\theta^\circ\) represent the same coset of
\(S_{A^\circ} \times S_{B^\circ}\), then \(|t^\circ \cdot \tau^\circ| = \pm |t^\circ \cdot \theta^\circ|\). 
Using \Cref{prop:image_of_permuted_tabloid} and that \(\Psi\) is a bijection, it follows that
\(|t \cdot \tau| = \pm|t \cdot \theta|\). Since the boxes \(A \sqcup B\) 
have distinct entries in \(t\), it follows that \(\tau\) and \(\theta\) represent the same coset of \(S_A \times S_B\).
Additionally, if \(\act{t^\circ \cdot \tau^\circ} = 0\) then \(\act{t \cdot \tau} = 0\), which contradicts \(\tau \in \mathcal{T}\).
Thus distinct elements of \(\mathcal{T}^\circ\) are  representatives of distinct cosets whose place permutation
actions do not send \(\act{ t^\circ }\) to \(0\).

On the other hand, given any permutation in \(S_{A^\circ \sqcup B^\circ}\), we may choose a coset representative \(\sigma\) that can 
be written as a product of disjoint transpositions \((a\ b)\) with \(a \in A^\circ\) and \(b \in B^\circ\).
Because \(t^\circ(A^\circ)\) and \(t^\circ(B^\circ)\) are disjoint, the support of \(\sigma\) necessarily has distinct entries in \(t^\circ\).
Supposing also that the place permutation action of this coset does not send \(|t^\circ|\) to \(0\), then this representative satisfies the conditions of \S\ref{subsec:third_step_col_sort_perms}, and we may perform the construction symmetric to \(\tau \mapsto \tau^\circ\).
We thus obtain a permutation \(\tau \in S_{A \sqcup B}\)
such that \(\tau \in \mathcal{T}\) and \(\tau^\circ = \sigma\).
We conclude that \(\mathcal{T}^\circ\) is complete with the specified exclusions.
\end{proof}

It would be very convenient 
to conclude from \eqref{eq:S_circ_corollary} and \Cref{subclaim:S_circ} that \(\Psi (\Grel_{(t,A,B)}) = (-1)^{\surplus(t)} \Grel_{(t^\circ, A^\circ, B^\circ)}\), finishing the proof.
However, it may not be the case that \(\abs{A^\circ} + \abs{B^\circ} > \lambda^{\circ \prime}_{\kc^\circ}\), and
this is a requirement for \((t^\circ, A^\circ, B^\circ)\) to label a Garnir relation. 
We address this problem by expanding the subset \(A^\circ\) of \(\Col_{\kc^\circ}[\lambda^\circ]\) in a way that does not affect the resulting relation: 
adding boxes which have entries lying also in column \(\jc^\circ\) of \(t^\circ\).

Let
\begin{align*}
    N_\jc &= \setbuild{ a \in \Col_{\jc^\circ}[\lambda^\circ] }{ t^\circ(a) \in t^\circ(\Col_{\kc^\circ}[\lambda^\circ]) } \\
    N_\kc &= \setbuild{ b \in \Col_{\kc^\circ}[\lambda^\circ] }{ t^\circ(b) \in t^\circ(\Col_{\jc^\circ}[\lambda^\circ]) }
\end{align*}
be the sets of boxes in columns \(\jc^\circ\) and \(\kc^\circ\) of \(\lambda^\circ\) respectively whose entries appear 
in both columns \(\jc^\circ\) and \(\kc^\circ\) of \(t^\circ\). (Thus \(N_\jc\) and \(N_\kc\) are the analogues for \(t^\circ\)
of \(D_\jc\) and \(D_\kc\).) 
In particular, \(N_\kc\) is disjoint from \(A^\circ\) and \(N_\jc\) is disjoint from \(B^\circ\). 
These sets, and
the sets introduced in the proof of the following claim, are shown in \Cref{fig:column_subsets}.

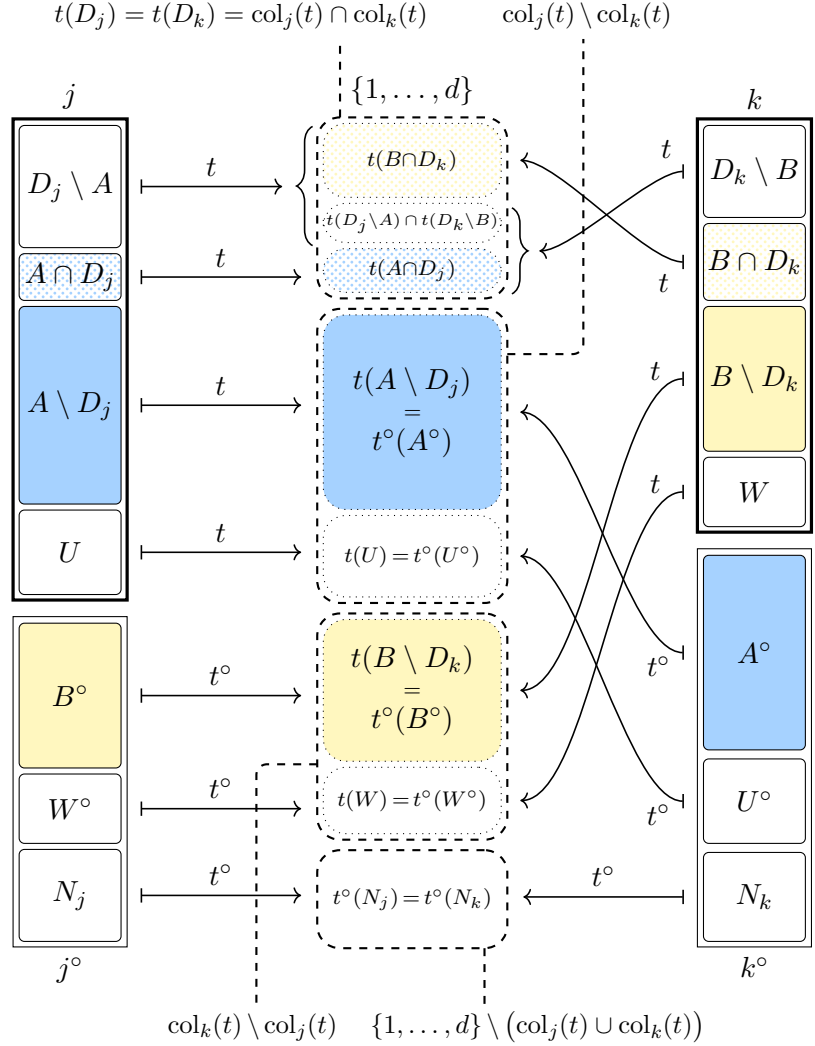
\begin{figure}[p]
\centering

\begin{tikzpicture}[x=1cm,y=-1cm]
\usetikzlibrary{calc, arrows.meta, decorations.pathreplacing, patterns}
\newcommand{\combinexy}[2]{#1 |- #2} 

\newcommand{\nudge}{0.04}
\newcommand{\colwidth}{1.5} 
\newcommand{\setwidth}{2.5} 
\newcommand{\extravertsep}{0.3} 
\newcommand{\horizsep}{4.5}

\tikzset{
outline/.style = {
    rectangle,
    inner sep=0,
    outer sep=0,
    minimum width={\colwidth cm},
    anchor=north, 
},
columnoutline/.style = {
    outline,
    very thick,
},
compcolumnoutline/.style = {
    outline,
},
colsubset/.style = {
    draw,
    rectangle,
    rounded corners = 0.07cm,
    minimum width = {\colwidth cm - 4*\nudge cm},
    inner sep = 0pt,
    anchor=north, 
},
compcolsubset/.style ={
    colsubset,
},
entrysubset/.style = {
    draw,
    thick,
    rectangle,
    rounded corners = 0.3cm,
    minimum width = {\setwidth cm - 0*\nudge cm},
    inner sep = 0pt,
    outer sep = 0pt,
    text centered,
    anchor=north, 
    align=center,
    dashed,
},
entrysubsubset/.style = {
    entrysubset,
    thin,
    minimum width = {\setwidth cm - 4*\nudge cm},
    dotted,
},
entrysubsubsubset/.style = {
    entrysubset,
    thin,
    minimum width = {\setwidth cm - 8*\nudge cm},
    dotted,
},
}
\newcommand{\Tfixedpattern}{crosshatch dots}

\newcommand{\totaldepth}{10.6}
\newcommand{\jdepth}{6.3}
\newcommand{\kdepth}{5.4}
\pgfmathsetmacro\jcircdepth{\totaldepth-\jdepth} 
\pgfmathsetmacro\kcircdepth{\totaldepth-\kdepth} 

\newcommand{\DjnotAlength}{1.7}
\newcommand{\DknotBlength}{1.3}
\newcommand{\Dlength}{2.4}
\pgfmathsetmacro\AcapDjlength{\Dlength - \DjnotAlength} 
\pgfmathsetmacro\BcapDklength{\Dlength - \DknotBlength} 
\newcommand{\Acirclength}{2.7}
\newcommand{\Bcirclength}{2}
\pgfmathsetmacro\Ulength{\jdepth - \Dlength - \Acirclength} 
\pgfmathsetmacro\Wlength{\kdepth - \Dlength - \Bcirclength} 
\pgfmathsetmacro\Nlength{\jcircdepth - \Bcirclength - \Wlength} 



\begin{scope}[shift={(-\horizsep,0)}]

\node[draw, columnoutline, minimum height={\jdepth cm + 2*\nudge cm}, at={(0,0-\nudge)}] {};
\node[above] at (0,-\nudge) {\(\jc\)};
\node[colsubset,
    minimum height={\DjnotAlength cm - 2*\nudge cm},
    at={(0,0+\nudge)}] 
    (DjnotAset)
    {\(D_\jc \setminus A\)};
\node[colsubset,
    minimum height={\Dlength cm - \DjnotAlength cm- 2*\nudge cm},
    pattern=\Tfixedpattern,
    pattern color=paleazure,
    at={(0,\DjnotAlength+\nudge)}]
    (AcapDjset)
    {\(A \cap D_\jc\)};
\node[colsubset,
    minimum height={\Acirclength cm - 2*\nudge cm},
    fill=paleazure,
    at={(0,\Dlength+\nudge)}] 
    (AnotDjset)
    {\(A \setminus D_\jc\)};
\node[colsubset,
    minimum height={\Ulength cm -2*\nudge cm},
    at={(0,\Dlength+\Acirclength+\nudge)}]
    (Uset)
    {\(U\)};

\begin{scope}[shift={(0,\jdepth)}] 
\begin{scope}[shift={(0,\extravertsep)}] 
\node[draw, compcolumnoutline, minimum height={\jcircdepth cm + 2*\nudge cm}, at={(0,0-\nudge)}] {};
\node[below] at (0,\jcircdepth+\nudge) {\(\jc^\circ\)};
\node[compcolsubset,
    minimum height={\Bcirclength cm - 2*\nudge cm},
    fill=palegold,
    at={(0,0+\nudge)}]
    (Bcircset)
    {\(B^\circ\)};
\node[compcolsubset,
    minimum height={\Wlength cm - 2*\nudge cm},
    at={(0,\Bcirclength+1*\nudge)}]
    (Wcircset)
    {\(W^\circ\)};
\node[compcolsubset,
    minimum height={\Nlength cm - 2*\nudge cm},
    at={(0,\Bcirclength+\Wlength+1*\nudge)}]
    (Njset)
    {\(N_\jc\)};
\end{scope} 
\end{scope} 

\end{scope} 

\begin{scope}[shift={(\horizsep,0)}]

\node[draw, columnoutline, minimum height={\kdepth cm + 2*\nudge cm}, at={(0,0-\nudge)}] {};
\node[above] at (0,-\nudge) {\(\kc\)};
\node[colsubset,
    minimum height={\DknotBlength cm - 2*\nudge cm},
    at={(0,0+1*\nudge)}]
    (DknotBset)
    {\(D_\kc \setminus B\)};
\node[colsubset,
    minimum height={\Dlength cm - \DknotBlength cm - 2*\nudge cm},
    pattern=\Tfixedpattern,
    pattern color=palegold,
    at={(0, \DknotBlength+\nudge)}]
    (BcapDkset)
    {\(B \cap D_\kc\)};
\node[colsubset,
    minimum height={\Bcirclength cm - 2*\nudge cm},
    fill=palegold,
    at={(0,\Dlength+\nudge)}]
    (BnotDkset)
    {\(B \setminus D_\kc\)};
\node[colsubset,
    minimum height={\Wlength cm - 2*\nudge cm},
    at={(0,\Dlength+\Bcirclength+\nudge)}]
    (Wset)
    {\(W\)};

\begin{scope}[shift={(0,\kdepth)}] 
\begin{scope}[shift={(0,\extravertsep)}] 
\node[draw, compcolumnoutline, minimum height={\kcircdepth cm + 2*\nudge cm}, at={(0,0-\nudge)}] {};
\node[below] at (0,\kcircdepth+\nudge) {\(\kc^\circ\)};
\node[compcolsubset,
    minimum height={\Acirclength cm - 3*\nudge cm},
    fill=paleazure,
    at={(0,0+1*\nudge)}]
    (Acircset)
    {\(A^\circ\)};
\node[compcolsubset,
    minimum height={\Ulength cm - 2*\nudge cm},
    at={(0,\Acirclength+1*\nudge)}]
    (Uicircset)
    {\(U^\circ\)};
\node[compcolsubset,
    minimum height={\Nlength cm - 3*\nudge cm},
    at={(0,\Acirclength+\Ulength+2*\nudge)}]
    (Nkset)
    {\(N_\kc\)};
\end{scope} 
\end{scope} 

\end{scope} 

{
\node[above] at (0,-2*\nudge) {\(\set{1, \ldots, d}\)};

\begin{scope}[shift={(0,-1.5*\nudge)}] 

\node[entrysubset,
    minimum height={\Dlength cm - 0*\nudge cm},
    at={(0,0)}]
    (bothcol) {};
    \node[entrysubsubset,
        minimum height={\BcapDklength cm - 3*\nudge cm},
        pattern=\Tfixedpattern,
        pattern color=palegold,
        at={(0,0+2*\nudge)}]
        (tofBcapDkentryset)
        {%
        \(\scriptstyle t(B \cap D_\kc)\)%
        };
    \node[entrysubsubset,
        minimum height={\DjnotAlength cm - \BcapDklength cm - 2*\nudge cm}, 
        at={(0,\BcapDklength+\nudge)}]
        {%
            \(\scriptscriptstyle t(D_\jc \setminus A) \,\cap\, t(D_\kc \setminus B)\)%
        };
    \node[entrysubsubset,
        minimum height={\AcapDjlength cm - 3*\nudge cm},
        pattern=\Tfixedpattern,
        pattern color=paleazure,
        at={(0,\DjnotAlength+\nudge)}]
        (tofAcapDjentryset)
        {%
        \(\scriptstyle t(A \cap D_\jc)\)%
        };
    \node[entrysubsubset, draw=none,
        minimum height={\DjnotAlength cm - 3*\nudge cm},
        at={(0,0+3*\nudge)}]
        (tofDjnotAentryset) {};
    \node[entrysubsubset, draw=none,
        minimum height={\DknotBlength cm - 4*\nudge cm},
        at={(0,\BcapDklength+2.5*\nudge)}]
        (tofDknotBentryset) {};
\begin{scope}[shift={(0,\extravertsep*0.33+1*\nudge)}] 
\begin{scope}[shift={(0,\Dlength)}] 
\node[entrysubset,
    minimum height={\Acirclength cm + \Ulength cm - 0*\nudge cm}, 
    at={(0,0)}]
    (jcol) {};
    \node[entrysubsubset,
        fill=paleazure,
        minimum height={\Acirclength cm - 3*\nudge cm},
        at={(0,0+2*\nudge)}]
        (tofAnotDjentryset)
        {%
        \(t(A \setminus D_\jc)\) \\[-4pt]
        \(\scriptstyle=\) \\[-3pt]
        \( t^\circ (A^\circ)\)%
        };
    \node[entrysubsubset,
        minimum height={\Ulength cm - 3*\nudge cm}, 
        at={(0,\Acirclength+\nudge)}]
        (tofUentryset)
        {%
        \(\scriptstyle t(U) \,=\, t^\circ (U^\circ)\)%
        };
\end{scope} 
\begin{scope}[shift={(0,\extravertsep*0.33+1*\nudge)}] 
\begin{scope}[shift={(0,\Dlength+\Acirclength+\Ulength)}] 
\node[entrysubset,
    minimum height={\Wlength cm+\Bcirclength cm - 0*\nudge cm},
    at={(0,0)}] (kcol) {};
    \node[entrysubsubset,
        fill=palegold,
        minimum height={\Bcirclength cm - 3*\nudge cm},
        at={(0,0+2*\nudge)}]
        (tofBnotDkentryset)
        {%
        \(t(B \setminus D_\kc)\) \\[-4pt]
        \(\scriptstyle=\) \\[-3pt]
        \( t^\circ (B^\circ)\)%
        };
    \node[entrysubsubset,
        minimum height={\Wlength cm - 3*\nudge cm},
        at={(0,\Bcirclength+1*\nudge)}]
        (tofWentryset)
        {%
        \(\scriptstyle t(W) \,=\, t^\circ (W^\circ)\)%
        };
\end{scope} 
\begin{scope}[shift={(0,\extravertsep*0.33+1*\nudge)}] 
\begin{scope}[shift={(0,\Dlength+\Wlength+\Bcirclength+\Acirclength+\Ulength)}] 
\node[entrysubset,
    minimum height={\Nlength cm - 0*\nudge cm},
    at={(0,0)}]
    (neithercol)
    {
    \(\scriptstyle t^\circ(N_\jc) \,=\, t^\circ(N_\kc)\)%
    };
\end{scope} 
\end{scope} 
\end{scope} 
\end{scope} 

\end{scope} 
}

\draw[thick, dashed, rounded corners]
    ($(bothcol.north west)+(0.3,0)$) coordinate(temp)
    -- (\combinexy{temp}{0,-1.1})
    node[shift={(-1.3,-0.3)}, inner sep=0]
    {\scalebox{0.9}{\(t(D_\jc) = t(D_\kc) = \col_\jc(t) \cap \col_\kc(t)\)}};
\draw[thick, dashed, rounded corners]
    ($(jcol.north east)+(0,0.6)$)
    --++ (1,0) coordinate(temp)
    -- (\combinexy{temp}{0,-1.1})
    node[shift={(0.05,-0.3)}, inner sep=0]
    {\scalebox{0.9}{\(\col_\jc(t)\,\backslash\, \col_\kc(t)\)}};
\draw[thick, dashed, rounded corners]
    ($(kcol.south west)+(0,-1)$)
    --++ (-0.8,0) coordinate(temp)
    -- (\combinexy{temp}{0,\totaldepth+1.1})
    node[shift={(-0.05,0.3)}, inner sep=0]
    {\scalebox{0.9}{\(\col_\kc(t) \,\backslash\, \col_\jc(t)\)}};
\draw[thick, dashed, rounded corners]
    ($(neithercol.south east)+(-0.3,0)$) coordinate(temp)
    -- (\combinexy{temp}{0,\totaldepth+1.1})
    node[shift={(0.7,0.3)}, inner sep=0]
    {\scalebox{0.9}{\(\set{1,\ldots, d} \,\backslash\, \bigl(  \Col_\jc(t)  \cup \Col_\kc(t) \bigr)\)}};

\newcommand{\arrowgap}{0.25}
\newcommand{\bracegap}{0.25}

\newcommand{\straightLtoRarrow}[3]{
\draw[|->, semithick]
    ($(#1.east) + (\arrowgap,0)$)
    -- node[above] {\(#2\)}
    (\combinexy{-\setwidth*0.5+\nudge-\arrowgap+#3,0}{#1});
}
\newcommand{\straightRtoLarrow}[3]{
\draw[|->, semithick]
    ($(#1.west) + (-\arrowgap,0)$)
    -- node[above] {\(#2\)}
    (\combinexy{\setwidth*0.5-\nudge+\arrowgap+#3,0}{#1});
}
\newcommand{\flexiRtoLarrow}[5]{
\draw[|->, semithick]
    ($(#1.west) + (-\arrowgap,0)$)
    to[out=180, in=0, looseness=0.5] node[pos=0.15, #5=0.1cm] {\(#2\)}
    ($(#4.east)+(\arrowgap+#3,0)$);
}

\draw[decorate,decoration={brace,amplitude=6pt,raise=4pt}, semithick] 
    (tofDjnotAentryset.south west)
     -- (tofDjnotAentryset.north west);
\straightLtoRarrow{DjnotAset}{t}{-\bracegap}
\straightLtoRarrow{AcapDjset}{t}{0}
\straightLtoRarrow{AnotDjset}{t}{0}
\straightLtoRarrow{Uset}{t}{0}

\straightLtoRarrow{Bcircset}{t^\circ}{0}
\straightLtoRarrow{Wcircset}{t^\circ}{0}
\straightLtoRarrow{Njset}{t^\circ}{0}

\draw[decorate,decoration={brace,amplitude=6pt,raise=4pt}, semithick] 
    (tofDknotBentryset.north east)
     -- (tofDknotBentryset.south east);
\flexiRtoLarrow{DknotBset}{t}{\bracegap}{tofDknotBentryset}{above}
\flexiRtoLarrow{BcapDkset}{t}{0}{tofBcapDkentryset}{below}
\flexiRtoLarrow{BnotDkset}{t}{0}{tofBnotDkentryset}{above}
\flexiRtoLarrow{Wset}{t}{0}{tofWentryset}{above}
    
\flexiRtoLarrow{Acircset}{t^\circ}{0}{tofAnotDjentryset}{below}
\flexiRtoLarrow{Uicircset}{t^\circ}{0}{tofUentryset}{below}
\straightRtoLarrow{Nkset}{t^\circ}{0}

\end{tikzpicture}
\caption{The sets of boxes and their entries considered in the proof of \Cref{prop:Psi_respects_Garnir_relations}.
Column \(\jc\) of \(\Y{\lambda}\) and column \(\jc^\circ\) of \(\Y{\lambda^\circ}\) are shown on the left, column \(\kc\) of \(\Y{\lambda}\) and column \(\kc^\circ\) of \(\Y{\lambda^\circ}\) are shown on the right, and the set \(\{1,\ldots, d\}\) containing their entries is shown in the middle.
The solid shading indicates the boxes, and their entries, that may be moved by elements of \(\mathcal{T}\); the dotted shading indicates the boxes, and their entries, which lie in \(A \sqcup B\) but which are fixed by \(\mathcal{T}\). 
\Cref{subclaim:bigEnough} states that the number of solidly shaded entries, plus \(|N_\kc|\), is strictly more than \(\lambda_{\kc^\circ}^{\circ'}\).
The sets \(W = \Col_\kc[\lambda] \backslash (B \cup D_\kc)\) and \(W^\circ = \setbuild{a \in \Col_{\jc^\circ}[\lambda^\circ] }{ t(a) \in t(W) }\) are defined analogously to the sets of boxes \(U\) and \(U^\circ\) used in the proof of \Cref{subclaim:bigEnough}; they are indicated here only in order to complete the partition and are not used in the proof.
\vspace*{24pt}
}
\label{fig:column_subsets}
\end{figure}

\begin{subexample}
\marginnote{
\begin{center}
\begin{tikzpicture}[x=\xlen cm, y=-\ylen cm]
\fill[paleazure] (0+\fb,2+\fb) rectangle (1-\fb,3-\fb);
\fill[paleazure] (0+\fb,3+\fb) rectangle (1-\fb,4-\fb);
\fill[paleazure] (0+\fb,4+\fb) rectangle (1-\fb,5-\fb);
\fill[palegold] (1+\hgap+\fb,0+\fb) rectangle (1+\hgap+1-\fb,1-\fb);
\fill[palegold] (1+\hgap+\fb,1+\fb) rectangle (1+\hgap+1-\fb,2-\fb);
\fill[palegold] (1+\hgap+\fb,2+\fb) rectangle (1+\hgap+1-\fb,3-\fb);
    \node at (0.5,-0.5) {\(j\)};                        \node at (1+\hgap+0.5,-0.5) {\(k\)};
\drawthreesides{0}{0}                               \drawthreesides{1+\hgap}{0}
                                                        \node at ($(1+\hgap,0)+(\DNshift)$) {\DNstyle{D_k}};
    \node at (0.5,0.5) {\(1\)};                         \node at (1+\hgap+0.5,0.5) {\(2\)};
                                                        \node[vdarkgold] at ($(1+\hgap,0)+(\ABshift)$) {\ABstyle{B}};
\drawthreesides{0}{1};                              \drawthreesides{1+\hgap}{1}
    \node at ($(0,1)+(\DNshift)$) {\DNstyle{D_j}};
    \node at (0.5,1.5) {\(2\)};                         \node at (1+\hgap+0.5,1.5) {\(3\)};
                                                        \node[vdarkgold] at ($(1+\hgap,1)+(\ABshift)$) {\ABstyle{B}};
\drawthreesides{0}{2};                              \drawthreesides{1+\hgap}{2}
    \node at ($(0,2)+(\DNshift)$) {\DNstyle{D_j}};
    \node at (0.5,2.5) {\(6\)};                         \node at (1+\hgap+0.5,2.5) {\(5\)};
    \node[darkazure] at ($(0,2)+(\ABshift)$) {\ABstyle{A}};\node[vdarkgold] at ($(1+\hgap,2)+(\ABshift)$) {\ABstyle{B}};
\drawthreesides{0}{3};                              \drawfoursides{1+\hgap}{3}
                                                        \node at ($(1+\hgap,3)+(\DNshift)$) {\DNstyle{D_k}};
    \node at (0.5,3.5) {\(8\)};                         \node at (1+\hgap+0.5,3.5) {\(6\)};
    \node[darkazure] at ($(0,3)+(\ABshift)$) {\ABstyle{A}};
\drawfoursides{0}{4};
    \node at (0.5,4.5) {\(9\)};
    \node[darkazure] at ($(0,4)+(\ABshift)$) {\ABstyle{A}};
\begin{scope}[shift={(0,\vgap)}] 
\fill[palegold] (0+\fb,6+\fb) rectangle (1-\fb,7-\fb);
\fill[palegold] (0+\fb,8+\fb) rectangle (1-\fb,9-\fb);
\fill[paleazure] (1+\hgap+\fb,4+\fb) rectangle (1+\hgap+1-\fb,5-\fb);
\fill[paleazure] (1+\hgap+\fb,5+\fb) rectangle (1+\hgap+1-\fb,6-\fb);
\fill[mygrey] (1+\hgap+\fb,6+\fb) rectangle (1+\hgap+1-\fb,7-\fb);
\fill[mygrey] (1+\hgap+\fb,7+\fb) rectangle (1+\hgap+1-\fb,8-\fb);
                                                    \drawthreesides{1+\hgap}{4}
                                                        \node at (1+\hgap+0.5,4.5) {\(9\)};
                                                        \node[darkazure] at ($(1+\hgap,4)+(\ABshift)$) {\ABstyle{A^\circ}};
\drawthreesides{0}{5};                              \drawthreesides{1+\hgap}{5}
    \node at ($(0,5)+(\DNshift)$) {\DNstyle{N_j}};
    \node at (0.5,5.5) {\(7\)};                         \node at (1+\hgap+0.5,5.5) {\(8\)};
                                                        \node[darkazure] at ($(1+\hgap,5)+(\ABshift)$) {\ABstyle{A^\circ}};
\drawthreesides{0}{6};                              \drawthreesides{1+\hgap}{6};
    \node at (0.5,6.5) {\(5\)};                         \node at (1+\hgap+0.5,6.5) {\(7\)};
                                                        \node at ($(1+\hgap,6)+(\DNshift)$) {\DNstyle{N_k}};
    \node[vdarkgold] at ($(0,6)+(\ABshift)$) {\ABstyle{B^\circ}};
\drawthreesides{0}{7};                              \drawthreesides{1+\hgap}{7};
    \node at ($(0,7)+(\DNshift)$) {\DNstyle{N_j}};      \node at ($(1+\hgap,7)+(\DNshift)$) {\DNstyle{N_k}};
    \node at (0.5,7.5) {\(4\)};                         \node at (1+\hgap+0.5,7.5) {\(4\)};
\drawfoursides{0}{8};                               \drawfoursides{1+\hgap}{8};
    \node at (0.5,8.5) {\(3\)};                         \node at (1+\hgap+0.5,8.5) {\(1\)};
    \node[vdarkgold] at ($(0,8)+(\ABshift)$) {\ABstyle{B^\circ}};
    \node at (0.5,9.5) {\(j^\circ\)};                   \node at (1+\hgap+0.5,9.5) {\(k^\circ\)};
\end{scope}
\end{tikzpicture}
\end{center}
}[-7.5cm]%
In the example shown in the margin, now with full annotations, 
\(A^\circ = \bigl\{(4,\kc^\circ), (5,\kc^\circ) \bigr\}\) and \(B^\circ = \bigl\{ (1,\jc^\circ), (3,\jc^\circ) \bigr\}\) so \(|A^\circ| + |B^\circ| = 4 \not> \lambdacircprime_{\kc^\circ} = 5\).
Therefore \(A^\circ\) and \(B^\circ\) cannot be used directly to define a Garnir relation. 
We have \(N_\kc = \bigl\{(2,k^\circ), (3,\kc^\circ) \bigr\}\), in bijection with \(N_\jc = \bigl\{(2,\jc^\circ), (4,\jc^\circ) \bigr\}\), and \(|A^\circ \sqcup N_\kc| + |B^\circ| = 6\).
Therefore \(A^\circ \sqcup N_\kc\) and~\(B^\circ\) define a Garnir relation.
The relevant boxes are shaded in the margin. 
By \Cref{subclaim:Psi} at the end of this proof, 
\(\Psi(\Grel_{(t,A,B)}) = (-1)^{\surplus(t)} \Grel_{(t^\circ,A^\circ\sqcup N_\kc, B^\circ)}\).
\end{subexample}

\begin{subclaim}\label{subclaim:bigEnough}
\(\abs{A^\circ \sqcup N_\kc} + \abs{B^\circ} > \lambdacircprime_{\kc^\circ}\).
\end{subclaim}

\begin{proof}
Let \(U = \Col_\jc[\lambda] \backslash (A \cup D_\jc)\),
and let \(U^\circ = \setbuild{b \in \Col_{\kc^\circ}[\lambda^\circ] }{ t^\circ(b) \in t(U) }\).
Since each entry in column \(k^\circ\) of \(t^\circ\) is,
by construction, not in column~\(\kc\) of~\(t\), and either (i) in~\(A\) and not in both columns of \(t\), so in \(A^\circ\);
(ii) not in~\(A\) but in column \(\jc\) of \(t\), so in \(U^\circ\); (iii) not in~\(A\) and not in column \(\jc\) of~\(t\),
so in \(N_k\), we have
\(\lambdacircprime_{\kc^\circ} = |A^\circ| + |N_k| + |U^\circ|\).
(This can be seen from the partition of column \(k^\circ\) in \Cref{fig:column_subsets}.)
Therefore the claim is equivalent to
\begin{align*}
|A^\circ| + |N_k| + |B^\circ| > \lambdacircprime_{\kc^\circ} &\iff |B^\circ| > |U^\circ| \\
&\iff |B| > |U^\circ| + |B| - |B^\circ| \\
&\iff |B| > |U| + |B \cap D_k|
\end{align*}
where the final line holds because \(U^\circ\) is in bijection, via \(t^\circ\) and \(t\), with \(U\) and, by~\eqref{eq:ABcirc},
\(t^\circ\) defines a bijection from \(B^\circ\) to \(B \backslash D_k\), giving \(|B \cap D_k| = |B| - |B^\circ|\).
(Both bijections can  be seen in \Cref{fig:column_subsets}.)
Now \(|B \cap D_k| \leq |D_j \backslash A|\) since each entry of \(B \cap D_k\) is in both column
\(\kc\) and column \(\jc\) of \(t\), and so 
is in \(D_\jc\), but these entries 
are not in~\(A\), since \(t(A) \cap t(B) = \varnothing\). (This can be seen in Figure~1 by following the arrow \(t\) from \(B \cap D_k\).)
Therefore it is sufficient to prove that
\[ |B| > |U| + |D_j \backslash A|. \]
This holds because \(|A| + |B| > \lambda_j'\), and so \(|B|\) is strictly more than the number of entries in column \(j\) of \(t\)
not in \(A\); these are precisely the entries in the boxes in \(U\) and \(D_j \backslash A\).
\end{proof}

\begin{subclaim}\label{subclaim:Psi}
\(\Psi (\Grel_{(t,A,B)}) = (-1)^{\surplus(t)} \Grel_{(t^\circ, A^\circ \sqcup N_\kc, B^\circ )}\).
\end{subclaim}

\begin{proof}
Let \(\mathcal{R}\) be a set of left coset representatives for \(S_{A^\circ \sqcup N_\kc} \times S_{B^\circ}\) in \(S_{A^\circ \sqcup N_\kc \sqcup B^\circ}\), chosen so that each representative that keeps all the boxes in~\(N_\kc\) in column \(\kc^\circ\) fixes all these boxes.
Let \(\mathcal{Q} \subseteq \mathcal{R}\) be this set of representatives fixing all the boxes in \(N_\kc\);
then \(\mathcal{Q}\) forms a complete irredundant set of left coset representatives of \(S_{A^\circ} \times S_{B^\circ}\) in \(S_{A^\circ\sqcup B^\circ}\).
By \Cref{subclaim:S_circ} we have
\(
     \sum_{\sigma \in \mathcal{Q}} |t^\circ \cdot \sigma| \sign{\sigma} 
     = 
     \sum_{\tau^\circ \in \mathcal{T}^\circ} |t^\circ \cdot \tau^\circ| \sign{\tau^\circ} .
\)
Thus
\begin{align*}
\Grel_{(t^\circ, A^\circ \sqcup N_\kc, B^\circ )}
    &=  \sum_{\sigma \in \mathcal{R} \backslash \mathcal{Q}} |t^\circ \cdot \sigma| \sign{\sigma} 
        + \sum_{\tau^\circ \in \mathcal{T}^\circ} |t^\circ \cdot \tau^\circ| \sign{\tau^\circ} .
\end{align*}
Each summand \(|t^\circ \cdot \sigma|\) in the first sum is \(0\), because \(\sigma\) moves a box containing
an entry in \(N_\kc\) into column \(\jc\), in which this entry is already contained in a box in \(N_\jc\).
By~\eqref{eq:S_circ_corollary} the second summand is \((-1)^{\surplus(t)} \Psi(\Grel_{(t,A,B)})\), as required.
\end{proof}

We thus have \(\Psi (\Grel_{(t,A,B)}) \in \GRspace{\lambda^\circ}{V^\dual}\), completing the proof.
\end{proof}

To complete the proof of \Cref{thm:comp_partition_iso} we 
need that \(\nabla^\lambda V\) and \(\nabla^{\lambda^\circ}V^\dual\) have the same dimension.
This result is well known: it is proved, for instance, in 
Proposition~7.1 in \cite{PagetWildonSL2}, where it is shown that
the map \(t \mapsto t^\circ\) is a bijection \(\SSYT_{\{1,\ldots,d\}}(\lambda) \rightarrow \SSYT_{\{1,\ldots,d\}}(\lambda^\circ)\),
and so the dimensions agree by \Cref{prop:ssytBasis}.

We can now prove the main results of this section.

\begin{proof}[Proof of \Cref{thm:comp_partition_iso}]
The \(KG\)-isomorphism \(\Psi \colon \Wedge^{\lambda'}\!V \to \Wedge^{\lambdacircprime}\!V^\dual \otimes (\det V)^s\) induces, by \Cref{prop:complement}, a homomorphism
\(\GRspace{\lambda}{V} \to \GRspace{\lambda^\circ}{V^\dual}\).
Hence, by \Cref{prop:nablaPresentation}, \(\Psi\) induces a surjective \(KG\)-homomorphism \(\nabla^\lambda V \rightarrow \nabla^{\lambda^\circ} V^\dual \otimes (\det V)^s\).
Since these modules have the same dimension, \(\Psi\) is an isomorphism.
\end{proof}

\comppartisoSL*

\begin{proof}
Take \(G = \SL_2(K)\), \(d = l+1\) and \(V = \Sym^{l} E\) in \Cref{thm:comp_partition_iso}.
Since \(V\) is the restriction of a polynomial representation of \(\GL_2(K)\), \(\det V\) is the
trivial representation.
By the discussion following \Cref{prop:SchurContravariantDual} and \Cref{lemma:SL2duality},
we have \((\Sym^{l} \! E)^\dual \cong \Sym_{l} E\). The corollary follows.
\end{proof}

\section{Wronskian isomorphism (proof of \Cref{thm:Wronskian})}
\label{sec:Wronskian}

This section consists of a proof of \Cref{thm:Wronskian}, restated below.

\Wronskian*

The proof is split into two subsections: the first shows that this map is a \(K\)-linear isomorphism, the second that it respects the group action.

\subsection{A \texorpdfstring{\(K\)}{K}-linear isomorphism}

We introduce some notation to describe the given map and show that it is a \(K\)-linear isomorphism \(\Sym_m \Sym^\ell\! E \to \Wedge^m \Sym^{\ell+m-1}\! E\).
As in the introduction, we write \(E = \langle X, Y \rangle_K\).
Thus for each \(r \in \mathbb{N}\), by~\eqref{eq:symUpperBasis},
\(\Sym^r E\) has a basis \(\set{Y^r, XY^{r-1}, \ldots, X^r}\).

An \(m\)-\emph{multiindex} is an element of \(\Z^m\). 
The symmetric group \(S_m\) acts on \(m\)-multiindices by place permutation:
\((i_1,\ldots,i_m) \cdot \sigma = (i_{1\sigma^{-1}}, \ldots, i_{m\sigma^{-1}})\). 
Let \(\Stab \ix \leq S_m\) denote the stabiliser of the \(m\)-multiindex \(\ix\).
Let \(\dx = (m-1,m-2,\ldots, 0)\).
Addition and subtraction of \(m\)-multiindices is defined componentwise.

\begin{definition}
Given an \(m\)-multiindex \(\ix\) with entries from \(\set{0,\ldots, \ell}\), define
\begin{align*}
    \Ftensor{\ix} &= X^{i_1}Y^{\ell-i_1} \otimes \cdots \otimes X^{i_m}Y^{\ell-i_m}; \\
    \Fsym{\ix} &= \textstyle{\sum_{\sigma \in \Stab \ix \backslash S_m} \Ftensor{\ix \cdot \sigma}};
\intertext{and given an \(m\)-multiindex \(\jx\) with entries from \(\set{0,\ldots, \ell+m-1}\), define}
    \Fwedge{\jx} &= X^{j_1}Y^{\ell+m-1-j_1} \wedge \cdots \wedge X^{j_m}Y^{\ell+m-1-j_m}.
\end{align*}
\end{definition}

By our definition of \(\Sym_m \Sym^\ell\bs V\) as the fixed points for the action of~\(S_m\) on \((\Sym^\ell\bs V)^{\otimes m}\),
this module has as a basis all \(\Fsym{\ix}\) for weakly decreasing \(\ix \in \{0,\ldots, \ell\}^m\);
let \(\Mdecreasing\) be the set of such multiindices.
Thus we can write the restriction to \(\Sym_m \Sym^l E\) of the map in the statement of \Cref{thm:Wronskian} as follows.

\begin{definition}
\label{def:zeta}
Let \(\zeta \colon \Sym_m \Sym^\ell\bs E \rightarrow \bigwedge^m \Sym^{\ell+m-1}\bs E\) be the \(K\)-linear map
defined by linear extension of
\[
\zeta \Fsym{\ix}
    = \longsubsum{\sum}{\sigma \in \Stab \ix \backslash S_m}{ \Fwedge{\ix \cdot \sigma + \mathbf{d}}}
\]
for \(\ix \in \Mdecreasing\).
\end{definition}

\newcommand{\firstcolwidest}{X^5 \crampedwedge X^2Y^3 \crampedwedge XY^4}
\newcommand{\secondcolwidest}{X^4Y \crampedwedge X^3Y^2 \crampedwedge XY^4}
\newcommand{\thirdcolwidest}{XY^2 \crampedtensor XY^2 \crampedtensor X^3}

\begin{example}\label{ex:zeta} Take \(m = \ell = 3\). Omitting some parentheses for readability, we have
\begin{align*}
\Fsym{3,1,1}
    &= \flexbox{\Ftensor{3,1,1}}{\firstcolwidest}
        + \flexbox{\Ftensor{1,3,1}}{\secondcolwidest}
        \!\!+ \flexbox{\Ftensor{1,1,3}}{XY^4 \crampedwedge X^2Y^3 \crampedwedge X^5} \\
    &= \flexbox{X^3 \crampedtensor XY^2 \crampedtensor XY^2}{\firstcolwidest}
        + \flexbox{XY^2 \crampedtensor X^3 \crampedtensor XY^2}{\secondcolwidest}
        \!\!+ XY^2 \crampedtensor XY^2 \crampedtensor X^3
\intertext{and, since \(\dx = (2,1,0)\),}
\zeta \Fsym{3,1,1}
    &= \flexbox{\Fwedge{5,2,1}}{\firstcolwidest}
        + \flexbox{\Fwedge{3,4,1}}{\secondcolwidest}
        \!\!+ \flexbox{\Fwedge{3,2,3}}{\thirdcolwidest} \\
    &= \flexbox{\Fwedge{5,2,1}}{\firstcolwidest}
        - \flexbox{\Fwedge{4,3,1}}{\secondcolwidest} \\
    &= X^5 \crampedwedge X^2Y^3 \crampedwedge XY^4
        - X^4Y \crampedwedge X^3Y^2 \crampedwedge XY^4
\end{align*}
where we have aligned the summands obtained by addition of \(\dx\) to the index.
\end{example}

\begin{remark}
\label{remark:extended_domain_of_zeta}
In our notation, the map $\mleft( \Sym^\ell \! E \mright)^{\tensor m} \to \Wedge^m \Sym^{\ell+m-1}\! E$ in the statement of \Cref{thm:Wronskian} is defined by
\[
    \Ftensor{\ix} \mapsto \Fwedge{\ix + \mathbf{d}}
\]
for all \(\ix \in \set{0, \ldots l}^m\).
This \(K\)-linear map will be useful in the following subsection when computing with \(\zeta\).
However, this extended map \emph{does not} respect the group action.
For example, let \(\ell = m = 2\), so \(\dx = (1,0)\), and let \(\ix = (1,2)\).
Then the extended map sends \(\Ftensor{\ix} = XY \otimes X^2\) to \(\Fwedge{\ix + \dx} = X^2Y \wedge X^2Y = 0\).
But choosing \(J = \begin{psmallmatrix}
            0 & 1 \\[1pt] -1 & 0 \\
\end{psmallmatrix} \in \SL_2(k) \), we have \(J \Ftensor{\ix} = - XY \tensor Y^2\), which is 
sent by the extended map to \(- X^2Y \wedge Y^3 \neq 0\).
\end{remark}

We totally order multiindices of a given length lexicographically comparing entries starting at the left. 
For example  \((3,2,2) < (3,3,1) < (4,2,1)\).

\begin{lemma}
The \(K\)-linear map \(\zeta\) is injective.
\end{lemma}

\begin{proof}
Let \(x = \sum_{\ix \in \Mdecreasing} \alpha_\ix \Fsym{\ix} \in \Sym_m \Sym^l E\) be non-zero.
There is some non-zero \(\alpha_\ix \in K\);
choose \(\mx\) maximal with \(\alpha_{\mx} \not= 0\).
Note that by~\eqref{eq:symUpperBasis}, a basis for \(\Wedge^m \Sym^{m+l-1} E\) is
\(
\setbuild{ \Fwedge{\jx} }{
    l+m-1 \geq j_1 > \ldots > j_{m} \geq 0
}
\).
We claim that, with respect to this basis, the coefficient in \(\zeta x\) of \(\Fwedge{\mx + \dx}\) is \(\alpha_{\mx}\), and hence \(\zeta x \neq 0\).

The summands of \(\zeta \Fsym{\ix}\) are labelled by \((\ix \cdot \sigma + \dx) \cdot \tau_\sigma \) where \(\sigma\) ranges over right coset representatives of \(\Stab{\ix}\) in \(S_m\) and \(\tau_\sigma \in S_m\) sorts the multiindex 
$\ix \cdot \sigma + \dx$ into weakly decreasing order (so that it indeed labels a basis element).
If $\ix \in \Mdecreasing$
then $(\ix \cdot \sigma + \dx) \cdot \tau_\sigma \leq \ix + \dx$ in the lexicographical order, with 
equality if and only if \(\sigma \in \Stab \ix\) (and hence \(\tau_\sigma = \id\)).
Hence, for \(\ix \leq \mx\), all such labels satisfy \((\ix \cdot \sigma + \dx) \cdot \tau_\sigma \leq \ix + \dx \leq \mx + \dx\), with equality if and only if \(\ix = \mx\) and \(\sigma \in \Stab{\mx}\) (and hence \(\tau_\sigma = \id\)).
Thus the basis element \(\Fwedge{\mx + \dx}\) has coefficient \(1\) in \(\zeta \Fsym{\mx}\) and zero in \(\zeta \Fsym{\ix}\) for \(\ix < \mx\).
The claim follows.
\end{proof}

A simple calculation shows that \(\Sym_m \Sym^l E\) and \(\Wedge^m \Sym^{l+m-1} \! E\) both have dimension \(\binom{\ell+m}{m}\).
Therefore \(\zeta\) is a \(K\)-linear isomorphism.

\subsection{\texorpdfstring{\(\zeta\)}{zeta} is a homomorphism of \texorpdfstring{\(K\!\SL_2(K)\)}{KSL2(K)}-modules}
\label{subsec:Wronskian_group_equivariance}

As indicated at the end of the introduction, to prove \Cref{thm:Wronskian} it suffices to show \(\zeta\) respects the action of \(\SL_2(K)\).
Since \(\SL_2(K)\) is generated by the elements
\[
M_\gamma =
    \left( \begin{matrix} 1 & 0 \\ \gamma & 1 \end{matrix} \right),
\quad
J =
    \left( \begin{matrix} 0 & 1 \\ -1 & 0 \end{matrix} \right)
\]
for \(\gamma \in K\), it suffices to show that \(\zeta\) commutes with their action.

For \(J\) this is 
straightforward.
Let \(\repx{s} = (s,\ldots, s)\), an \(m\)-multiindex, for \(s \in \N_0\).
Let \(\Tot(\mbf{h})\) be the sum of all entries of the multiindex~\(\mbf{h}\).
From \(J(X^i Y^{s-i}) = (-1)^i Y^i X^{s-i}\) we get
\(J \Ftensor{\ix} = (-1)^{\Tot(\ix)} \Ftensor{\repx{\ell} - \ix}\)
and
\(J \Fwedge{\jx} = (-1)^{\Tot(\jx)}  
\Fwedge{\repx{\ell+m-1} - \jx}\).
Clearly \(\Stab \ix = \Stab (\repx{\ell}-\ix)\), so
\begin{align*}
\zeta \bigl(  J\Fsym{\ix} \bigr)
    &=
    (-1)^{\Tot(\ix)} \zeta \Fsym{\repx{\ell}-\ix} \\
    &=
    (-1)^{\Tot(\ix)} \sum_{\sigma \in \mathcal{C}} \Fwedge*{
        ( \repx{\ell} - \ix ) \cdot  \sigma + \dx
    }
\intertext{where  \(\mathcal{C}\) is a fixed set of coset representatives for \(\Stab \ix \backslash S_m\). 
Similarly,}
J \bigl( \zeta \Fsym{\ix} )
    &= J \sum_{\sigma \in \mathcal{C}} \Fwedge{\ix \cdot \sigma + \dx} \\
    &= (-1)^{\Tot(\ix) + \Tot(\dx)} 
\sum_{\sigma \in \mathcal{C}} \Fwedge*{\repx{\ell+m-1} - \ix \cdot \sigma - \dx}.
\end{align*}
Observe that \(\repx{\ell +m -1} - \dx = (\repx{\ell} + \dx) \cdot \tau\)
where \(\tau \in S_m\) is the permutation reversing the order of an \(m\)-multiindex, which has sign \((-1)^{\lfloor \frac{m}{2} \rfloor} = (-1)^{m(m-1)/2} = (-1)^{\Tot(\dx)}\).
We can thereby rewrite each summand above as
\begin{align*}
\Fwedge*{\repx{\ell+m-1} - \ix \cdot \sigma - \dx}
    &= \Fwedge*{ (\repx{\ell} - \ix \cdot \sigma \tau + \dx) \cdot \tau } \\
    &= (-1)^{\Tot(\dx)}
        \Fwedge*{ (\repx{\ell} - \ix) \cdot \sigma\tau + \dx}.
\end{align*}
Substituting into our expression for \(J \bigl( \zeta \Fsym{\ix} )\), and using that another suitable set of coset representatives is \(\mathcal{C} \tau\), we deduce that \(J \bigl( \zeta \Fsym{\ix} ) = \zeta \bigl(  J\Fsym{\ix} \bigr)\) as required.

For \(M_\gamma\) we use a  trick, which while technical we believe is of independent interest, to reduce to the Lie algebra action of \(\sl_2(\C)\).

\subsubsection*{Reduction}
\newcommand{\rhoLHS}{\rho_{\scriptscriptstyle\mathrm{sym}}}
\newcommand{\rhoRHS}{\rho_{\scriptscriptstyle\wedge}}

Let \(Z\) be the matrix representing the linear map \(\zeta\) with respect to our given bases.
Let \(\rhoLHS\) and \(\rhoRHS\) be the group homomorphisms \(\SL_2(K) \to \GL_{d}(K)\)
representing the action of \(\SL_2(K)\) on \(\Sym_m \Sym^l E\) and \(\Wedge^m \Sym^{l+m-1} E\) respectively, where $d=\binom{l+m}{m}$.
Then we are required to show that
\begin{equation}
\label{eq:zetaMatrixEquation}
Z \rhoLHS(M_\gamma) = \rhoRHS(M_\gamma) Z
\end{equation}
for all \(\gamma \in K\).
If \(K\) has characteristic zero, then this is a system of equations between polynomials in \(\gamma\) with coefficients in \(\Z\): the entries of \(Z\) are plainly integers, and the entries of \(\rhoLHS(M_\gamma)\) and \(\rhoRHS(M_\gamma)\) are, for our choices of canonical bases, polynomials in \(\gamma\) with integer coefficients.
Moreover, reducing the entries of these matrices modulo \(p\) yields the entries for the corresponding matrices when \(K\) has characteristic \(p\).
Thus showing \eqref{eq:zetaMatrixEquation} for any particular \(\gamma\) which is transcendental over~\(\Z\) establishes \eqref{eq:zetaMatrixEquation} for all elements \(\gamma\) of all fields.
We choose to consider \(K = \C\), and prove \eqref{eq:zetaMatrixEquation} for all \(\gamma \in \C\).

We next reduce to the Lie algebra \(\sl_2(\C)\).
Since \(\SL_2(\C)\) is a connected and simply-connected Lie group, by a 
basic result from Lie theory (see for instance after Definition 8.11 in \cite{FultonHarrisReps}), 
we may regard \(\Sym_m \Sym^\ell E\) and \(\Sym^m \bigwedge^{\ell+m-1} \! E\) as modules for the Lie algebra \(\sl_2(\C)\),
and establish the required property for \(\sl_2(\C)\)-modules.
The one-parameter subgroup containing all~\(M_\gamma\) for \(\gamma \in \C\) has infinitesimal generator
\[
f = \left( \begin{matrix} 0 & 0 \\ 1 & 0 \end{matrix} \right).
\]
Therefore to show \eqref{eq:zetaMatrixEquation} and hence prove \Cref{thm:Wronskian}, it suffices to show
\begin{equation}
\label{eq:zetaIdentitiesLieAlgebra}
\zeta \bigl( f \Fsym{\ix} \bigr) = f \bigl( \zeta \Fsym{\ix} \bigr)
\end{equation}
for all \(\ix \in \Mdecreasing\).

\subsubsection*{Proof of~\eqref{eq:zetaIdentitiesLieAlgebra} using \texorpdfstring{\(\sl_2(\C)\)}{sl2(C)} action}
Recall that if \(V\) is a \(\sl_2(\C)\)-module then the action of \(x \in \sl_2(\C)\) on \(V^{\otimes r}\) is defined by 
\[ x (u_1 \crampedtensor \cdots \crampedtensor u_r) = (x u_1) \crampedtensor \cdots \crampedtensor u_r \ +\ \cdots \ +\ u_1 \crampedtensor \cdots \crampedtensor
(xu_r) \]
with similar rules for the action on \(\Sym^r V\) and~\(\bigwedge^ r V\).
Since \(fX = Y\) and \(fY = 0\),  we have
\(f X^i Y^{l-i} = iX^{i-1} Y^{l-i+1}\) for \(i \in \set{0,\ldots, \ell}\) (permitting \(X^{-1}Y^{\ell+1}\) to appear with zero coefficient).
For \(\svar \in \{1,\ldots, m\}\) let \(\unitx{\svar} = (0,\ldots,1,\ldots,0)\) where \(1\) appears in position~\(\svar\).
Again using the multilinear action of \(\sl_2(\C)\) we get
\( f \Ftensor{\ix} = \sum_{\svar=1}^m \ix_\svar \Ftensor{\ix-\unitx{\svar}}\)
and \( f \Fwedge{\jx} = \sum_{\svar=1}^m \jx_\svar \Fwedge{\jx-\unitx{\svar}}\).

Using that \(\Fsym{\ix} = |\Stab \ix|^{-1} \sum_{\sigma \in S_m} \Ftensor{\ix \cdot \sigma}\),
and that \(\zeta\) extends as a \(K\)-linear map to the domain \((\Sym^\ell E)^{\otimes m}\) by \(\zeta \Ftensor{\ix} = \Fwedge{\ix + \dx}\) (as noted in \Cref{remark:extended_domain_of_zeta}),
we get
\begin{align*}
\zeta \bigl( f \Fsym{\ix} \bigr)
    &=  |\Stab \ix\,|^{-1} \zeta
        \sum_{\sigma\in S_m} \sum_{\svar=1}^m 
        (\ix \cdot \sigma)_\svar
        \Ftensor{\ix \cdot \sigma - \unitx{\svar}} \\
    &= |\Stab \ix\,|^{-1}
        \sum_{\sigma \in S_m} \sum_{\svar=1}^m
        (\ix \cdot \sigma)_\svar
        \Fwedge{ \ix \cdot \sigma - \unitx{\svar} + \dx }
\intertext{and}
f \bigl( \zeta  \Fsym{\ix} \bigr)
    &=  |\Stab \ix\,|^{-1}
        f \sum_{\sigma\in S_m}
        \Fwedge{\ix \cdot \sigma + \dx} \\
    &= |\Stab \ix\,|^{-1}
    \sum_{\sigma\in S_m} \sum_{\svar=1}^m
    (\ix \cdot \sigma + \dx)_\svar
    \Fwedge{\ix \cdot \sigma + \dx - \unitx{\svar}}.
\end{align*}
Since \((\ix \cdot \sigma + \dx)_\svar = (\ix \cdot \sigma)_\svar + m - \svar\), it follows that
\(\zeta\) commutes with the action of \(f\) if and 
only if \(\sum_{\svar=1}^m  \sum_{\sigma \in S_m} \Fwedge{\ix \cdot \sigma - \unitx{\svar} + \dx}(m - \svar) = 0\).
This holds, taking each \(\svar\) separately, by the lemma below.

\begin{lemma}
Let \(\ix \in \Mdecreasing\) and let \(\svar \in \set{1,\ldots, m-1}\). 
Over any field we have \(\sum_{\sigma \in S_m} \Fwedge{\ix \cdot \sigma + \dx - \unitx{\svar}}= 0\).
\end{lemma}

\begin{proof}
The summands for \(\sigma\) and \(\sigma (\svar\ \svar+1)\) are
\begin{align*}
&\Fwedge{\ 
    \;\ldots,\;\,
    \flexbox[l]{\ix_{\svar\sigma^{-1}} + m- \svar - 1,}{\ix_{(\svar+1)\sigma^{-1}} + m- \svar -1,}\,
    \ix_{(\svar+1)\sigma^{-1}} + m- (\svar+1),\,
    \;\ldots\;
\ }, \\
&\Fwedge{\
    \;\ldots,\;\,
    \ix_{(\svar+1)\sigma^{-1}} + m- \svar -1,\,
    \flexbox[l]{\ix_{\svar\sigma^{-1}} + m- (\svar+1),}{\ix_{(\svar+1)\sigma^{-1}} + m- (\svar+1)}\,
    \;\ldots\;
\ }
\end{align*}
respectively.
The two multiindices appearing above differ by the place permutation \((\svar\ \svar+1)\). Hence
\(\Fwedge{\ix \cdot \sigma - \unitx{\svar} + \dx} = -\Fwedge{\ix \cdot \sigma (\svar\ \svar+1) - \unitx{\svar} + \dx}\).
and so the summands  cancel in pairs.
\end{proof}

We have now completed the proof of \Cref{thm:Wronskian}.

\section{Hermite reciprocity (proof of \Cref{cor:Hermite_reciprocity})}
\label{sec:Hermite_reciprocity}

We deduce \Cref{cor:Hermite_reciprocity}, the modular version of Hermite reciprocity restated below,
from the complementary partition isomorphism and the Wronskian isomorphism.
In fact we need only the special case  \Cref{cor:exteriorSymmetric} of the former isomorphism;
this corollary was proved at the end of \S\ref{subsec:first_step_exterior_powers}.

\Hermitereciprocity*

\begin{proof}
Since the representations are polynomial of equal degree \(\ell m\),
it suffices to establish the isomorphism for representations of \(\SL_2(K)\)
by the argument at the end of the introduction.
We have
\begin{align*}
\Sym_m \Sym^l E
    &\iso \Wedge^m \Sym^{l+m-1} E
        && \text{by \Cref{thm:Wronskian}} \\
    &\iso \Wedge^l \Sym_{l+m-1} E 
        && \text{by \Cref{cor:exteriorSymmetric}} \\
    &\iso {(\Wedge^l \Sym^{l+m-1} E )}^\condual && \text{by \Cref{prop:SchurContravariantDual}} \\
    &\iso {\mleft( \Sym_l \Sym^m E \mright)}^\condual
        && \text{by \Cref{thm:Wronskian}} \\
    &\iso \Sym^l \Sym_m E && \text{by \Cref{prop:SchurContravariantDual}},
\end{align*}
as required.
\end{proof}

We illustrate how to explicitly compose the maps above with an example.
(In practice it is convenient to address duality in a different order than in the proof of \Cref{cor:Hermite_reciprocity}.)

\begin{example}
\label{eg:explicit_Hermite}
Suppose \(l = m = 2\), and write \(E = \langle X, Y \rangle_K\) as in \S\ref{sec:Wronskian}.
Given distinct basis vectors \(x, y \in V\), write \((x \crampedtensor y)_\sym = x \crampedtensor y + y \crampedtensor x \in \Sym_2 V\).
In this example we identify the image in \(\Sym^l \Sym_m E\) of the basis element \((X^2 \crampedtensor Y^2)_\sym = X^2 \crampedtensor Y^2 + Y^2 \crampedtensor X^2 \in \Sym_m \Sym^l E\).

We first apply the Wronskian isomorphism \(\zeta\) (\Cref{def:zeta}), giving
\begin{align*}
\Sym_m \Sym^l E
    &\to \Wedge^m \Sym^{l+m-1} E \\
(X^2 \crampedtensor Y^2)_\sym
    &\mapsto
    X^3 \crampedwedge Y^3 - X^2 Y \crampedwedge XY^2.    
\end{align*}
Next we apply the complementary partition isomorphism \(\psi\) (defined by \eqref{eq:psi_definition} in \S\ref{subsec:first_step_exterior_powers}): we replace each summand with the wedge product of the duals of the complementary basis elements (and also pick up a sign, which in our example is~\(+\)).
Composing with the isomorphism \((\Wedge^r V)^\dual \iso \Wedge^r V^\dual\) of \Cref{lemma:exteriorPowerDuality} we obtain
\begin{align*}
\Wedge^m \Sym^{l+m-1} E
    &\to (\Wedge^l \Sym^{l+m-1} E)^\dual  \\
X^3 \wedge Y^3 - X^2 Y \wedge XY^2
    &\mapsto
    \bigl( X^2 Y \wedge XY^2 \bigr)^\dual - \bigl( X^3 \wedge Y^3 \bigr)^\dual.  
\end{align*}
Now we apply the dual \(\zeta^\star\) of the Wronskian isomorphism.
To find the image \(\zeta^\dual(x^\dual)\), we seek those basis elements \(y\) such that \(\zeta(y)\) has \(x\) as a summand.
For \(x = X^2 Y \wedge XY^2\), there are two such basis elements: \(XY \tensor XY\) and the symmetrisation of \(X^2 \tensor Y^2\) (the latter appearing with sign \(-1\)); for \(x = X^3 \wedge Y^3\), the symmetrisation of \(X^2 \tensor Y^2\) is the only such basis element.
Thus
\begin{align*}
(\Wedge^l \Sym^{l+m-1} E)^\dual
    &\to  (\Sym_l \Sym^{m} E)^\dual \\
\bigl( X^2 Y \wedge XY^2 \bigr)^\dual
        - \bigl( X^3 \wedge Y^3 \bigr)^\dual 
    &\mapsto
    \bigl( XY \tensor XY \bigr)^\dual
        - 2\bigl( X^2 \crampedtensor Y^2\bigr)_\sym^\dual.
\end{align*}
The isomorphism \((\Sym^r V)^\dual \iso \Sym_r V^\dual\) is, analogously to the proof of \Cref{lemma:exteriorPowerDuality}, given by the explicit map interchanging symmetrisations and products, yielding
\begin{align*}
(\Sym_l \Sym^{m} E)^\dual
    &\to  \Sym^l \Sym_{m} E^\dual \\
\begin{aligned}
    &\bigl( XY \tensor XY \bigr)^\dual \\
    &\qquad - 2\bigl( X^2 \crampedtensor Y^2\bigr)_\sym^\dual 
\end{aligned}
    \ &\mapsto \
\begin{aligned}
    &(X^\dual \crampedtensor Y^\dual)_\sym \cdot (X^\dual \crampedtensor Y^\dual)_\sym \\
     & \qquad   - 2 (X^\dual\tensor X^\dual) \cdot (Y^\dual \tensor Y^\dual).
\end{aligned}
\end{align*}
Finally we use \Cref{lemma:SL2duality}: there is an isomorphism \(E^\dual \iso E^\condual \iso E\) given by the basis change matrix \(J = \twobytwosmallmatrix{0}{1}{-1}{0}\), which in our case replaces \(X^\dual\) with \(-Y\) and \(Y^\dual\) with \(X\).
We have
\begin{align*}
\Sym^l \Sym_{m} E^\dual
    &\to  \Sym^l \Sym_{m} E \\
\begin{aligned}
    &(X^\dual \crampedtensor Y^\dual)_\sym \cdot (X^\dual \crampedtensor Y^\dual)_\sym \\
     & \qquad   - 2 (X^\dual\tensor X^\dual) \cdot (Y^\dual \tensor Y^\dual)
\end{aligned}
    \ &\mapsto \
\begin{aligned}
    &(X \crampedtensor Y)_\sym \cdot (X \crampedtensor Y)_\sym \\
     & \qquad   - 2 (X \tensor X) \cdot (Y \tensor Y).
\end{aligned}
\end{align*}

Thus our overall map sends
\begin{align*}
\Sym_m \Sym^l E
    &\to \Sym^l \Sym_{m} E \\
(X^2 \crampedtensor Y^2)_\sym
    \ &\mapsto \
    (X \crampedtensor Y)_\sym \cdot (X \crampedtensor Y)_\sym
     - 2 (X \tensor X) \cdot (Y \tensor Y).
\end{align*}
Notice in particular that we have not merely interchanged symmetrisations and products.
Thus this map is of interest even in characteristic \(0\), where it corresponds to a non-trivial automorphism of \(\Sym^2 \Sym^2 E\).
\end{example}

As an application, we recall that stated in the language of representations of $\GL(V)$ where $V$ is a $d$-dimensional complex
vector space,
Foulkes' Conjecture asserts that if $\ell < m$ then 
$\Sym^\ell \Sym^m \! V$ is isomorphic to a subrepresentation of $\Sym^m \Sym^\ell \! E$.
For arbitrary $d$ the conjecture has been proved only when $\ell \leq 5$: see \cite{CheungIkenmeyerMkrtchyan} for this result and a survey of earlier work.
When $d=2$, Foulkes' Conjecture holds by Hermite Reciprocity.
In~\cite{GiannelliDecomp}, Giannelli showed
that the modular analogue of Foulkes' Conjecture for symmetric groups is false in general.
It is therefore notable that the modular version of Hermite reciprocity in \Cref{cor:Hermite_reciprocity} gives a family of
special cases of Foulkes' Conjecture for which there is a modular analogue.

\section{Defect sets (Proof of \Cref{thm:hookObstructions})}
\label{sec:weights}

Throughout this section, we assume \(K\) is a field of characteristic \(p\) (though the definitions of weight spaces and defect sets make sense in characteristic zero also).
We use the notation from \S\ref{sec:Wronskian}
in which \(E = \langle X, Y \rangle_K\) is the natural representation of \(\SL_2(K)\).

\subsection{Weight spaces}

Suppose for this setup that \(K\) is infinite.
Let \(T\) be the torus of diagonal matrices in \(\SL_2(K)\).
Let \(V\) be a representation of a subgroup of \(\SL_2(K)\) containing \(T\).
Recall that for \(r \in \Z\), the \emph{\(r\)-weight space} of \(V\) is 
\begin{equation}
\label{eq:weightSpace}
V_r
    = \setbuild*{
        v \in V
    }{
        \left(
            \begin{matrix} \alpha & 0 \\ 0 &\alpha^{-1} \end{matrix}
        \right) v = \alpha^r v
        \text{ for all \(\alpha \in K^\times\)}
    }.
\end{equation}
An integer \(r\) such that \(V_r \neq 0\) is called a \emph{weight} of \(V\); an element of an \(r\)-weight space is called a \emph{weight vector} with weight \(r\).

We say that \(T\) \emph{acts diagonalisably} on \(V\) if \(V\) has a basis of weight vectors, or equivalently if \(V = \bigoplus_{r \in \Z} V_r\).
If \(V\) is a \(\KSLK\)-module on which \(T\) acts diagonalisably and \(m \in \Z\) is maximal such that \(V_m \not= 0\),
then we say that \(V_m\) is the \emph{highest weight space} of \(V\), and that a non-zero \(v \in V_m\) is a \emph{highest weight vector}.
We say \(v \in V_m\) is a \emph{unique highest weight vector} if $V_m$ is one-dimensional.

Let \(\Borel\) be the Borel subgroup of \(\SL_2(K)\) consisting of lower triangular matrices.
As in \S\ref{sec:Wronskian}, for $\gamma \in K$ we let
\[
M_\gamma = \begin{pmatrix} 1 & 0 \\ \gamma & 1 \end{pmatrix} \in \Borel.
\]

We introduce the following invariant, which we will use to distinguish non-isomorphic representations and hence obtain the results of this section.

\begin{definition}\label{defn:defectSet}
Let \(V\) be a \(\KSLK\)-module on which \(T\) acts diagonalisably with unique highest weight vector \(v\) of weight \(m\).
Let \(\Borel v\) denote the \(K\Borel\)-submodule of \(V\) generated by \(v\).
We define the \emph{defect set} of \(V\), denoted \(\defect(V)\), by
\[
    \defect(V) = \setbuild{d \in \N_0 }{ (Bv)_{m-2d} \not=0}.
\]
\end{definition}

\begin{example}\label{ex:defectSet}
Let \(\alpha \geq 1\).
The module \(\Sym^{p^\alpha}\! E\) has weight vector basis \(\set{X^{p^\alpha}, \ldots, X^{{p^\alpha}-i}Y^i, \ldots, Y^{p^\alpha}}\), where $X^{p^\alpha -i}Y^{i}$ has weight \({p^\alpha}-2i\).
Thus the weights are \({p^\alpha} \ldots, {p^\alpha}-2i, \ldots, -{p^\alpha}\), and  \(X^{p^\alpha}\)
is a unique highest weight vector.
Observe that \(M_\gamma X^{p^\alpha} = (X + \gamma Y)^{p^\alpha} = X^{p^\alpha} + \gamma^{p^\alpha} Y^{p^\alpha}\), and hence
\(\Borel X^{p^\alpha}\) is spanned by \(X^{p^\alpha}\) and \( Y^{p^\alpha}\) whose weights are \({p^\alpha}\) and \(-{p^\alpha}\) respectively.
Hence the defect set is \(\defect(\Sym^{p^\alpha}\! E) = \set{0, {p^\alpha}}\). 
\end{example}

We generalise this example to arbitrary upper and lower symmetric powers in \Cref{lemma:defectSetSym}.

\subsubsection*{Finite fields}
To obtain the full version of \Cref{thm:hookObstructions}
we need the extension of 
\Cref{defn:defectSet} to \(\KSLK\)-modules when \(K\) is finite.
Suppose that \hbox{\(|K| = q\)}.
Defining \(V_r\) as in~\eqref{eq:weightSpace} leads to ambiguity:
the weight~\(r\) is now well-defined only up to multiples of \(q-1\),
and we have \(V = \sum_{r \in \Z} V_r\), no longer direct in general.
Therefore, for the purposes of our work, we restrict the definition of weights to integers in the 
range $\{-\frac{q-1}{2} + 1, \ldots, \frac{q-1}{2} \}$ when $q$ is odd, and in the 
range $\{-\frac{q}{2} + 1, \ldots, \frac{q}{2}-1 \}$ when $q$ is a $2$-power.
Correspondingly, in the definition of the defect set, \Cref{defn:defectSet}, we take only those 
\(d\) in $\{0,1,\ldots, \frac{q-1}{2} \}$ if $q$ is odd
or in  $\{0,1,\ldots, \frac{q}{2}-1 \}$ if $q$ is a $2$-power.
Note that with these definitions, \(T\) acts
diagonalisably on any \(\KSLK\)-module (by a well-known generalisation
of Maschke's Theorem, using that \(T\) is isomorphic to the cyclic group~\(K^\times\) of order \(q-1\)).

\begin{example}
\label{ex:defectSetFiniteField}
We revisit \Cref{ex:defectSet}, now supposing \(K\) is a finite field.
For \(K\) sufficiently large (\(\abs{K} \geq p^{\alpha+2}\)  suffices),  
all the weights written down in \Cref{ex:defectSet} are within the required range, and no changes are needed.
However, when $|K| \leq 1 + 2m$, where $m$ is the highest weight defined for an infinite field, the
behaviour can be very different. 

Consider \(\Sym^4\! E\) when \(K = \mathbb{F}_8\).
Weights are restricted to be between \(-3\) and \(3\) (inclusive), and so 
\(X^4\) has weight \(-3\) (rather than \(4\) as in the infinite field case).
A unique highest weight vector is \(Y^4\) with weight~\(3\) (the other weight vectors are \(X^3Y\) with weight \(2\), \(X^2Y^2\) with weight~\(0\), and \(XY^3\) with weight \(-2\)).
The submodule \(BY^4\) is spanned by \(Y^4\) and thus the defect set is \(\defect(\Sym^4 E) = \set{0}\).

Consider instead \(\Sym^5 E\) when \(K = \mathbb{F}_5\).
Weights are restricted to be between \(-1\) and \(2\) (inclusive), and so \(\Sym^5 E\) has weights \(1\) (with weight vectors \(X^5\), \(X^3Y^2\) and \(XY^4\)) and \(-1\) (with weight vectors \(X^4Y\), \(X^2Y^3\) and~\(Y^5\)).
In particular there is not a unique highest weight vector and so the defect set is not defined.
\end{example}

\subsubsection*{Identifying defect sets for images of Schur functors}

We first verify that defect sets are defined for the modules we wish to distinguish using them.
We assume throughout that \(\abs{K} \geq 4\) (as otherwise weights are only permitted to be in the sets \(\set{0}\) or \(\set{0,1}\), which is too restrictive).

The natural representation \(E\) has weight vector basis \(\set{X, Y}\), where $X$ is a unique highest weight vector of weight \(1\) and \(Y\) has weight \(-1\).
It is straightforward to  identify weight vector bases for the images of \(E\) under iterated Schur functors and their duals, and observe that there is a unique highest weight vector and hence that the defect set is defined.

\begin{proposition}
\label{prop:weight_vector_bases_for_Schur_functors_and_duals}
Let \(V\) be a \(\KSLK\)-module with weight vector basis \(\set{v_1, \ldots, v_l}\), where \(v_i\) has weight \(r_i\), for some integers \(r_1 \leq \cdots \leq r_{l-1} < r_l\).
\begin{thmlist}
    \item\label{item:weight_vectors_for_Schur_functors}
The basis of \(\nabla^\lambda V\) consisting of semistandard polytabloids is a weight vector basis, 
in which \(\polyt(t)\) has weight \(\sum_{b \in \Y{\lambda}} r_{t(b)}\) (modulo \(\abs{K} - 1\)). 
Let \(\tmax\) be the semistandard tableau obtained by filling each column from the bottom with integers decreasing from \(\ell\), and suppose that \(\abs{K} > 1 + 2\sum_{b \in \Y{\lambda}} r_{\tmax(b)}\).
Then a unique highest weight vector is \(\polyt(\tmax)\).
    \item\label{item:weight_vectors_for_duals}
The basis \(\set{v_1^\dual, \ldots, v_l^\dual}\) for \(V^\condual\) dual to \(\set{v_1, \ldots, v_l}\) is a weight vector basis,
in which \(v_i^\dual\) has weight \(r_i\). 
A unique highest weight vector is \(v_l^\dual\), of weight \(r_l\).
\end{thmlist}
\end{proposition}

\begin{proof}
The claimed weights are clear; that the semistandard polytabloids form a basis is \Cref{prop:ssytBasis}.
Since $r_{\ell-1} < r_\ell$, there is in each case a unique highest weight vector.
\end{proof}

To identify which of the weight spaces intersect the \(K\Borel\)-submodule generated by the highest weight vector, it suffices to consider the action of unipotent lower triangular matrices on the highest weight vector.
This is made precise by the following lemma.

\begin{lemma}
\label{lemma:necessary_and_sufficient_condition_for_W_weights}
Let \(V\) be a \(\KSLK\)-module on which \(T\) acts diagonalisably, and let \(U\) be a \(K\Borel\)-submodule of \(V\) generated by some weight vector \(v \in V\).
Then \(U_r \neq 0\) if and only if there exists some \(\gamma \in K\) such that the component of \(M_\gamma v\) in \(V_r\) is non-zero.
\end{lemma}

\begin{proof}
For the `if' direction,
it suffices to prove that if \(v_1, \ldots, v_n\) are non-zero weight vectors with distinct weights \(r_1, \ldots, r_n\) such that \(v_1 + \cdots + v_n \in U\), then each \(v_i\) lies in \(U\).
We use induction on \(n\).
The case \(n=1\) is clear.
Suppose \(n>1\), and write \(x = v_1 + \cdots + v_n\).
Choose \(\alpha \in K\) such that \(\alpha^{r_1} \neq \alpha^{r_n}\) (when $K$ is finite this is possible since \(\abs{K} > \abs{r_1} + \abs{r_n}\) by our definition of weights), and let \(g = \begin{psmallmatrix}
            \alpha & 0 \\ 0 & \alpha^{-1} \\
\end{psmallmatrix} \in \Borel \leq \SL_2(K)\).
Then
\[
U \ni gx - \alpha^{r_n} x = (\alpha^{r_1} - \alpha^{r_n})v_1 + \ldots + (\alpha^{r_{n-1}} - \alpha^{r_n}) v_{n-1}.
\]
By the inductive hypothesis, \(v_1 \in U\), and hence \(x - v_1 \in U\).
Then by the inductive hypothesis applied to $x-v_1$, we also have \(v_2, \ldots, v_n \in U\).

Conversely, suppose \(U_r \neq 0\). 
Then there exists some \(g \in \Borel\) such that~\(g v\) has non-zero component in \(V_r\).
An element of \(\Borel\) can be written as \(g = M_\gamma \twobytwosmallmatrix{\alpha}{0}{0}{\alpha^{-1}}\) for some \(\alpha, \gamma \in K\), and since \(v\) is a weight vector we have that \(\twobytwosmallmatrix{\alpha}{0}{0}{\alpha^{-1}} v\) is a non-zero scalar multiple of \(v\).
Thus \(M_\gamma v\) has non-zero component in \(V_r\).
\end{proof}

Finally in this subsection we record a lemma which is of great use when ruling out certain elements from being in  defect sets.
Given subsets \(I, J \subseteq \N_0\), let \(I+J = \setbuild{i+j}{i \in I,\, j \in J}\).

\begin{lemma}
\label{lemma:defect_sets_for_images_and_tensors}
Suppose \(V\) and \(W\) are \(\KSLK\)-modules on which \(T\) acts diagonalisably with a unique highest weight vector.
\begin{thmlist}
\item\label{item:defect_sets_for_images}
    If \(\phi \colon V \to W\) is a homomorphism that does not annihilate the highest weight vector of \(V\), then \(\defect(\im \phi)\) is defined and \(\defect(\im \phi) \subseteq \defect(V)\).
    In particular, if \(W\) is a quotient of \(V\), then \(\defect(W) \subseteq \defect(V)\).
\item\label{item:defect_sets_for_tensors}
    Suppose \(\abs{K} - 1\) is strictly greater than twice the sum of the highest weights of \(V\) and \(W\).
    Then the set \(\defect(V \tensor W)\) is defined and \(\defect(V \tensor W) \subseteq \defect(V) + \defect(W)\).
\end{thmlist}
\end{lemma}

\begin{proof}
This follows easily from \Cref{lemma:necessary_and_sufficient_condition_for_W_weights}, using in \ref{item:defect_sets_for_images}
that if $v$ is a highest weight vector in $V$ then $\phi(v) \neq 0$ is a highest weight vector in~$W$;
and in~\ref{item:defect_sets_for_tensors} that if also $w$ is a highest weight vector in $W$ then, by the hypothesis on the field size, $v \otimes w$ is a highest weight vector in $V \otimes W$.
\end{proof}

\subsection{Symmetric powers and carry-free sums}

In this subsection we identify the defect sets for iterated symmetric powers.
This prepares the ground for the proof of \Cref{thm:hookObstructions}, and also yields \Cref{prop:symDuals}, characterising when symmetric powers are isomorphic to their duals, and Propo\-sition 6.12, demonstrating that our \Cref{cor:Hermite_reciprocity} is the unique modular generalisation of Hermite reciprocity.

For \(a \in \set{0,\ldots, \ell}\), let \((X^{\otimes \ell-a} \otimes Y^{\otimes a})_\sym \in \Sym_\ell E\) be the sum of all~\(\binom{\ell}{a}\)
pure tensors \(Z_1 \otimes \cdots \otimes Z_\ell\) where exactly \(\ell-a\) of the factors are \(X\) and the remaining \(a\) are \(Y\).

Binomial and multinomial coefficients will frequently appear when expanding the action of matrices \(M_\gamma\) on symmetric powers.
To determine when these coefficients are non-zero modulo \(p\), we require the notion of carry-free sums.

\begin{definition}
Let \(a_1, \ldots, a_s \in \N_0\), and write \(a_i^{(j)}\) for the base \(p\) digit of~\(a_i\) corresponding to the power of \(p^j\).
We say that the sum \(a_1 + \cdots + a_s\) is \emph{carry-free in base \(p\)} if \(a_1^{(j)} + \dots + a_s^{(j)} \leq p-1\) for all \(j\).
For \(a, l \in \N_0\), we say that \(a\) is a \emph{carry-free summand} of \(l\), denoted \(a \cfreesummand l\), if \(a \leq l\) and the sum \(a + (l-a)\) is carry-free.
\end{definition}

Equivalently, $a_1 + \cdots + a_s$ is carry-free in base $p$
if the sum can be computed in base \(p\) without carrying, by the usual algorithm taught in schools for base~\(10\). 
Lucas's Theorem (see for instance \cite[Lemma~22.4]{James}) states that the binomial coefficient \(\binom{l}{a}\) is non-zero modulo \(p\) if and only if \(a \cfreesummand l\), and more generally that the multinomial coefficient \(\binom{a_1 + \cdots + a_s}{a_1, \ldots, a_s}\) is non-zero modulo \(p\) if and only if the sum \(a_1 + \cdots + a_s\) is carry-free.

\begin{lemma}\label{lemma:defectSetSym}
Let \(\ell \in \N_0\) and suppose \(\abs{K} > 1+ 2\ell\).
Then:
\begin{thmlist}
\item
    \(\defect(\Sym_\ell E) = \set{0,\ldots, \ell}\);
\item
    \(\defect(\Sym^\ell\! E) = \setbuild{d \in \set{0,\ldots, \ell}}{ d \cfreesummand \ell }\).
\end{thmlist}
\end{lemma}

\begin{proof}
A highest weight vector of 
\(\Sym_\ell E\) is \(X^{\tensor l}\) and a highest weight vector of \(\Sym^l\! E\) is \(X^\ell\).
A simple calculation yields
\begin{align*}
M_\gamma(X^{\otimes \ell})
    &= \sum_{d=0}^\ell \gamma^d (X^{\otimes \ell-d}Y^{\otimes d})_\sym, \\
M_\gamma(X^\ell)
    &= \sum_{d=0}^\ell \gamma^d \binom{\ell}{d} X^{\ell -d}Y^d.
\end{align*}
Note that \(X^{\tensor l - d} \tensor Y^{\tensor d}\) and \(X^{\ell -d}Y^d\) have weight \(\ell - 2d\); using \Cref{lemma:necessary_and_sufficient_condition_for_W_weights} and Lucas's Theorem mentioned above,
the defect sets are then clear.
\end{proof}

The part of the following proposition for fields of characteristic zero is well-known and is included for logical completeness.

\begin{proposition}\label{prop:symDuals}
Let \(\ell \in \N_0\) and suppose \(\abs{K} > 1 + 2\ell\). 
Then \(\Sym^\ell\! E \cong \Sym_\ell\bs E\) if and only if
\(\ell < p\)
or
\(\ell = p^\epsilon -1\) for some \(\epsilon \in \N\).
If \(K\) is replaced with a field of characteristic zero then \(\Sym^\ell\! E \cong \Sym_\ell\bs E\) for any \(\ell\).
\end{proposition}

\begin{proof}
The condition that \(\ell < p\) or \(\ell=p^\epsilon -1\) for some \(\epsilon \in \N\) is equivalent to the condition that \(a \cfreesummand \ell\) for all \(a \in \set{0,\ldots, \ell}\): if \(\ell < p\) then we clearly have \(a \cfreesummand \ell\) for all \(a \in \set{0,\ldots, \ell}\);
if \(\ell \geq p\) then \(a \cfreesummand \ell\) for all \(a \in \set{0,\ldots, \ell}\) if and only if all base \(p\) digits of~\(\ell\) are \(p-1\), which is if and only if \(\ell = p^\epsilon - 1\).

By \Cref{lemma:defectSetSym}, if \(\Sym^\ell\! E \cong \Sym_\ell\bs E\) then \(a \cfreesummand \ell\) for all \(a \in \{0,\ldots, \ell\}\), as required.
Conversely,
consider the composition of the canonical maps
\[
\Sym_\ell\bs E \hookrightarrow E^{\otimes \ell} \twoheadrightarrow \Sym^\ell\! E
\]
which sends \((X^{\otimes \ell-a}\otimes Y^{\otimes a})_\sym \in \Sym_\ell \bs E\) to \(\binom{\ell}{a} X^{\ell -a} Y^a\).
Supposing \(a \cfreesummand \ell\) for all \(a \in \set{0,\ldots, \ell}\),
or supposing instead the ground field has characteristic zero, we have that \(\binom{\ell}{a} \not=0\), and so this is an isomorphism.
\end{proof}

\begin{lemma}\label{lemma:defectSetSymSym}
Let \(m, l \in \N_0\) and suppose \(\abs{K} > 1 + 2lm\).
Then:
\begin{align*}
\defect(\Sym_m\bs \Sym_l\bs E) 
    &= \set{0, 1, 2, \ldots, lm}; \\
\defect(\Sym_m\bs \Sym^l\! E) 
    &= \setbuild*{\sum_{j=0}^l jm_j}{\begin{aligned}
        &\text{\(m_0, \ldots, m_l \in \N_0\), \(m_0 + \cdots + m_l = m\),} \\[-5pt]
        &\text{\(j \cfreesummand l\) for all \(j\) such that \(m_j \neq 0\)}
    \end{aligned}};  \\
\defect(\Sym^m\bs \Sym_l\bs E) 
    &= \setbuild*{\sum_{j=0}^l jm_j}{\begin{aligned}
        &\text{\(m_0, \ldots, m_l \in \N_0\), \(m_0 + \cdots + m_l = m\),} \\[-5pt]
        &\text{\(m_0 + \cdots + m_l\) is carry-free}
    \end{aligned}};  \\
\defect(\Sym^m\bs \Sym^l\! E) 
    &= \setbuild*{\sum_{j=0}^l jm_j}{\begin{aligned}
        &\text{\(m_0, \ldots, m_l \in \N_0\), \(m_0 + \cdots + m_l = m\),} \\[-5pt]
        &\text{\(m_0 + \cdots + m_l\) is carry-free,} \\[-5pt]
        &\text{\(j \cfreesummand l\) for all \(j\) such that \(m_j \neq 0\)}
    \end{aligned}}. 
\end{align*}
\end{lemma}

\begin{proof}
We compute \(\defect(\Sym^m\bs \Sym^l\bs E)\).
The highest weight vector is \((X^l)^m\) of weight \(lm\), so it suffices to consider the expansion
\begin{align*}
\bigl( (X+\gamma Y)^l \bigr)^m
    &= \left( \sum_{j=0}^l \binom{l}{j} \gamma^j X^{l-j}Y^j \right)^m \\
    &= \sum_{\substack{m_0, \ldots, m_l \in \N_0 \\ m_0 + \cdots + m_l = m}}
        \binom{m}{m_0, \ldots, m_l} \prod_{j=0}^l \left( \binom{l}{j} \gamma^j X^{l-j} Y^j \right)^{m_j}.
\end{align*}
The vectors of weight \(lm - 2d\) are precisely the elements \(\prod_{j=0}^l (X^{l-j} Y^j)^{m_j}\) where 
\smash{\(\sum_{j=0}^l jm_j = d\)}, and such an element appears with non-zero coefficient in this expansion if and only if the corresponding binomial and multinomial coefficients are non-zero.
Lucas's Theorem then yields the claimed defect set. The other parts follow similarly, with the binomial and/or multinomial coefficients not appearing in the expansion when the first and/or second symmetric powers are lower respectively.
\end{proof}

\begin{example}\label{ex:defectSetSymSym}
Let \(\alpha, \beta \geq 1\) and suppose \(\abs{K} > 1 + 2p^{\alpha+\betass}\).
We compute the defect set of \(\Sym^{p^\alpha} \!\Sym^{p^\betaxss} \!E\).
Consider non-negative integers \(m_0, \ldots, m_{p^\betaxss}\) summing to \(p^\alpha\) such that this sum is carry-free and that the only non-zero summands are indexed by carry-free summands of \(p^\betass\).
The only carry-free summands of a power of \(p\) are \(0\) and itself, so by the first condition we have \(m_i = p^\alpha\) for some \(i\) and \(m_j = 0\) for all other \(j\), and by the second condition we have \(m_k=0\) unless \(k \in \set{0, p^\betass}\).
Thus \(\defect( \Sym^{p^\alpha} \!\Sym^{p^\betass} \!E) = \set{0, p^{\alpha + \betass}}\).
\end{example}

\begin{proposition}\label{prop:converseHermite}
Let \(\eps > 1\) and suppose \(\abs{K} > 1 + 2p^{\eps+1}\).
The eight modules obtained from \(\Sym^p \Sym^{p^\eps}\! E\) by exchanging the order of the symmetric powers and replacing upper symmetric powers with lower symmetric powers
are pairwise non-isomorphic, with the exceptions of
\(\Sym_p \Sym^{p^\eps}\! E \cong \Sym^{p^\eps} \Sym_p\bs E\) and its dual \( \Sym^p \Sym_{p^\eps}\bs E \cong \Sym_{p^\eps} \Sym^p\! E\),
and the possible exceptions of an isomorphism
$\Sym^p \Sym^{p^\eps} \hskip-0.32em E \cong \Sym^{p^\eps} \hskip-0.1em \Sym^p \hskip-0.18em E$
and its dual
$\Sym_p \Sym_{p^\eps} \! E \cong \Sym_{p^\eps} \Sym_p \! E$.
Thus there are either four or six isomorphism classes of modules.
If $p=2$ the possible exceptions do not occur and there are precisely six isomorphism classes of modules. 
\end{proposition}

\newlength\setone
\setlength\setone{\widthof{$\{0,1,\ldots,p^{\eps+1} \}$}}
\newlength\settwo
\setlength\settwo{\widthof{$
    \setbuild{jp^\eps}{0 \leq j \leq p}
$}}
\newlength\setthree
\setlength\setthree{\widthof{$
    \setbuild{jp}{0 \leq j \leq p^\eps}
$}}
\newlength\setfour
\setlength\setfour{\widthof{$\{0,p^{\eps+1} \}$}}

\newlength\rowonegap
\setlength\rowonegap{(\settwo-\setone)/2}
\newlength\rowthreegap
\setlength\rowthreegap{(\settwo-\setthree)/2}
\newlength\rowfourgap
\setlength\rowfourgap{(\settwo-\setfour)/2}

\begin{proof}
Calculations using \Cref{lemma:defectSetSymSym} similar to those of \Cref{ex:defectSetSymSym} yield
\begin{align*}
\defect(\Sym_p \Sym_{p^\eps} \!E)
    &= \hspace{\rowonegap} \{0,1,\ldots, p^{\eps+1} \} \hspace{\rowonegap}
    = \defect(\Sym_{p^\eps} \Sym_p\! E), \\
\defect(\Sym_p \Sym^{p^\eps}\!E)
    &=
        \setbuild{jp^\eps}{0 \leq j \leq p}
    = \defect(\Sym^{p^\eps} \Sym_p\! E), \\ 
\defect(\Sym^p \Sym_{p^\eps} \!E)
    &= \hspace{\rowthreegap}
        \setbuild{jp}{0 \leq j \leq p^\eps} \hspace{\rowthreegap}
    = \defect(\Sym_{p^\eps} \Sym^p\! E), \\
\defect(\Sym^p \Sym^{p^\eps}\!E)
    &= \hspace{\rowfourgap} \{0,p^{\eps+1} \} \hspace{\rowfourgap}
    = \defect(\Sym^{p^\eps} \Sym^p\! E).
\end{align*}
Distinctness of defect sets rules out isomorphisms between these modules except those stated in the theorem.
Indeed the first pair of stated isomorphisms hold by modular Hermite reciprocity (\Cref{cor:Hermite_reciprocity}) and its dual.
By the discussion following \Cref{prop:SchurContravariantDual}, \(\Sym_p \Sym_{p^\eps}\bs E \cong (\Sym^p \Sym^{p^\eps}\!E)^\dual\) 
and \(\Sym_{p^\eps} \Sym_p E \cong (\Sym^{p^\eps}\bs \Sym^p\!E)^\dual\), so either both or neither of the possible exceptions occur.
Therefore it remains only to prove, when \(p=2\), that \(\Sym^2 \Sym^{2^\eps}\! E \not\cong \Sym^{2^\eps}\! \Sym^2 E\).

Again we use weight spaces, this time identifying a difference in the \(K \Borel\)-submodules generated by the \(0\)-weight space.
The \(0\)-weight space of \(\Sym^{2^\eps} \Sym^2 E\) is spanned by all
\((X^2)^{2^{\eps-1}-a} \hskip1pt\cdot\hskip1pt (XY)^{2a} \hskip1pt\cdot\hskip1pt (Y^2)^{2^{\eps-1}-a}\) for \(0 \leq a \leq 2^{\eps-1}\).
Applying \(M_\gamma\) to such an element  we get 
\[
(X^2 + \gamma^2 Y^2)^{2^{\eps-1}-a} \cdot \bigl( (X+\gamma Y)Y \bigr)^{2a} \cdot (Y^2)^{2^{\eps-1}-a},
\] 
in which  each factor has only even powers of \(X\) and \(Y\).
Thus the \(K \Borel\)-submodule of \(\Sym^{2^\eps} \Sym^2 E\) generated by the \(0\)-weight space has all weights congruent to \(0\) modulo \(4\).
Meanwhile the \(0\)-weight space of \(\Sym^2 \Sym^{2^\eps}\! E\)
contains \((X^{2^\eps-1}Y) \cdot (XY^{2^\eps-1})\); applying \(M_\gamma\) to this we get \((X+\gamma Y)^{2^\eps-1} Y \cdot (X+\gamma Y)Y^{2^\eps-1}\), whose expansion has \(X^{2^\eps-1} Y \cdot \gamma Y^{2^\eps}\) with coefficient \(1\).
Therefore the \(K \Borel\)-submodule of \(\Sym^2 \Sym^{2^\eps} E\) generated by the \(0\)-weight space has \(-2\) as a weight.
\end{proof}

If we work instead over \(\C\),
all eight modules in \Cref{prop:converseHermite} are isomorphic (by classical Hermite reciprocity and \Cref{prop:symDuals}).

\subsection{Defect sets for hook Schur functors}
\label{subsection:defect_sets_for_nabla}

Our overall strategy is to use defect sets to distinguish the eight modules in \Cref{thm:hookObstructions}.
The reader is invited to refer ahead to \S\ref{subsection:hookObstructions_proof} to see how this is accomplished using the properties of defect sets identified in this subsection and the next.
In this subsection
we study
the defect sets of the modules \(\nabla^{(a+1,1^b)} \Sym^l E\) and \(\nabla^{(a+1,1^b)} \Sym_l E\); in the next, we do the same with \(\Delta\) in place of \(\nabla\).

To identify elements of the defect sets, we need to evaluate the action of \(M_\gamma\) on the highest weight vectors.
Working with \(\nabla^{(a+1, 1^b)}\), we can use the simple multilinear expansion rule for the polytabloids exemplified in \Cref{eg:polytabloid_action_expansion}.
We also need the description of the action of \(M_\gamma\) on the canonical bases of \(\Sym^l E\) and \(\Sym_l E\), given by the following lemma.

\begin{samepage}
\begin{lemma}\label{lemma:Maction}
We have
\begin{thmlist}
\item\label{item:Maction_on_lower_sym}
    \(\displaystyle M_\gamma(X^{\otimes i} \otimes Y^{\otimes l-i})_\sym = \sum_{j=0}^{i} \gamma^{i-j} \binom{l-j}{l-i} (X^{\otimes j} \otimes Y^{\otimes l-j})_\sym\),
\item\label{item:Maction_on_upper_sym}
    \(\displaystyle M_\gamma (X^{i} Y^{l-i}) = \sum_{j=0}^{i} \gamma^{i-j} \binom{i}{j} X^{j} Y^{l-j}\).
\end{thmlist}
\end{lemma}
\end{samepage}

\begin{proof}
Part \ref{item:Maction_on_upper_sym} is obvious from expanding \((X+\gamma Y)^{i}Y^{l-i}\).
For part \ref{item:Maction_on_lower_sym}, observe that \(M_\gamma (X^{\otimes i} \otimes Y^{\otimes l-i})_\sym\) is the sum of all \(\binom{\ell}{i}\) tensor products \(Z_1 \otimes \cdots \otimes Z_\ell\) where exactly \(i\) of the factors are \(X+\gamma Y\) and the remaining \(\ell-i\) are \(Y\).
Expanding into pure tensors in \(X\) and \(Y\), there are \(\binom{l}{i} \binom{i}{j}\) summands with \(j\) factors of \(X\) and \(\ell-j\) factors of \(Y\) (each with coefficient \(\gamma^{i-j}\)).
Then since \(\binom{l}{j}\) such summands are required to form \((X^{\otimes j} \otimes Y^{\otimes l-j})_\sym\), the number of times this vector (with coefficient \(\gamma^{i-j}\)) occurs is 
\smash{\(\binom{l}{i} \binom{i}{j} \binom{l}{j}^{-1} = \binom{l-j}{l-i}\)}.
\end{proof}

\begin{lemma}
\label{lemma:defect_set_for_nabla_lowersym}
Let \(a, b, \ell \in \N\) and suppose \(\abs{K} > 1 + 2(a+b+1)\ell - b(b+1)\).
If \(b\not\equiv -1\) mod \(p\), then \(1 \in \defect(\nabla^{(a+1,1^b)} \Sym_\ell E)\).
\end{lemma}

\begin{proof}
Let \(\tmax\) be the $(a+1,1^b)$-tableau labelling the highest weight vector of \(\nabla^{(a+1,1^b)} \Sym_l E\) identified in \Cref{prop:weight_vector_bases_for_Schur_functors_and_duals}; by this proposition,
its weight is $(a+1)\ell + (\ell-1) + \cdots + (\ell-b) = (a+b+1)\ell - b(b+1)/2$, whence the bound on $|K|$.
Let \(s\) be the tableau obtained from \(\tmax\) 
by reducing the entry in the top-left corner by~\(1\).
That is,
\[
\begin{tikzpicture}[x=1.3cm,y=-0.8cm,line width=0.4pt]
    \node at (-0.5, 2.5) {\(\tmax = \)};
    \draw(0, 0)--(3.8, 0); \draw(0, 1)--(3.8, 1);
    \draw(0, 2)--(1, 2);
    \draw(0, 3)--(1, 3);
    \draw(0, 4)--(1, 4);
    \draw(0, 5)--(1, 5);
    \draw(0, 0)--(0, 5); \draw(1,0)--(1, 5);
    \draw(2, 0)--(2, 1);
    \draw(2.8, 0)--(2.8, 1);
    \draw(3.8, 0)--(3.8, 1);
    \node at (0.5, 0.5) {\(
        l-b
    \)};
        \node at (1.5,0.5) {\(
            l
        \)};
        \node at (2.4,0.5) {\(\cdots\)};
        \node at (3.3,0.5) {\(
            l
        \)};
    \node at (0.5,1.5) {\(
        l{-}b{+}1
    \)};
    \node at (0.5,2.4) {\(\vdots\)};
    \node at (0.5,3.5) {\(
        l- 1
    \)};
    \node at (0.5,4.5) {\(
        l
    \)};
\end{tikzpicture}
\mkern-10mu
\begin{tikzpicture}[x=1.3cm,y=-0.8cm,line width=0.4pt]
    \node at (-0.65, 2.5) {\!and \(\, s = \)};
    \draw(0, 0)--(3.8, 0); \draw(0, 1)--(3.8, 1);
    \draw(0, 2)--(1, 2);
    \draw(0, 3)--(1, 3);
    \draw(0, 4)--(1, 4);
    \draw(0, 5)--(1, 5);
    \draw(0, 0)--(0, 5); \draw(1, 0)--(1, 5);
    \draw(2, 0)--(2, 1);
    \draw(2.8, 0)--(2.8, 1);
    \draw(3.8, 0)--(3.8, 1);
    \node at (0.5,0.5) {\(
        l{-}b{-}1
    \)};
        \node at (1.5,0.5) {\(
            l
        \)};
        \node at (2.4,0.5) {\(\cdots\)};
        \node at (3.3,0.5) {\(
            l
        \)};
    \node at (0.5,1.5) {\(
        l{-}b{+}1
    \)};
    \node at (0.5,2.4) {\(\vdots\)};
    \node at (0.5,3.5) {\(
        l- 1
    \)};
    \node at (0.5,4.5) {\(
        l
    \)};
\end{tikzpicture}
\]
where an entry of \(i\) corresponds to the basis vector \(v_i = (X^{\tensor i} \tensor Y^{\tensor l-i})_\sym\).

We  compute \(M_\gamma \polyt(t_\maxt)\) by acting on the entry in each box of \(t_\maxt\), 
 as in \Cref{eg:polytabloid_action_expansion}, and then using Garnir relations (see \Cref{defn:Garnir})
  to express the result in the basis of semistandard polytabloids. 
Note that the Garnir relations do not change the multiset of entries of a tableau; thus to identify the coefficient of a semistandard polytabloid, it suffices to consider only those tableaux with the same multiset of entries.
By \Cref{lemma:Maction}\ref{item:Maction_on_lower_sym}, \(M_\gamma v_i = \sum_{j=0}^{i} \gamma^{i-j} \binom{\ell-j}{\ell-i} v_j\).
The  action of \(M_\gamma\) on the entries of \(\tmax\) yields
\begin{center}
\begin{tikzpicture}[x=3.5cm,y=-1.2cm,line width=0.4pt]
	\newcommand{\e}{0.1}
    \draw(0, 0)--(3.5+\e, 0); \draw(0, 1)--(3.5+\e, 1);
    \draw(0, 2)--(1+\e, 2);
    \draw(0, 3)--(1+\e, 3);
    \draw(0, 4)--(1+\e, 4);
    \draw(0, 5)--(1+\e, 5);
    \draw(0, 0)--(0, 5); \draw(1+\e, 0)--(1+\e, 5);
    \draw(2+\e, 0)--(2+\e, 1);
    \draw(2.5+\e, 0)--(2.5+\e, 1);
    \draw(3.5+\e, 0)--(3.5+\e, 1);
    \node at (0.5+\e/2,0.5) {\(
        \sum\limits_{j=0}^{\ell-b} \gamma^{\ell-b-j} \binom{\ell-j}{b} v_j
    \)};
        \node at (1.5+\e,0.5) {\(
            \sum\limits_{j=0}^{\ell} \gamma^{\ell -j} v_{j}
        \)};
        \node at (2.25+\e,0.5) {\(\cdots\)};
        \node at (3+\e,0.5) {\(
            \sum\limits_{j=0}^{\ell}  \gamma^{\ell -j} v_{j}
        \)};
    \node at (0.5+\e/2,1.5) {\(
        \sum\limits_{j=0}^{\mathclap{\ell-b+1}} \gamma^{\ell-b+1-j} \binom{\ell-j}{b-1} v_j 
    \)};
    \node at (0.5+\e/2,2.4) {\(\vdots\)};
    \node at (0.5+\e/2,3.5) {\(
        \sum\limits_{j=0}^{\ell-1} \gamma^{\ell-1-j} \binom{\ell-j}{1} v_j
    \)};
    \node at (0.5+\e/2,4.5) {\(
        \sum\limits_{j=0}^{\ell} \gamma^{\ell-j} v_j
    \)};
\end{tikzpicture}
\end{center}
before multilinear expansion.
Consider how we can choose summands to obtain a tableau with the same multiset of entries as \(s\).
Since $v_\ell$ must occur $a+1$ times, we must choose $v_\ell$ from the sums in the $a+1$ boxes in which it appears;
then \(v_{l-1}\) must occur once, so must be chosen in the only remaining sum in which it appears; and so on,
until we choose $v_{\ell-b+1}$ from the box immediately below the top-left box.
Finally we must choose $v_{\ell-b-1}$ from the box in the top-left.
The coefficients arising from this choice are 
$\binom{b+1}{1} \gamma$ from the top-left box and $1$s from every remaining
box. Since this sequence of choices gives the semistandard tableau $s$,
no rewriting using Garnir relations is necessary,
and it follows that
 the coefficient of $\polyt(s)$ in $M_\gamma \polyt(t_\maxt)$ is
$(b+1)\gamma$; this is non-zero by the hypothesis on $b$.
\end{proof}

\begin{lemma}
\label{lemma:defect_set_for_nabla_uppersym}
Let \(\alpha, \beta, \epsilon \in \N\) with \(\alpha \neq \beta\) and \(\alpha, \beta < \epsilon\).
Suppose \(\abs{K} > 1+2(p^\epsilon+p^\betass)(p^\alpha+p^\betass+1)-p^\betass(p^\betass+1)\).
Then
\begin{thmlist}
\item\label{item:nabla_uppersym_defects}
    \(p^{\betass+\epsilon} - p^\epsilon \in \defect( \nabla^{ (p^\alpha+1, 1^{p^\betaxss}) } \Sym^{p^\epsilon + p^\betass} E )\);
\item\label{item:nabla_uppersym_non-defects}
    \(1,\, p^\alpha,\, p^\betass,\, p^{\alpha+\epsilon} - p^\epsilon \not\in \defect( \nabla^{ (p^\alpha+1, 1^{p^\betaxss}) } \Sym^{p^\epsilon+p^\betass} E )\).
\end{thmlist}
\end{lemma}

\begin{proof}
For part \ref{item:nabla_uppersym_defects}, we consider (as in the proof of \Cref{lemma:defect_set_for_nabla_lowersym}) how we can expand \(M_\gamma \polyt({\tmax})\)
to obtain tableaux with certain multisets of entries.
This time we choose the tableau \(s\) obtained from \(\tmax\) by reducing all the entries in the first column by \(p^\epsilon\), except the first and last.
That is,
\[
\begin{tikzpicture}[x=1.51cm,y=-0.8cm,line width=0.4pt]
    \newcommand{\e}{0.2}
    \node at (-0.4,2.5) {\(\tmax \!= \)};
    \draw(0,0)--(3.2+\e,0); \draw(0,1)--(3.2+\e,1);
    \draw(0,2)--(1+\e,2);
    \draw(0,3)--(1+\e,3);
    \draw(0,4)--(1+\e,4);
    \draw(0,5)--(1+\e,5);
    \draw(0,0)--(0,5); \draw(1+\e,0)--(1+\e,5);
    \draw(1.9+\e,0)--(1.9+\e,1);
    \draw(2.3+\e,0)--(2.3+\e,1);
    \draw(3.2+\e,0)--(3.2+\e,1);
    \node at (0.5+\e/2,0.5) {\(
        p^\epsilon
    \)};
        \node at (1.45+\e,0.5) {\(
            p^\epsilon + p^\betass
        \)};
        \node at (2.1+\e,0.5) {\(\cdots\)};
        \node at (2.75+\e,0.5) {\(
            p^\epsilon + p^\betass
        \)};
    \node at (0.5+\e/2,1.5) {\(
        p^\epsilon + 1
    \)};
    \node at (0.5+\e/2,2.4) {\(\vdots\)};
    \node at (0.5+\e/2,3.5) {\(
        p^\epsilon {+} p^\betass {-} 1
    \)};
    \node at (0.5+\e/2,4.5) {\(
        p^\epsilon + p^\betass
    \)};
\end{tikzpicture}
\mkern-10mu
\begin{tikzpicture}[x=1.51cm,y=-0.8cm,line width=0.4pt]
    \newcommand{\e}{0.05}
    \node at (-0.48,2.5) {and \(s \!= \)};
    \draw(0,0)--(3.2+\e,0); \draw(0,1)--(3.2+\e,1);
    \draw(0,2)--(1+\e,2);
    \draw(0,3)--(1+\e,3);
    \draw(0,4)--(1+\e,4);
    \draw(0,5)--(1+\e,5);
    \draw(0,0)--(0,5); \draw(1+\e,0)--(1+\e,5);
    \draw(1.9+\e,0)--(1.9+\e,1);
    \draw(2.3+\e,0)--(2.3+\e,1);
    \draw(3.2+\e,0)--(3.2+\e,1);
    \node at (0.5+\e/2,0.5) {\(
        p^\epsilon
    \)};
        \node at (1.45+\e,0.5) {\(
            p^\epsilon + p^\betass
        \)};
        \node at (2.1+\e,0.5) {\(\cdots\)};
        \node at (2.75+\e,0.5) {\(
            p^\epsilon + p^\betass
        \)};
    \node at (0.5+\e/2,1.5) {\(
        1
    \)};
    \node at (0.5+\e/2,2.4) {\(\vdots\)};
    \node at (0.5+\e/2,3.5) {\(
        p^\betass - 1
    \)};
    \node at (0.5+\e/2,4.5) {\(
        p^\epsilon + p^\betass
    \)};
\end{tikzpicture}
\]
where an entry of \(i\) corresponds to the basis vector \(w_i = X^{i} Y^{p^\epsilon + p^\betass -i} \in \Sym^\ell E\).
By \Cref{lemma:Maction}\ref{item:nabla_uppersym_non-defects},
\(M_\gamma w_i = \sum_{j=0}^{i} \gamma^{i-j} \binom{i}{j} w_j\).
Acting by \(M_\gamma\) on each entry of \(\tmax\) yields
\begin{center}
\scalebox{1}{\begin{tikzpicture}[x=3.6cm,y=-1.2cm,line width=0.4pt]
    \draw(0,0)--(3.2,0); \draw(0,1)--(3.2,1);
    \draw(0,2)--(1,2);
    \draw(0,3)--(1,3);
    \draw(0,4)--(1,4);
    \draw(0,5)--(1,5);
    \draw(0,0)--(0,5); \draw(1,0)--(1,5);
    \draw(1.9,0)--(1.9,1);
    \draw(2.3,0)--(2.3,1);
    \draw(3.2,0)--(3.2,1);
    \node at (0.5,0.5) {\(
            \sum\limits_{j=0}^{p^\epsilon}\gamma^{\star}\binom{p^\epsilon}{j}  w_j
    \)};
        \node at (1.45,0.5) {\(
            \;\,\sum\limits_{j=0}^{\mathclap{p^\epsilon + p^\betaxss}} \gamma^\star
            \binom{p^\epsilon + p^\betaxss}{j}
            w_j
        \)};
        \node at (2.1,0.5) {\(\cdots\)};
        \node at (2.75,0.5) {\(
            \;\,\sum\limits_{j=0}^{\mathclap{p^\epsilon + p^\betaxss}} \gamma^\star 
            \binom{p^\epsilon+p^\betaxss}{j} w_j
        \)};
    \node at (0.5,1.5) {\(
        \;\sum\limits_{j=0}^{p^\epsilon +1} \gamma^\star
        \binom{p^\epsilon+1}{j} w_j
    \)};
    \node at (0.5,2.4) {\(\vdots\)};
    \node at (0.5,3.5) {\(
        \;\;\;\sum\limits_{j=0}^{\mathclap{p^\epsilon + p^\betaxss-1}}\gamma^\star
        \binom{p^\epsilon + p^\betaxss-1}{j} w_j
    \)};
    \node at (0.5,4.5) {\(
        \;\,\sum\limits_{j=0}^{\mathclap{p^\epsilon + p^\betaxss}} \gamma^\star
        \binom{p^\epsilon + p^\betaxss}{j} w_j
    \)};
\end{tikzpicture}}
\end{center}
before multilinear expansion, where $\gamma^\star$ denotes a power of $\gamma$ omitted for reasons of space (whose precise value is not required).
Consider how we can choose summands to obtain a tableau with the same multiset of entries as \(s\). 
As before, 
since \(w_{p^\epsilon + p^\betaxss}\) must 
occur \(p^\alpha +1 \) many times, we must choose $w_{p^\epsilon + p^\betaxss}$ from
the sums in the \(p^\alpha + 1\) boxes in which it appears, which are those at the bottom of each column.

For the remaining $p^\beta$ boxes (those in the first column except the bottom), note that for \(0 \leq i,j < p^\epsilon\), we have 
$\binom{p^\epsilon+i}{j} = \binom{i}{j}$ which is non-zero if and only if $j \cfreesummand i$, which in particular requires \(j \leq i\).
Thus, since \(\beta < \epsilon\), the only remaining sum in which \(w_{p^\betaxss-1}\) 
appears with non-zero coefficient is that in the penultimate box in the first column, so it must be chosen there;
continuing, we must choose 
\(w_{j}\) from the sum in box \((j,1)\) for all \(2 \leq j \leq p^\betass -1\).
Finally, in the top-left box \(w_{p^\epsilon}\) must then be chosen.
Thus there is a unique way to obtain a tableau with the same multiset of entries as \(s\) with non-zero coefficient.
Therefore, writing \(s'\) for the semistandard tableau obtained from \(s\) by sorting the first column into ascending order, the coefficient of \(\polyt(s')\) in $M_\gamma \polyt(\tmax)$ is
non-zero, as required.

For \ref{item:nabla_uppersym_non-defects}, recall that the module \(\nabla^{(p^\alpha+1,1^{p^\betaxss})}\Sym^{p^\epsilon+p^\betaxss}\! E\) is the image of the partition-labelled exterior power
\(\Wedge^{(p^\betass + 1,1^{p^\alpha})} \!\Sym^{p^\epsilon+p^\betaxss}\! E\)
under the canonical quotient map \(\act{t} \mapsto \polyt(t)\).
We claim that this map factors through
\begin{equation}\label{eq:factor_through}
\Wedge^{p^\betass +1}  \Sym^{p^\epsilon+p^\betaxss}\! E \,\otimes\, \Sym^{p^\alpha} \! \Sym^{p^\epsilon + p^\betaxss}\! E.
\end{equation}
Indeed, if \(t\) and \(t'\) are tableaux differing only by swapping two entries in the top row (excluding the top-left box), then, writing \(j\) and \(k\) for the columns of the swapped boxes, by the Garnir relation \(\Grel_{(t,\{(1,j)\},\{(1,k)\}}\) we have \(\polyt(t) = \polyt(t')\).
Thus \(\nabla^{(p^\alpha+1,1^{p^\betaxss})}\Sym^{p^\epsilon+p^\betaxss}\! E\) is a homomorphic image of the module \eqref{eq:factor_through} above, and using both parts of \Cref{lemma:defect_sets_for_images_and_tensors} we have
\begin{align*}
\defect( \nabla^{(p^\alpha+1,1^{p^\betaxss})}\Sym^{p^\epsilon+p^\betaxss}E )
    &\subseteq
    \defect(\Wedge^{p^\betaxss + 1} \Sym^{p^\epsilon+p^\betaxss} \!E)
        + \defect(\Sym^{p^\alpha} \!\Sym^{p^\epsilon + p^\betaxss} \!E).
\end{align*}
The Wronskian isomorphism (\Cref{thm:Wronskian}) gives that \(\Wedge^{p^\betass+1} \Sym^{p^\epsilon+p^\betaxss} E \iso \Sym_{p^\betass + 1} \Sym^{p^\epsilon} \!E\).
The two defect sets on the right-hand side above can then be identified with \Cref{lemma:defectSetSymSym}, yielding
\begin{align*}
 \defect( &\nabla^{(p^\alpha+1,1^{p^\betaxss})}\Sym^{p^\epsilon+p^\betaxss}E ) \\
    &\qquad\subseteq
        \setbuild{cp^\epsilon}{0 \leq c \leq p^\betass +1}
        \,+\,
        \set{0, p^{\alpha + \beta}, p^{\alpha + \epsilon}, p^{\alpha+\beta} + p^{\alpha+\epsilon}}.
\end{align*}
It is clear that \(1\), \(p^\alpha\), \(p^\betass\) and \(p^{\alpha + \epsilon} - p^\epsilon\) are not in this set.
\end{proof}

\subsection{Defect sets for dual hook Schur functors}
\label{subsection:defect_sets_for_Delta}
In this section we show that the module \(\Delta^{(a+1,1^b)} V\) is isomorphic to a submodule of the partition-labelled exterior power \(\Wedge^{b+1}\! V \tensor V^{\tensor a}\), and moreover this submodule contains the highest weight vector. Thus we can compute $\defect(\Delta^{(a+1,1^b)} V)$ by working in 
$\Wedge^{b+1}\! V \tensor V^{\tensor a}$, which has a canonical basis labelled by $(a+1,1^b)$-column tabloids
(see \S\ref{subsec:Schur}).

\begin{lemma}
\label{lemma:Delta_submodule_of_exterior_power}
Let \(V\) be a \(\KSLK\)-module with a basis \(\set{v_1, \ldots, v_l}\) of weight vectors, in which
 \(v_i\) has weight \(r_i\), for some integers \(r_1 \leq \cdots \leq r_{l-1} < r_l\).
Let \(\lambda\) be any partition, and let \(\tmax\) be the semistandard tableau obtained by filling each column from the bottom with integers decreasing from~\(\ell\). Suppose that \(\abs{K} > 1 + 2\sum_{b \in \Y{\lambda}} r_{\tmax(b)}\).
Then \(\Delta^{\lambda} V\) is isomorphic to a submodule of \smash{\(\Wedge^{\lambda'} V\)} containing a unique highest weight vector 
\(\act{\tmax}\) of~\(\Wedge^{\lambda'} V\).
In particular, \(\defect(\Delta^\lambda V) = \defect(\Wedge^{\lambda'} V)\).
\end{lemma}

\begin{proof}
Applying contravariant duality to the \(KG\)-surjection \(e\) from \Cref{prop:nablaPresentation} (with \(V^\condual\) in place of \(V\)), we get a \(KG\)-injection
\(e^\circ \colon \Delta^{\lambda} V \rightarrow (\Wedge^{\lambda'}\! V^\condual)^\condual\).
By the comments after \Cref{prop:SchurContravariantDual}, the codomain is isomorphic to \(\Wedge^{\lambda'} V\).
It remains to show that the image of this injection contains the highest weight vector.

It is clear that \(\act{\tmax}\) and \(\polyt(\tmax^\dual)^\dual\) are unique highest weight vectors of \(\Wedge^{\lambda'}\! V\) and \(\Delta^\lambda V\) respectively, where \(\tmax^\dual\) indicates the tableau \(\tmax\) but with entries understood to correspond to the dual basis of \(V^\condual\).
Since these highest weight vectors are of equal weight, it suffices to show that the image \(e^\condual( \polyt(\tmax^\dual)^\dual )\) is non-zero in \((\Wedge^{\lambda'} V^\condual)^\condual\).
Indeed, evaluating at \(\act{\tmax^\dual}\), we see \(
    e^\condual( \polyt(\tmax^\dual )^\dual )(\act{\tmax^\dual}) = \polyt(\tmax^\dual)^\dual(\polyt(\tmax^\dual)) = 1
\) and thus \(e^\circ( \polyt(\tmax^\dual)^\dual) \neq 0\).
\end{proof}
\begin{lemma}
\label{lemma:defect_set_for_Delta_lowersym}
Let \(a, b, \ell \in \N\).
Suppose that \(\abs{K} > 1+ 2(a+b+1)\ell - b(b+1)\). 
Then \(1 \in \defect(\Delta^{(a+1,1^b)} \Sym_l E)\).
\end{lemma}

\begin{proof}
By \Cref{lemma:Delta_submodule_of_exterior_power}, it is equivalent to show
that \(1 \in \defect(\Wedge^{b+1} \Sym_l\hskip-0.5pt E  \tensor (\Sym_l E)^{\tensor a})\).
A unique highest weight vector of \(\Wedge^{b+1} \Sym_l E \hskip0.5pt\tensor\hskip0.5pt (\Sym_l\hskip-1pt E)^{\tensor a}\) is the column tabloid for the tableau \(\tmax\) from \Cref{lemma:defect_set_for_nabla_lowersym}; let \(s\) be the column standard tableau obtained from \(\tmax\) by reducing the entry in box \((1,2)\) by~\(1\).
Then
\begin{align*}
\act{\tmax} &=
    \bigl( (X^{\otimes \ell-b} \crampedtensor Y^{\otimes b})_\sym 
    \wedge
        \cdots
    \wedge
        X^{\otimes \ell}
    \bigr)
    \tensor
        \bigl( X^{\otimes \ell} \bigr) ^{\tensor a}, \\
\act{s} &=
    \bigl( (X^{\otimes \ell-b} \crampedtensor Y^{\otimes b})_\sym 
    \wedge
        \cdots
    \wedge
        X^{\otimes \ell}
    \bigr)
    \tensor
        (X^{\otimes \ell-1} \crampedtensor Y)_\sym
    \tensor
        \bigl( X^{\otimes \ell} \bigr) ^{\tensor a-1}.
\end{align*}
The coefficient of \(\act{s}\) in \(M_\gamma \act{\tmax}\) is the coefficient of \((X^{\otimes \ell-1} \tensor Y)_\sym\) in \(M_\gamma X^{\tensor l}\), which is \(\gamma\).
Thus \(\act{s}\) is in the $K \Borel$-submodule generated by the highest weight vector, giving the required defect.
\end{proof}

\begin{lemma}
\label{lemma:defect_set_for_Delta_uppersym}
Let \(\alpha, \beta, \epsilon \in \N\) with \(\alpha \neq \beta\) and \(\alpha, \beta < \epsilon\).
Suppose that \(\abs{K} > 1 + 2(p^\epsilon + p^\betass)(p^\alpha + p^\betass + 1) - p^\betass(p^\betass + 1)\).
Then
\begin{thmlist}
\item\label{item:Delta_uppersym_defects}
    \(p^\betass \in \defect(\Delta^{(p^\alpha+1,1^{p^\betaxss})} \Sym^{p^\epsilon+p^\betaxss} E)\);
\item\label{item:Delta_uppersym_non-defects}
    \(1\), \(p^{\alpha} \not\in \defect(\Delta^{(p^\alpha+1,1^{p^\betaxss})} \Sym^{p^\epsilon+p^\betaxss} E)\).
\end{thmlist}
\end{lemma}

\begin{proof}
As in the proof of \Cref{lemma:defect_set_for_Delta_lowersym}, we use \Cref{lemma:Delta_submodule_of_exterior_power} to work in \(\Wedge^{p^\betaxss+1} \Sym^{p^\epsilon + p^\betaxss} E \tensor (\Sym^{p^\epsilon + p^\betaxss} E)^{\tensor p^\alpha}\) rather than \(\Delta^{(p^\alpha + 1, 1^{p^\betaxss})} \Sym^{p^\epsilon + p^\betaxss} E\).

The highest weight vector of \(\Wedge^{p^\betaxss+1} \Sym^{p^\epsilon + p^\betaxss} E \tensor (\Sym^{p^\epsilon + p^\betaxss} E)^{\tensor p^\alpha}\) is the column tabloid for the tableau \(\tmax\) from \Cref{lemma:defect_set_for_nabla_uppersym}; let \(s\) be the column standard tableau obtained from \(\tmax\) by reducing the entry in box \((1,2)\) by~\(p^\beta\).
Then
\begin{align*}
\act{\tmax} &=
    \bigl( X^{p^\epsilon} Y^{p^\betaxss} 
    \wedge
        \cdots
    \wedge
        X^{ p^\epsilon + p^\betaxss}
    \bigr)
    \tensor
        \bigl( X^{ p^\epsilon + p^\betaxss} \bigr)^{\tensor p^\alpha}, \\
\act{s} &=
    \bigl( X^{p^\epsilon} Y^{ p^\betaxss}
    \wedge
        \cdots
    \wedge
        X^{p^\epsilon + p^\betass}
    \bigr)
    \tensor
        X^{p^\epsilon} Y^{p^\betaxss}
    \tensor
        \bigl( X^{ p^\epsilon + p^\betaxss} \bigr)^{\tensor p^\alpha - 1}.
\end{align*}
The coefficient of \(\act{s}\) in \(M_\gamma \act{\tmax}\) is the coefficient of \(X^{p^\epsilon} Y^{p^\betaxss}\) in \(M_\gamma X^{p^\epsilon + p^\betaxss}\), which is \(\gamma^{p^\beta} \binom{p^\epsilon+p^\betaxss}{p^\betaxss} \neq 0\).
Thus \(\act{s}\) is in the $K \Borel$-submodule generated by the highest weight vector, proving \ref{item:Delta_uppersym_defects}.

For \ref{item:Delta_uppersym_non-defects}, we use
\Cref{lemma:defect_sets_for_images_and_tensors}\ref{item:defect_sets_for_tensors} and the Wronskian isomorphism (\Cref{thm:Wronskian}) to find that
\begin{align*}
\defect( &\Wedge^{p^\betass+1} \Sym^{p^\epsilon + p^\betass} E \tensor (\Sym^{p^\epsilon + p^\betass} E)^{\tensor p^\alpha}) \\
    &\qquad \subseteq \defect(\Sym_{p^\betaxss+1} \Sym^{p^\epsilon} \!E)
        + \defect(\Sym^{p^\betaxss+p^\epsilon} \!E)
        + \cdots
        + \defect(\Sym^{p^\betaxss+p^\epsilon} \!E)
\end{align*}
where there are \(p^\alpha\) summands of \(\defect(\Sym^{p^\betaxss+p^\epsilon} \!E)\).
From \Cref{lemma:defectSetSym,lemma:defectSetSymSym},
\begin{align*}
\defect(\Sym^{p^\betaxss+p^\epsilon} E)
    &= \set{0, p^\betass, p^\epsilon, p^\betass + p^\epsilon}, \\
\defect(\Sym_{p^\betaxss+1} \Sym^{p^\epsilon} E)
    &= \setbuild{cp^\epsilon}{0 \leq c \leq p^\betass +1}.
\end{align*}
Using that \(\alpha < \epsilon\) and \(\alpha \neq \beta\), it is clear that \(1\) and \(p^\alpha\) are not in this set.
\end{proof}

\begin{remark} 
\label{remark:Delta}
\leavevmode
\begin{rmklist}

    \item 
It follows from \cite[5.3(b)]{GreenGLn} and 
\Cref{lemma:Delta_submodule_of_exterior_power} that when \(K\) is infinite,
\(\Delta^\lambda V\) is isomorphic to the submodule of \(\bigwedge^{\lambda'}\! V\) generated by its unique highest weight vector.
This explicit construction of \(\Delta^\lambda V\) is in
some cases more convenient than the presentation by relations given in \cite[Ch.~5]{GreenGLn}.
Furthermore, when~\(K\) is infinite,
by 
 \cite[Proposition 31.2]{HumphreysAlgGps}, 
 the submodule generated by the highest weight vector is the same whether we act by \(\Borel\) or all of \(\SL_2(K)\); thus in this case we have that every weight of \(\Delta^\lambda V\) contributes to the defect set. That is, writing \(m\) for the highest weight, we have \(\defect(\Delta^\lambda V) = \setbuild{d \in \N_0}{(\Delta^\lambda V)_{m-2d} \neq 0}\). This can be used to give  alternative proofs of the two previous lemmas, when~$K$ is infinite.
    \item
Using \Cref{lemma:defect_sets_for_images_and_tensors} and the result 
from \Cref{lemma:Delta_submodule_of_exterior_power} that  \(\defect(\Delta^\lambda V) = \defect(\Wedge^{\lambda'} V)\), we find that \(\defect(\Delta^\lambda V) \subseteq \sum_{j=1}^{\lambda_1} \defect(\Wedge^{\lambda'_j} V)\).
When \(K\) is algebraically closed, it can be shown that this is an equality: indeed, under the conditions of \Cref{lemma:defect_sets_for_images_and_tensors}, there is equality \(\defect(V \tensor W) = \defect(V) + \defect(W)\) because any two matrices $M_\gamma$ and $M_\delta$ are conjugate in \(\SL_2(K)\) by diagonal matrices,
and so, up to a scalar, $M_\gamma v \otimes M_\delta w $ is equal to $M_\kappa (v \otimes w)$ for some suitable $\kappa \in K$.
\end{rmklist}
\end{remark}

\subsection{Proof of \Cref{thm:hookObstructions}}
\label{subsection:hookObstructions_proof}

We are now ready to prove the main theorem of this section.

\hookObstructions*

\begin{proof}
From \Cref{lemma:defect_set_for_nabla_uppersym,lemma:defect_set_for_Delta_uppersym} we have
\begin{align*}
    1,\, p^\alpha,\, p^\beta,\, p^{\alpha + \epsilon} - p^\epsilon &\not\in 
        \defect( \nabla^{(p^{\alpha}+1, 1^{p^\betaxss})} \Sym^{p^\epsilon + p^\betaxss} )
            \ni p^{\betass + \epsilon} - p^\epsilon \\
    1,\, p^\alpha,\, p^\beta,\, p^{\betass + \epsilon} - p^\epsilon &\not\in
        \defect( \nabla^{(p^{\betaxss}+1, 1^{p^\alpha})} \Sym^{p^\epsilon + p^\alpha} )
            \ni p^{\alpha + \epsilon} - p^\epsilon\\
    1,\, p^\alpha &\not\in
        \defect( \Delta^{(p^{\alpha}+1, 1^{p^\betaxss})} \Sym^{p^\epsilon + p^\betaxss} )
            \ni p^\beta \\
    1,\, p^\beta &\not\in
        \defect( \Delta^{(p^{\betaxss}+1, 1^{p^\alpha})} \Sym^{p^\epsilon + p^\alpha} )
            \ni p^\alpha
\end{align*}
and from \Cref{lemma:defect_set_for_nabla_lowersym,lemma:defect_set_for_Delta_lowersym} we have that \(1\) lies in each of the defect sets where \(\Sym^\blank\) is replaced with \(\Sym_\blank\).
Thus it is clear that the four modules whose defect sets are displayed above are pairwise non-isomorphic, and that none is isomorphic to any of the four modules obtained by replacing \(\Sym^\blank\) with \(\Sym_\blank\).
Finally, by applying contravariant duality to
an isomorphism between any two of the latter four modules we obtain an isomorphism between two modules defined using \(\Sym^\blank\).
Therefore no two of the latter four modules are isomorphic.
\end{proof}

\section*{Acknowledgements}
We thank Abdelmalek Abdesselam for the reference to \cite{AproduEtAl}.
We are grateful to an anonymous referee for their thorough reading of our manuscript and their numerous helpful comments.

\end{document}